\documentclass{amsart}

\usepackage{amssymb}
\usepackage{mathrsfs}

\theoremstyle{plain}

\newtheorem{theorem}{Theorem}[section]
\newtheorem{lemma}[theorem]{Lemma}     
\newtheorem{corollary}[theorem]{Corollary}
\newtheorem{proposition}[theorem]{Proposition}

\newtheorem{theorem-var}{Theorem}[]
\newtheorem{lemma-var}[theorem-var]{Lemma}     
\newtheorem{corollary-var}[theorem-var]{Corollary}
\newtheorem{proposition-var}[theorem-var]{Proposition}

\theoremstyle{definition}
\newtheorem{definition}[theorem]{Definition}

\theoremstyle{remark}
\newtheorem{remark}{Remark}
\newtheorem{example}{Example}

\newtheorem*{question}{Question}
\newtheorem*{twoht}{(2-$\mathrm{ht}$)}

\numberwithin{equation}{section}

\DeclareMathOperator{\Pic}{Pic}
\DeclareMathOperator{\End}{End}

\DeclareMathOperator{\Hom}{Hom}

\DeclareMathOperator{\supp}{supp}
\DeclareMathOperator{\alt}{ht}
\DeclareMathOperator{\height}{height}

\DeclareMathOperator{\GL}{GL}
\DeclareMathOperator{\SL}{SL}
\DeclareMathOperator{\Spin}{Spin}
\DeclareMathOperator{\SO}{SO}
\DeclareMathOperator{\Sp}{Sp}
\DeclareMathOperator{\Ad}{Ad}

\title[Projective normality of model varieties]{Projective normality of model varieties\\ 
and related results}

\author{Paolo Bravi, Jacopo Gandini, Andrea Maffei}
\date{\today}

\email{bravi@mat.uniroma1.it}

\curraddr{\textsc{Dipartimento di Matematica\\ Sapienza Universit\`a di Roma\\ 
Piazzale Aldo Moro 5\\ 00185 Roma, Italy}}

\email{jacopo.gandini@sns.it}

\curraddr{\textsc{Scuola Normale Superiore\\
Piazza dei Cavalieri 7\\ 56126 Pisa, Italy}}

\email{maffei@dm.unipi.it}
  
\curraddr{\textsc{Dipartimento di Matematica\\ Universit\`a di
    Pisa\\ Largo Bruno Pontecorvo 5\\ 56127 Pisa, Italy}}


\begin{document}

\begin{abstract}
We prove that the multiplication of sections of globally generated
line bundles on a model wonderful variety $M$ of simply connected type
is always surjective. This follows by a general argument which works
for every wonderful variety and reduces the study of the
surjectivity for every couple of globally generated line bundles to a
finite number of cases. As a consequence, the cone defined by a
complete linear system over $M$ or over a closed $G$-stable subvariety
of $M$ is normal. We apply these results to the study of
the normality of the compactifications of model varieties in simple projective
spaces and of the closures of the spherical nilpotent orbits. 
Then we focus on a particular case   
proving two specific conjectures of Adams, Huang and Vogan 
on an analogue of the model orbit of the group of type $\mathsf E_8$.
\end{abstract}

\maketitle

\section*{Introduction}

Let $G$ be a complex linear algebraic group, semisimple and simply
connected. A $G$-variety $M$ is called \textit{wonderful} of rank $n$
if it satisfies the following conditions:
\begin{itemize}
\item[--] $M$ is smooth and projective;
\item[--] $M$ possesses an open orbit whose complement is a union of
  $n$ smooth prime divisors (the \textit{boundary divisors}) with
  non-empty transversal intersections;
\item[--] any orbit closure in $M$ equals the intersection of the
  prime divisors which contain it.
\end{itemize}

Examples of wonderful varieties are the \textit{flag varieties}, which
are the wonderful varieties of rank zero, and the \textit{complete
  symmetric varieties} introduced by C.~De Concini and C.~Procesi \cite{CP},
which we rather call \textit{adjoint symmetric wonderful
  varieties}. Wonderful varieties were then considered in full generality by D.~Luna, 
who started a program of classification in terms of
combinatorial invariants \cite{Lu2}.

Consider the following.

\begin{question} Let $M$ be a wonderful variety and, for
$\mathcal L, \mathcal L' \in \Pic(M)$, consider the multiplication map
\[
m_{\mathcal L,\mathcal L'} \colon \Gamma(M,\mathcal L) \otimes
\Gamma(M,\mathcal L') \longrightarrow \Gamma(M,\mathcal L \otimes
\mathcal L').
\]
Is $m_{\mathcal L,\mathcal L'}$ surjective for all globally generated $\mathcal L, \mathcal L'$?
\end{question}

In the case of a flag variety, the answer to the previous question is affirmative, 
indeed by the Borel-Weil theorem $\Gamma(M,\mathcal L)$ is a simple
$G$-module for all $\mathcal L \in \Pic(M)$. 
In the case of the wonderful compactification of an adjoint
group (regarded as a symmetric $G\times G$-variety) a still affirmative answer was obtained
by S.S.~Kannan \cite{Ka} with a very explicit analysis. 
In the case of any adjoint
symmetric wonderful variety the same was obtained by R.~Chiriv{\`\i} and the third named author in
\cite{CM} with an inductive argument. 
In general, when all the above-mentioned multiplication maps are surjective,
it follows that the image of $M$ in the
dual projective space of the complete linear system associated to any
globally generated line bundle is projectively normal.

Another remarkable class of wonderful varieties is that of the \textit{model wonderful varieties}, introduced by Luna \cite{Lu3}. 
Given a central subgroup $\Gamma \subset G$, a \textit{model homogeneous space} for the algebraic group $G_\Gamma := G/\Gamma$ is a 
quasi-affine homogeneous space $G/H$ such that $\Gamma \subset H$ and the coordinate ring $\mathbb{C}[G/H]$ is a model of the representations 
of $G_\Gamma$ in the sense of I.M.~Gel'fand (see \cite{BGG}, \cite{GZ1}, \cite{GZ2}), that is, $\mathbb{C}[G/H]$ is isomorphic as a $G$-module 
to the direct sum of all the irreducible representations of $G_\Gamma$. 
Main examples of model homogeneous spaces arise as nilpotent orbits for the adjoint action of $G$,
see also \cite{CM2} where a standard monomial theory for some classical
model homogeneous spaces was developed.
The model homogeneous spaces for $G_\Gamma$ were classified in \cite{Lu3},
where Luna constructed a variety $M^{\mathrm{mod}}_{G_\Gamma}$, which is wonderful for the action of $G$, whose orbits parametrize 
the model homogeneous spaces for $G_\Gamma$. Varieties of the shape $M^{\mathrm{mod}}_{G_\Gamma}$ for some central subgroup $\Gamma \subset G$ 
are called model wonderful varieties. Given a model wonderful variety $M^{\mathrm{mod}}_{G_\Gamma}$, we say that it is of \textit{simply 
connected type} if $M^{\mathrm{mod}}_{G_\Gamma}$ and $ M^\mathrm{mod}_G$ are $G$-equivariantly isomorphic. For $G$ almost simple (i.e., a non-commutative group having no proper closed connected normal subgroups), it 
follows by Luna's description that $M^{\mathrm{mod}}_{G_\Gamma}$ is not of simply connected type if and only if $G_\Gamma = \SO(2r+1)$ (in 
which case $G = \Spin(2r+1)$ and $\Gamma \simeq \mathbb{Z}/2\mathbb{Z}$ is the center of $G$).

The first result of the paper is the following.

\begin{theorem-var}	[see Theorem~\ref{teo: model}] \label{teo: modello-intro}
Let $M$ be a model wonderful variety of simply connected type. The multiplication of global sections
\[
m_{\mathcal L,\mathcal L'} \colon \Gamma(M,\mathcal L) \otimes
\Gamma(M,\mathcal L') \longrightarrow \Gamma(M,\mathcal L \otimes
\mathcal L')
\]
is surjective for all globally generated line bundles $\mathcal
L,\mathcal L'$ on $M$.
\end{theorem-var}

This is false if $M$ is a model wonderful variety not of
simply connected type. Indeed, in case $M =
M^{\mathrm{mod}}_{\SO(2r+1)}$ the multiplication of global
sections $m_{\mathcal L,\mathcal L'}$ is surjective for all globally
generated line bundles $\mathcal L,\mathcal L'$ on $M$ if and only if
$r < 4$ (see Section \ref{subsect: non projnorm}). This is essentially
a consequence of the fact that the tensor product of an almost simple group of
type $\mathsf{B}_r$ does not satisfy the saturation property in the sense of
A.~Klyachko (see \cite{KT} and \cite{Ku}).

Theorem~\ref{teo: modello-intro} follows from a general argument which works for every
wonderful variety $M$ and reduces the surjectivity of the maps
$m_{\mathcal L,\mathcal L'}$ for all globally generated line bundles
$\mathcal L, \mathcal L'$ on $M$ to the surjectivity for a finite number of
couples. 
Indeed, we also prove
similar theorems for other classes of wonderful varieties 
arising in some specific applications.\\

We now explain in more details our reduction in the study of the surjectivity
of the multiplication maps on a wonderful variety, and how we apply Theorem~\ref{teo: modello-intro}
and its analogues in different directions.

\subsection*{A general reduction.}
Fix a Borel subgroup $B\subset G$ and a maximal torus $T \subset
B$. Denote by $\Sigma$ the set of $G$-stable prime divisors of $M$ and
by $\Delta$ the set of $B$-stable prime divisors of $M$ which are not
$G$-stable: since $M$ possesses an open $B$-orbit, $\Delta$ is a finite
set. Then the Picard group $\Pic(M)$ is freely generated by the line
bundles $\mathcal L_D$ with $D \in \Delta$ and the free group $\mathbb{Z} \Sigma$
embeds in $\Pic(M)$, so that we may regard $\mathbb{Z}\Sigma$ as a sublattice
of $\mathbb{Z}\Delta$. Moreover, via the isomorphism $\Pic(M) \simeq \mathbb{Z}\Delta$,
the semigroup of globally generated line bundles is identified with
the semigroup $\mathbb{N}\Delta$. We denote by $\mathcal L_E$ the globally generated
line bundle associated to an element $E \in \mathbb{N}\Delta$. Define the
following partial order relation on $\mathbb{N}\Delta$:
\[	E \leqslant_\Sigma F \qquad \text{ if } \qquad E-F \in \mathbb{N}\Sigma.	\]
This partial order is tightly related to the isotypic
decomposition of the spaces of global sections of the globally
generated line bundles on $M$, which we may always assume linearized
(see Proposition~\ref{prop: decomposizione sezioni}).

In case $M$ is the wonderful compactification of the adjoint group
$G_\mathrm{ad} = G/Z(G)$ regarded as a $G\times G$-variety, then $\Sigma$ is
naturally identified with the basis of the root system of $G$, while
$\Delta$ is naturally identified with the set of the fundamental weights
of $G$. More generally, this is true whenever $M$ is the wonderful compactification
of a non-Hermitian symmetric space, in which case there always exists a root
system $\Phi_\Sigma$ (the \textit{reduced root system}) such that $\Sigma$
is a basis of $\Phi_\Sigma$ and $\Delta$ is the corresponding set of
fundamental weights. This is no longer true in the case of a general
wonderful variety: while there always exists a root system $\Phi_\Sigma
\subset \mathbb{Z}\Delta$ with $\Sigma$ as set of simple roots, the fundamental
weights of $\Phi_\Sigma$ associated to $\Sigma$ may differ from
$\Delta$. Therefore we may think the couple $(\Sigma, \Delta)$ as a
generalization of a root datum. 

Suppose that $E,F \in \mathbb{N} \Delta$ are such that $E <_\Sigma F$ and there is no
$D$ with $E <_\Sigma D <_\Sigma F$: then we say that $F-E \in \mathbb{N}\Sigma$
is a \textit{covering difference}. The set of the covering differences
is finite and in the case of an usual root system it was studied by J.R.~Stembridge
in \cite{St}.

Given $E = \sum_{D \in \Delta} n_D D \in \mathbb{Z} \Delta$, define the
\textit{positive part} $E^+ = \sum_{n_D > 0} n_D D$ and the
\textit{height} $\alt(E) = \sum_{D \in \Delta} n_D$. We prove the
following. 

\begin{lemma-var} [see Lemma~\ref{lem: proiettiva normalita}] \label{lem: riduzione-intro}
Let $M$ be a wonderful variety and let $n$ be such that $\alt(\gamma^+)
\leqslant n$ for every covering difference $\gamma$. If the multiplication
map
\[m_{\mathcal L_E, \mathcal L_F} \colon \Gamma(M,\mathcal L_E) \otimes \Gamma(M,\mathcal L_F) \longrightarrow  \Gamma(M,\mathcal L_{E+F})\]
is surjective for all $E,F \in \mathbb{N} \Delta$ with $\alt(E+F)\leqslant n$, then
it is surjective for all $E,F \in \mathbb{N}\Delta$.
\end{lemma-var}

We use Lemma~\ref{lem: riduzione-intro} to prove Theorem~\ref{teo: modello-intro}.
We first study the covering relation in the case of a model wonderful
variety proving that $\alt(\gamma^+) \leqslant 2$ for all covering differences $\gamma$,
then we study the multiplication maps $m_{\mathcal L_E,\mathcal L_F}$ in the
fundamental cases $E,F \in \Delta$. 
To check the inclusions arising in the exceptional group cases we use the computer. 

The fact that $\alt(\gamma^+) \leqslant 2$
for all covering differences $\gamma$ is an easy exercise in case the
couple $(\Sigma, \Delta)$ corresponds to a root system, and as far as we
know it could be a general fact which holds for all wonderful
varieties.

Proceeding inductively on the partial order $\leqslant_\Sigma$, it is easy to
reduce the surjectivity of the multiplication map
$m_{\mathcal L_E,\mathcal L_F}$ for every $E,F \in \mathbb{N}\Delta$ to the fact that
some special submodules of $\Gamma(M,\mathcal L_{E+F})$ occur in the image of
$m_{\mathcal L_E,\mathcal L_F}$.
This leads to the definition of \textit{low triple} (see
Definition~\ref{dfn: low triples}), which was already introduced in \cite{CM}
to treat the case of an adjoint symmetric
wonderful variety. To prove Lemma~\ref{lem: riduzione-intro} we show that it is possible
to treat inductively (w.r.t.\ the height) the low triples of $M$. 

\subsection*{Spherical orbits in simple projective spaces.}
Our first application of the surjectivity of the multiplication maps 
regards the study of the normality of the closure of a spherical orbit in a simple projective space, i.e., the projective space of a simple $G$-module. 
Recall that a $G$-variety $X$ is called \textit{spherical} if it is normal and possesses an open $B$-orbit, 
wonderful varieties form a very relevant class of spherical varieties and play a prominent role in their classification.

Let $V$ be a simple $G$-module, 
let $G \cdot [v] \subset \mathbb{P}(V)$ be a spherical orbit and consider its closure $X = \overline{G \cdot [v]}$, 
which is not necessarily normal. 
Let $H$ be the stabilizer of $[v]$, 
then the spherical homogeneous space $G/H$ can always be embedded as the open orbit of a wonderful variety $M$, 
called the \textit{wonderful compactification} of $G/H$, 
which dominates any other projective compactification of $G/H$ with a unique closed orbit.

If the multiplication of global sections of globally generated line bundles on $M$ is surjective,
the normality of every orbit closure $X \subset \mathbb{P}(V)$ dominated by $M$ can be reduced to the normality of some ``fundamental'' orbit closures. On the other hand, sometimes such fundamental orbit closures are easily shown to be normal. In particular, we have the following.

\begin{theorem-var} [see Corollary~\ref{cor:normality-in-sps}] \label{teo: normality-in-sps-intro}
Let $V$ be a simple $G$-module, let $X \subset \mathbb{P}(V)$ be the closure of a spherical orbit 
and consider a wonderful variety $M$ which dominates $X$. 
If $M$ is symmetric with reduced root system of type $\mathsf{A}$ or model for a connected semisimple group of type $\mathsf{A} \mathsf{D}$, 
then $X$ is normal.
\end{theorem-var}

\subsection*{Spherical nilpotent orbits.} 
Our second application regards the study of the normality of the closure of a spherical nilpotent orbit $\mathcal{O} \subset \mathfrak{g}$, 
where $\mathfrak{g}$ denotes the Lie algebra of $G$. 
In particular, if we consider the projectivization of $\mathcal{O}$, 
we get again a spherical orbit $\mathcal{U} \subset \mathbb{P}(\mathfrak{g})$, 
and the closure $\overline{\mathcal{O}}$ coincides with the affine cone over $\overline{\mathcal{U}}$.

Following \cite{CDM} and a suggestion of D.~Luna, to study the normality of $\overline{\mathcal{O}}$
we may consider a wonderful variety $M_\mathcal{O}$, namely the wonderful compactification of $\mathcal{U}$. 
It turns out that the wonderful variety $M_\mathcal{O}$ is either symmetric or model, or closely related to one of these. 
If the surjectivity of the multiplication maps of $M_\mathcal{O}$ is known, 
the description of the $G$-module structure of the coordinate ring of $\overline{\mathcal{O}}$ is an easy task.

Therefore, we prove the following (with computer-aided computations in the exceptional group cases).

\begin{theorem-var}	[see Theorem~\ref{teo:orbite}]	\label{teo-orbite-intro}
Let $\mathcal{O} \subset \mathfrak{g}$ be a spherical nilpotent orbit and let $M_\mathcal{O}$
be the associated wonderful variety. Then the multiplication map
\[ m_{\mathcal L,\mathcal L'}\colon\Gamma(M_\mathcal{O},\mathcal L) \otimes \Gamma(M_\mathcal{O},\mathcal L')
\longrightarrow \Gamma(M_\mathcal{O},\mathcal L\otimes \mathcal L')\] 
is surjective for all globally generated line bundles $\mathcal L,\mathcal L'$ on $M_\mathcal{O}$.
\end{theorem-var}

As a corollary (see Corollary~\ref{cor: normality-sno}) 
we obtain the classification of the spherical nilpotent adjoint orbits $\mathcal{O}$ whose closure is normal 
(these results are already known, see for instance \cite[Theorem~5.1]{Cos} or \cite[Table~2]{Pa}).
As a particular case, among these orbits we also find the model orbit of $\mathsf E_8$ 
studied by J.~Adams, J-S.~Huang and D.A.~Vogan~Jr.\ in \cite{AHV}.
 
\subsection*{The real model orbit of type $\mathsf{E}_8$.}
Following \cite{AHV}, we also consider an analogue
of the model nilpotent adjoint orbit $\mathcal{O}$ of $\mathsf{E}_8$.
More explicitly, we consider a $K$-orbit, 
where $K$ (an algebraic group) is the complexification 
of a maximal compact subgroup $K_\mathbb R$ of the split real form of $\mathsf E_8$.
Then $K$ is the fixed point subgroup of an involution of $\mathsf E_8$,
this involution passes to the Lie algebra and $K$ acts on the eigenspace $\mathfrak{p}$ of eigenvalue $-1$.
The analogue of the model orbit that we consider is the intersection
$\mathcal{O}_\mathfrak{p} = \mathcal{O} \cap \mathfrak{p}$, which is a $K$-orbit.

For the closure of $\mathcal{O}_\mathfrak{p}$ we prove its normality and
we describe its coordinate ring (Theorem~\ref{teo:E8R}). 
Furthermore, we describe the space of $K_{\mathbb R}$-finite
vectors of the unitary representation 
of the split real form of $\mathsf E_8$ 
that should be associated to this $K$-orbit via the so-called orbit method (Theorem~\ref{teo:E8chi}). 
Both descriptions were already present in \cite{AHV} as consequences of some conjectures, 
which as far as we know are still open.

In order to prove these theorems, we are led to consider another class
of wonderful varieties, which we call \textit{comodel wonderful varieties}
since they are somewhat dual to the model wonderful varieties
(see Theorem~\ref{teo: comodello-esistenza} for a precise definition).
For the comodel wonderful varieties we also show 
that the multiplication of global sections of globally generated line bundles is surjective
(see Theorem~\ref{teo: comodello-proj-norm}).

It turns out indeed that $\mathcal{O}_\mathfrak{p}$ is the cone over a homogeneous space
whose wonderful compactification is closely related to a comodel wonderful variety.

\subsection*{Multiplication of functions on a spherical homogeneous space}

Suppose that $G/H$ is a spherical homogeneous space, in the last section of the paper we also consider the problem of multiplying $G$-modules of functions in the coordinate ring $\mathbb C[G/H]$. Indeed, $\mathbb C[G/H]$ is a multiplicity-free $G$-module (that is, every isotypical component is irreducible), and given two irreducible components $V_\lambda$ and $V_{\lambda'}$ of highest weight $\lambda$ and $\lambda'$ it is well defined their product $V_\lambda V_{\lambda'} \subset \mathbb C[G/H]$, that is, the $G$-module generated by the products $ff'$ with $f \in V_\lambda$ and $f' \in V_{\lambda'}$.

More generally, given $\xi \in \mathcal{X}(H)$, one may consider the corresponding $G$-module of $H$-semiinvariant function
$$
	\mathbb C[G]^{(H)}_\xi = \{ f \in \mathbb C[G] \, : \, f(g h) = \xi(h) f(g) \quad  \forall h \in H\},
$$
and given $\xi, \xi' \in \mathcal{X}(H)$ there is a natural multiplication map
$$
	\mathbb C[G]^{(H)}_\xi \otimes \mathbb C[G]^{(H)}_{\xi'} \longrightarrow  \mathbb C[G]^{(H)}_{\xi + \xi'}.
$$
As in the case $\xi = 0$, the $G$-module $\mathbb C[G]^{(H)}_\xi$ is multiplicity free, and we denote by $V_{\lambda, \xi}$ its irreducible component of highest weight $\lambda$. Therefore, given $\xi, \xi' \in \mathcal{X}(H)$, it is well defined the product $V_{\lambda,\xi} V_{\lambda', \xi'} \subset \mathbb C[G]^{(H)}_{\xi+\xi'}$.

After discussing some non-degeneracy property of such multiplication and giving some examples, we describe the semigroup generated by the set of differences
$$\{\lambda + \lambda' - \mu  \ \, : \, \ V_{\mu, \xi + \xi'} \subset V_{\lambda, \xi} V_{\lambda', \xi'} \quad \exists \ \xi, \xi' \in \mathcal{X}(H)\}$$
in terms of the spherical roots of a suitable wonderful variety as follows. There is a canonical spherical subgroup of $G$ containing $H$, called the \textit{spherical closure} of $H$ and denoted $\overline H$, such that $G/\overline H$ admits a wonderful compactification $M$. Moreover $\mathcal{X}(H) = \mathcal{X}(\overline H)$, and for all $\xi \in \mathcal{X}(H)$ it holds $\mathbb C[G]^{(H)}_\xi = \mathbb C[G]^{(\overline H)}_\xi$. This allows us to reformulate the multiplication of $H$-semiinvariant functions in terms of sections of line bundles on $M$, and we prove the following fact.

\begin{proposition-var}[see Proposition \ref{prop: semigruppo-differenze}]	\label{prop-intro: differenze}
Let $\mathscr M$ be the semigroup generated by the set of differences
$$\{\lambda + \lambda' - \mu  \ \, : \, \ V_{\mu, \xi + \xi'} \subset V_{\lambda, \xi} V_{\lambda', \xi'} \quad \exists \ \xi, \xi' \in \mathcal{X}(H)\},$$
and denote by $\Sigma$ the set of spherical roots of $M$. Then $\mathscr{M} = \mathbb{N} \Sigma$.
\end{proposition-var}

\subsection*{Structure of the paper.}
The first part of the paper is entirely devoted to the proof
of Theorem~\ref{teo: modello-intro} and Lemma~\ref{lem: riduzione-intro}, 
whereas in the second part we have collected the various applications.

In Section~1 we fix the notation and recall some results about the wonderful
varieties and their line bundles. 
In Section~2 we define the low triples and prove Lemma~\ref{lem: riduzione-intro}. 
In Section~3 and in Section~4 we focus on the case of a model wonderful variety, 
first classifying the covering differences and then classifying the low fundamental triples
and studying the associated inclusions.

In Section~5 we introduce the comodel wonderful varieties and
prove the surjectivity of the multiplication for this class of varieties.  

In Section~6 we explain in general how the surjectivity
of the multiplication map can give information on
the normality of the closure of a spherical orbit in the
projective space of a simple $G$-module
(proving by-the-way Theorem~\ref{teo: normality-in-sps-intro}).
In Section~7 we prove Theorem~\ref{teo-orbite-intro}
and we use our results to deduce the normality and the non-normality
of the spherical nilpotent orbit closures.
In Section~8 we concentrate on the real model orbit of type $\mathsf E_8$.

In Section~9 we give the above-mentioned counterexample to the surjectivity of the multiplication 
in the case of a model wonderful variety of not simply connected type.
This leads us to discuss some general properties of the multiplication map and to prove Proposition \ref{prop-intro: differenze}.\\

\textit{Aknowledgments.} We would like to thank Domingo Luna, who suggested us some years ago to study the normality of spherical nilpotent orbit closures via the projective normality of wonderful varieties.

The paper was partially written during a staying of the second author at Friedrich-Alexander-Universit\"at Erlangen-N\"urnberg between March and October 2012, partially supported by a DAAD fellowship. He is grateful to the Emmy Noether Zentrum and especially to Friedrich Knop for warm hospitality; also, many thanks are due to Friedrich Knop for fruitful discussions.

\section{Generalities}

Let $G$ be a simply connected semisimple complex algebraic group.
Fix a maximal torus
$T$ of $G$ and a Borel subgroup $B \supset T$. For any group $K$ we
denote by $\mathcal{X}(K)$ the multiplicative characters of $K$. We denote
also by $\mathcal{X}(T)^+$ the set of dominant characters w.r.t.\ $B$ and if
$\lambda \in \mathcal{X}(T)^+$ we denote by $V(\lambda)$ an irreducible
representation of highest weight $\lambda$ and by $-\lambda^*$ the lowest
weight of $V(\lambda)$. We denote by $S$ the set of simple roots.

Let $M$ be a wonderful $G$-variety with (unique) closed $G$-orbit $Y$. We assume that the center $Z(G)$ acts trivially on $M$, and we will keep this assumption throughout the paper: all the wonderful varieties we will deal with will be $G_{\mathrm{ad}}$-varieties. 
By \cite{Lu1}, $M$ is \textit{spherical}, i.e.\ it possesses an open
$B$-orbit, say $B\cdot x_0 \subset G\cdot x_0 \subset M$.
We denote by $H$ the stabilizer of $x_0$ in $G$.

\subsection{Colors and spherical roots}

Since $B\cdot x_0$ is affine, $G\cdot x_0 \smallsetminus B\cdot x_0$ is a union of finitely many $B$-stable prime
divisors and we denote by $\Delta$ the set of their closures in $M$:
\[ \Delta = \{ D \subset M \, : \, D \text{ is a $B$-stable prime divisor, }
D \cap G\cdot x_0 \neq \varnothing \}.\]
The elements of $\Delta$ are called the \textit{colors} of $M$.

Denote by $B^-$ the opposite Borel subgroup of $B$ and let $y_0 \in
Y$ be the unique $B^-$-fixed point of $M$. 
The normal space of $Y$ in $M$ at $y_0$, 
$\mathrm T_{y_0} M/\mathrm T_{y_0} Y$, is a multiplicity-free $T$-module. 
Set
\[ \Sigma = \left\{ \textrm{$T$-weights of } 
\mathrm T_{y_0} M/\mathrm T_{y_0} Y \right\}:\] 
the elements of $\Sigma$ are called the \textit{spherical roots} of
$M$ and they naturally correspond to local equations of the
boundary divisors of $M$, which are $G$-stable. If $\sigma\in \Sigma$, we
denote by $M^\sigma$ the associated boundary divisor of $M$ such that
$\mathrm T_{y_0}M/\mathrm T_{y_0}M^\sigma$ is the 1-dimensional
$T$-module of weight $\sigma$.

The rank one wonderful varieties were classified by D.N.~Akhiezer \cite{Ak}, 
and we denote by $\Sigma(G) \subset \mathcal{X}(T)$ the finite set of all possible weights 
occurring as the spherical root of a rank one wonderful variety. 
These are called the \textit{spherical roots} of $G$. 
Every spherical root is either a positive root, or the sum of two orthogonal positive roots,  
we refer to \cite[Table 1]{BL} for a list of the spherical roots of $G$.

\subsection{Picard group and Cartan pairing} 

Recall that every line bundle on $M$ or on $Y$ has a unique
$G$--linearization.

We may identify $\Pic(Y)$ with a sublattice of $\mathcal{X}(T)$ and
$\Pic(G\cdot x_0)$ with $\mathcal{X}(H)$ (see \cite{KKV}): we identify $\mathcal L \in
\Pic(Y)$ with the character of $T$ acting on the fiber of $\mathcal L$ over
$y_0$, and we identify $\mathcal L \in \Pic(G\cdot x_0)$ with the character of
$H$ acting on the fiber over $x_0$.

Consider now the maps $\omega\colon \Pic(M) \longrightarrow \mathcal{X}(T)$ and $\xi\colon \Pic(M)
\longrightarrow \mathcal{X}(H)$ defined by the restriction to the closed and to the
open orbit. We may regard $\Pic(M)$ as a sublattice of $\mathcal{X}(T)
\times \mathcal{X}(H)$ by identifying $\mathcal L \in \Pic(M)$ with the couple
$(\omega(\mathcal L), \xi(\mathcal L))$ (see \cite{Br2}). Moreover, we have the following exact sequence 
\[ 0\longrightarrow \mathbb{Z}\Sigma \longrightarrow \Pic(M) \longrightarrow \mathcal{X}(H) \longrightarrow 0.\] 

Given $E \in \mathbb{Z} \Delta$, denote by $\mathcal L_E:=\mathcal{O}(E)$ the associated line bundle. As a group, $\Pic(M)$ is freely generated by the line bundles $\mathcal L_D$ with $D\in\Delta$ (see \cite[Proposition~2.2]{Br1}). For all $E\in\mathbb{Z}\Delta$, the associated
line bundle $\mathcal L_E$ is globally generated (resp.\ ample)
if and only if $E$ is a non-negative (resp.\ positive) combination of
colors. We set $\omega_E = \omega(\mathcal L_E)$, $\xi_E =
\xi(\mathcal L_E)$.

There exists a natural $\mathbb{Z}$-bilinear pairing (called the
\textit{Cartan pairing} of $M$) \[ c \colon \mathbb{Z}\Delta \times \mathbb{Z}\Sigma \longrightarrow
\mathbb{Z} \] which maps the couple $(D, \sigma)$ to the coefficient of
$[M^\sigma]$ along $[D]$. So, regarding $\mathbb{Z}\Sigma$ as a sublattice of
$\mathbb{Z}\Delta$, for $\sigma \in \mathbb{Z} \Sigma$ we have
\[ \sigma = \sum_{D\in \Delta} c(D,\sigma) D. \]

\subsection{Global sections and multiplication}\label{ssec:sezioniglobali}

We now recall the description of the space of global sections $\Gamma(M,\mathcal L)$ of a line bundle $\mathcal L$. 
Notice first that since $M$ is spherical the decomposition of
$\Gamma(M,\mathcal L)$ into simple $G$-modules is multiplicity free for all
$\mathcal L\in\Pic(M)$. If $E \in \mathbb{N}\Delta$ then $\mathcal L_E$ is generated by global
sections. In particular $\Gamma(M,\mathcal L_E)$ must contain a copy of
$V(\omega_E)$ (which is the space of sections of $\mathcal L_E$ on $Y$),
hence $\omega_E$ is dominant. We denote by $V_E$ the unique
simple $G$-submodule of $\Gamma(M,\mathcal L_E)$ of highest weight
$\omega_E$. Notice also that the image of $x_0$ in
$\mathbb{P}(\Gamma(M,\mathcal L_E)^*)$ is a point fixed by $H$. In particular, since
$BH \subset G$ is open, it follows that the \emph{space of spherical vectors}
\[ V(\omega_E^*)^{(H)}_{\xi_E} = \{ v \in V(\omega_E^*)
\, : \, hv = \xi_E(h)v \quad \forall h \in H \}
\] has dimension one and we denote by $h_E$ a generator of this line.

If $\gamma = \sum a_\sigma \sigma\in \mathbb{N}\Sigma$, we denote by $s^\gamma \in
\Gamma(M,\mathcal L_{M^\gamma})$ a section whose divisor is equal to
$M^\gamma= \sum a_\sigma M^\sigma$. Notice that this section is $G$-invariant.  
Recall, as defined in the introduction, that we say that $F
\leqslant_\Sigma E$ if $E-F \in \mathbb{N}\Sigma$. If $E\in \mathbb{N} \Delta$ and
$E\leqslant_\Sigma F$ the multiplication by $s^{F-E}$ induces a 
$G$-equivariant map from the sections of $\mathcal L_E$ to the sections of
$\mathcal L_F$, in particular we have $s^{F-E}V_E\subset
\Gamma(M,\mathcal L_F)$.

\begin{proposition}[{\cite[Proposition~2.4]{Br1}}]	\label{prop: decomposizione sezioni}
Let $F \in \mathbb{Z}\Delta$. Then 
\[
	\Gamma(M,\mathcal L_F) = \bigoplus_{E \in \mathbb{N} \Delta \, : \, E \leqslant_\Sigma F} s^{F-E} V_E.
\]
\end{proposition}

If $E,F \in \mathbb{N}\Delta$, consider the multiplication of sections
\[
	m_{E,F} \colon \Gamma(M,\mathcal L_E) \otimes \Gamma(M,\mathcal L_F) \longrightarrow \Gamma(M,\mathcal L_{E+F}).
\]
and denote by $V_E V_F$ the image of $V_E \otimes V_F$. A way to translate the description of this map into a problem on
spherical vectors is the following.

\begin{lemma}[{\cite[Lemma~19]{CLM}}] \label{lemma: supporto moltiplicazione}
Let $D,E,F \in \mathbb{N}\Delta$ be such that $D \leqslant_\Sigma E+F$. Then
$s^{E+F-D} V_D \subset V_E V_F$ if and only if the projection of $h_E
\otimes h_F\in V(\omega^*_E) \otimes V(\omega^*_F)$ onto the isotypic
component of highest weight $\omega^*_D$ is non-zero.
\end{lemma}

\subsection{Distinguished sets of colors}\label{ss:quoziente}

Let $M'$ be a wonderful $G$-variety together with a surjective
equivariant morphism $\phi \colon M \longrightarrow M'$ with connected fibers and
denote $\Delta_\phi \subset \Delta$ the set of colors which map dominantly
onto $M'$. Then the semigroup
\[ (\mathbb{N}\Sigma)/\Delta_\phi = \{\gamma \in \mathbb{N} \Sigma \, : \, c(D,\gamma) = 0 \quad
\forall D \in \Delta_\phi \}
\] is free and its basis, which we denote by $\Sigma/\Delta_\phi$,
coincides with the set of spherical roots of $M'$, while the set of
colors of $M'$ is identified with $\Delta \smallsetminus \Delta_\phi$ 
(see \cite[Proposition~3.3.2]{Lu2}).

Conversely, if $\Delta_0 \subset \Delta$, then there exists a wonderful
variety $M'$ (unique up to isomorphism) together with a surjective
equivariant morphism $\phi \colon M \longrightarrow M'$ with connected fibers if and
only if $\Delta_0$ is \textit{distinguished},
that is, there exists an element $D \in \mathbb{N}_{>0}\Delta_0$ such that
$c(D,\sigma) \geqslant 0$ for every $\sigma \in \Sigma$ 
(see \cite[Proposition~3.3.2]{Lu2} and \cite[Theorem~3.1]{Bra}).

If $\Delta_0 \subset \Delta$ is distinguished, then we say that the
associated wonderful variety $M'$ is the \textit{quotient} of $M$ by
$\Delta_0$, denoted by $M/\Delta_0$. 

Recall the following general fact.

\begin{proposition}
Let $X,Y$ be normal varieties and suppose that $\phi\colon X \longrightarrow Y$ is a
surjective proper morphism with connected fibers. If $\mathcal L \in
\Pic(Y)$, then $\Gamma(Y,\mathcal L) = \Gamma(X,\phi^*\mathcal L)$.
\end{proposition}

If $E = \sum_{D\in\Delta} a_D D \in \mathbb{Z}\Delta$ we define $\supp (E)$,
the \emph{support} of $E$, as the set of colors $D$ such that $a_D\neq
0$.

\begin{corollary}	\label{cor: sezioni e quozienti}
Let $E \in \mathbb{N}\Delta$ and suppose that $\Delta_0$ is a distinguished
subset of $\Delta$ such that $\Delta_0 \cap \supp(E) = \varnothing$. Then
$\Gamma(M,\mathcal L_E) = \Gamma(M/\Delta_0,\mathcal L_E)$.
\end{corollary}

\subsection{Parabolic induction: spherical roots and colors}

We describe now a standard way to construct a wonderful variety for the group
$G$ from a wonderful variety for a Levi subgroup of a parabolic subgroup of $G$.

Let $L$ be a proper Levi subgroup of $G$ which contains $T$. 
Let $Q$ be the parabolic subgroup
associated to $L$ and containing $B$, and let $Q^-$ be the opposite
parabolic subgroup. Denote by $R_Q$ and $R_Q^-$ the solvable radicals
and by $U_Q$ and $U_Q^-$ the unipotent radicals of $Q$ and
$Q^-$, respectively. 
Finally let $L^{\mathrm{ss}}$ be the semisimple part of $L$ and denote
by $L_{\mathrm{ad}}=L^{\mathrm{ss}}/Z(L^{\mathrm{ss}})=L/Z(L)$ the adjoint quotient of $L$.

Let $N$ be a wonderful variety for the group $L_{\mathrm{ad}}$. Although wonderful varieties have been defined for semisimple groups,
in this case it is convenient to look at $N$ as an $L$-variety and to
consider $L$-linearized line bundles on $N$. The definitions given
for a wonderful variety can be extended to this situation.

Extend
the action of $L$ on $N$ to $Q^-$, with $U^-_Q$ acting trivially, and
define the parabolic induction of $N$ as
\[ M = G \times_{Q^-} N. \] 
We have a projection $\pi\colon M\longrightarrow G/Q^-$, 
$N$ is the fiber over $Q^-$ and moreover it is the set of points of
$M$ fixed by $U^-_Q$. 
One can easily show the following properties.
\begin{enumerate}
\item $M$ is a wonderful $G$-variety.
\item The intersection of a subset of $M$ with $N$ induces a bijection
  from the set of $G$-stable subsets of $M$ to the set of $L$-stable subsets of
  $N$.  For every spherical root $\sigma$ of $M$ we denote by $N^\sigma$ the
  intersection $M^\sigma \cap N$.
\item The restriction of line bundles from $M$ to $N$ induces an
  isomorphism $\rho$ between the class groups of linearized line bundles
  $\Pic_G(M)$ and $\Pic_L(N)$.
\item We have an injective map $\varepsilon$ from the set $\Delta(N)$ of colors of $N$
  to the set $\Delta(M)$ of colors of $M$ given by
  $\varepsilon(D) = \overline{B\cdot  D}$ for all $D\in\Delta(N)$.
\item The exact sequence $0\longrightarrow \mathcal{X}(L/L^{\mathrm{ss}}) \longrightarrow \Pic_L(N) \longrightarrow
  \Pic_{L^{\mathrm{ss}}}(N) \longrightarrow 0$ has a natural splitting given by
  $\mathcal L_D \longmapsto \mathcal L_{\varepsilon(D)}\bigr|_N$ for every color $D \in \Delta(N)$.
\item If $D \in \Delta(M)\smallsetminus \varepsilon(\Delta(N))$ then the restriction
  of $\mathcal{O}(D)$ to $N$ is a trivial line bundle on $N$.
\item The set of colors $\varepsilon(\Delta(N))$ is distinguished 
  and $M/\varepsilon(\Delta(N)) \simeq G/Q^-$.
\item Since $L^{\mathrm{ss}}$ is simply connected we have $\Pic(N)\simeq
  \Pic_{L^{\mathrm{ss}}}(N)$.
\end{enumerate}

Denote
\[ \Sigma = \left\{ \textrm{$T$-weights of } \mathrm T_{y_0}
N/\mathrm T_{y_0} (L\cdot y_0) \right\}, \] 
where we recall that $y_0$
is the point of $M$ fixed by $B^-$, hence $L\cdot y_0$ is the unique
closed orbit of $N$. The set $\Sigma$ is in correspondence with the
set of $L$-stable divisors and in particular $\sigma$ is the $T$-weight
of $\mathrm T_{y_0} N/\mathrm T_{y_0} N^\sigma$. 
Moreover, we have that $\Sigma$ equals the set of spherical roots of $M$
and every element of $\Sigma$ is a sum of simple roots of $L$. 
  
By the above properties (3) and (4) for every color $D$ in $\Delta(N)$ we have
a canonical choice of a linearization of the associated line bundle 
and we have a natural decomposition
\[ \Pic_L(N)\simeq \mathcal{X}(L/L^{\mathrm{ss}}) \oplus \bigoplus_{D\in \Delta(N)}
\mathbb{Z} [D].
\] 
For all $\sigma\in\Sigma$, $[N^\sigma] = \sum_{D\in \Delta(M)} c(D,\sigma)\rho([D])$. 
If $\mathcal L,\mathcal L'\in \Pic_L(N)$ we define $\mathcal L \geqslant_{\Sigma} \mathcal L'$ if,
using the additive notation, 
$\mathcal L-\mathcal L'=\sum a_\sigma \mathcal{O}(N^\sigma)$ with $a_\sigma\geqslant0$ for all $\sigma$, 
similarly to what we have done in Section~\ref{ssec:sezioniglobali}. 
In this way Proposition~\ref{prop: decomposizione sezioni} holds for $N$ without any
change.

It is easy to see when a $G$-wonderful variety $M$ can be obtained by
parabolic induction. Let $P^-$ be the stabilizer of $y_0$
and let $S^p$ the set of simple roots that are in the set of roots of
$P^-$. Assume that $S'=S^p\cup \supp_S \Sigma \neq S$ and let $L$ be
the standard Levi subgroup associated to $S'$. Let $Q, Q^-, R_Q, R_Q^-,
U_Q, U_Q^-, L^{\mathrm{ss}}, L_{\mathrm{ad}}$ be defined as above. 
Set $\widetilde \Delta = \{ D \in \Delta \, : \, \text{ if } P_\alpha \cdot D \neq D \text{ then } \alpha \in S' \}$,
then $\widetilde \Delta$ is a distinguished set and $M/\widetilde \Delta =
G/Q^-$. Moreover, $N=M^{U_Q^-}$ is a wonderful variety for $L_{\mathrm{ad}}$ 
and $M$ is obtained by parabolic induction from $N$ as above. 

\subsection{Parabolic induction: global sections}

Let now $M$ be obtained by parabolic induction from $N$ as above.
Let $\mathcal{M}$ be a line bundle on $M$ and denote by $\mathcal{N}$ its
restriction to $N$. We want to compare the sections of $\mathcal{M}$ and
$\mathcal{N}$. The restriction of sections induces a map $r_\mathcal{M}\colon \Gamma(M,\mathcal{M})\longrightarrow
\Gamma(N,\mathcal{N})$. Notice that $U\cdot N$ is a dense subset of $M$,
hence the restriction of $r_\mathcal{M}$ to $\Gamma(M,\mathcal{M})^U$ is
injective. We first describe the kernel of $r_\mathcal{M}$.

Let $\mathcal{M}=\mathcal L_F$, then $\Gamma (M,\mathcal{M})=\bigoplus s^{F-E} V_E$
where the sum is over all $E\in \mathbb{N} \Delta$ such that $E \leqslant_\Sigma F$.  For each $E\in \mathbb{N} \Delta$ set $W_E = V_E^{U_Q}$. This is an
irreducible $L$-submodule with the same highest weight as $V_E$.
Let $I_E$ be the $L$-stable complement of $W_E$ in $V_E$.

\begin{lemma}\label{lem:zeri}
We have
\[ \ker r_\mathcal{M} = \bigoplus_{E\in \mathbb{N} \Delta\, : \, E\leqslant_\Sigma F}
s^{F-E} I_E. \]
\end{lemma}

\begin{proof}By the above discussion it is enough to prove that 
$r_{\mathcal L_E}(I_E)=0$. Let $v$ be a highest
  weight vector (for the action of $L$) in $I_E$ of weight $\lambda$.
  Then $\omega_E-\lambda=\sum_{\alpha\in S} a_\alpha \alpha$ with $a_\alpha \in \mathbb{N}$.  If $S'$ is the set of simple roots for $L$ notice that, since
  $v \in I_E$, there exists $\alpha \in S\smallsetminus S'$ such that
  $a_\alpha \neq 0$. On the other hand, by the generalization of 
  Proposition~\ref{prop: decomposizione sezioni} for $N$ discussed in the previous
  section, the weights of the highest weight vectors in $\Gamma (N,\mathcal L_E)$ are
  of the form $\omega_E -\beta$ with $\beta \in \mathbb{N} S'$. In particular
  we must have $r_{\mathcal L_E}(v)=0$.
\end{proof}

Since the restriction commutes with the multiplication we recover the following.

\begin{proposition}[{\cite[Proposition~2.9]{CM}}] \label{prop: ind-par}
For all $E,F \in \mathbb{N}\Delta$ and $\gamma \in \mathbb{N}\Sigma$ such that $E+F - \gamma \in \mathbb{N} \Delta$, $s^\gamma \, V_{E+F-\gamma} \subset V_E V_F$ if and only if $s^\gamma \, W_{E+F-\gamma} \subset W_E W_F$.
\end{proposition}

In particular, using the property (6) of the previous section, we have the following. 

\begin{corollary}\label{cor: moduli ortogonali}
For all $D\in\mathbb{N}(\Delta\smallsetminus\widetilde\Delta)$ and $E\in\mathbb{N}\Delta$, $V_D V_E=V_{D+E}$.
\end{corollary}

\section{Projective normality and the covering relation}

\subsection{The covering relation}

Let $\{\omega_\alpha:\alpha\in S\}$ be the set of fundamental weights
w.r.t.\ the simple roots $S$.

For all $\lambda = \sum k_\alpha \omega_\alpha \in \mathcal{X}(T)$, denote by
$\supp(\lambda)$ the set of $\alpha\in S$ such that $k_\alpha\neq 0$ and
define its positive part $\lambda^+$, resp.\ its negative
  part $\lambda^-$, as the dominant weights
\[ \lambda^+ = \sum_{k_\alpha > 0} k_\alpha \omega_\alpha \qquad \text{and }
\qquad \lambda^- = \lambda^+-\lambda. \] 
If $\lambda\in\mathcal{X}(T)^+$, define also the height of $\lambda$ as the number
$\alt(\lambda) = \sum_{\alpha \in S} k_\alpha$.

Suppose that $\lambda$ and $\mu$ are dominant weights with $\lambda < \mu$
(w.r.t.\ the usual dominance order) and suppose that there is no
dominant weight $\nu$ such that $\lambda < \nu < \mu$: then one says
that $\mu$ covers $\lambda$ and we call $\mu - \lambda$ a
covering difference in $\mathcal{X}(T)^+$. Notice that an element $\gamma
\in \mathbb{N} S$ is a covering difference in $\mathcal{X}(T)^+$ if and only if
$\gamma^+$ covers $\gamma^-$. Although the
following proposition is an immediate consequence of \cite[Theorem~2.6]{St}, 
here we give an easy independent proof.

\begin{proposition}\label{prop: altezza 2}
If $\gamma \in \mathbb{N} S$ is a covering difference in $\mathcal{X}(T)^+$, then
$\alt(\gamma^+) \leqslant 2$.
\end{proposition}

\begin{proof}
Let $\gamma \in \mathbb{N} S$ be a covering difference and suppose that $\alpha
\in S$ is such that $\langle \gamma^+, \alpha^\vee \rangle \geqslant 2$: then
$\gamma^+ - \alpha \in \mathcal{X}(T)^+$, hence we must have $\gamma = \alpha$ and
$\alt(\gamma^+) = 2$. Hence we may assume that $\langle \gamma^+,
\alpha^\vee \rangle \leqslant 1$ for every $\alpha \in S$. Suppose that
$\alt(\gamma^+) \geqslant 3$, (up to reindexing the simple roots) we
can take $\alpha_i, \alpha_j \in \supp_S(\gamma^+)$ with $i<j$ such that
$\{\alpha_i,\dots, \alpha_j\}$ generates an irreducible subsystem of
type $\mathsf{A}$: then $\gamma^+ - (\alpha_i +\dots + \alpha_j) \in \mathcal{X}(T)^+$ and
it follows $\gamma = \alpha_i +\dots + \alpha_j$ and $\alt(\gamma^+)
= 2$ (a contradiction).
\end{proof}

We now consider the covering relation in the more general context of
wonderful varieties. Let $M$ be a wonderful variety with set of
spherical roots $\Sigma$ and with set of colors $\Delta$. For all $E =
\sum_{D \in \Delta} k_D D \in \mathbb{Z}\Delta$, define its \textit{positive
  part} $E^+$ and its \textit{negative part} $E^-$ as
\[ E^+ = \sum_{k_D > 0} k_D D \qquad \text{and } \qquad E^- = E^+-E.\] 
If $E\in\mathbb{N}\Delta$, define the \textit{height} of $E$ as the number $\alt(E) =
\sum_{D \in \Delta} k_D$. On the other hand, for all $\gamma = \sum_{\sigma
  \in \Sigma} a_\sigma \sigma \in \mathbb N \Sigma$ define its
$\Sigma$-\textit{height} as $\alt_\Sigma(\gamma) = \sum_{\sigma \in
  \Sigma} a_\sigma$.

Let $E$ and $F$ be in $\mathbb{N} \Delta$ with $E <_\Sigma F$ and suppose that
there is no $D \in \mathbb{N}\Delta$ such that $E <_\Sigma D <_\Sigma F$: then we
say that $F$ covers $E$ and we call $F - E$ a
covering difference in $\mathbb{N} \Delta$. Again, $\gamma \in \mathbb{N} \Sigma$
is a covering difference in $\mathbb{N} \Delta$ if and only $\gamma^+$ covers
$\gamma^-$.

\begin{remark}\label{rem:finiti}
Notice that there are only finitely many covering differences.
Indeed, consider the lattice $X=\{(D,E)\in \mathbb{Z}\Delta\times \mathbb{Z}\Delta\, : \, E-D\in
\mathbb{Z}\Sigma\}$ and the finitely generated monoid
$X^+=\{(D,E)\in\mathbb{N}\Delta\times\mathbb{N}\Delta\, : \, D-E\in \mathbb{N}\Sigma\}$. If
$\gamma$ is a covering difference then $(\gamma^+,\gamma^-)$ is an
indecomposable element of $X^+$.

Finally the number of indecomposable elements in $X^+$ is finite and
can be controlled as follows: let $\ell_1,\dots,\ell_t$ be the
half-lines generating the cone $X^+_\mathbb{Q}$, let $v_i=(D_i,E_i)$ be a
generator of $\ell_i\cap X^+$ and let $n_i=\alt(D_i)+\alt(E_i)$. Then if
$(D,E)$ is indecomposable $\alt(D)+\alt(E)\leqslant \sum_i n_i$.
\end{remark}

As already said in the introduction, 
in the case of the wonderful compactification of a non-Hermitian symmetric space, there
exists a root system $\Phi_\Sigma$ (the \textit{restricted root system})
which is generated by the spherical roots and such that $\Delta$ is
naturally identified with the set of fundamental weights of
$\Phi_\Sigma$ and the pairing between $\Sigma$ and $\Delta$ is the Cartan
pairing of $\Phi_\Sigma$. Although $\Sigma$ is always the basis of a root
system $\Phi_\Sigma$, in the general case it is not possible to identify
$\Delta$ with the fundamental weights of $\Phi_\Sigma$: in particular, it
may happen that $\mathbb{Z}\Sigma \subset \mathbb{Z}\Delta$ is not a sublattice of
finite index and that the semigroup of dominant weights of
$\Phi_\Sigma$ contained in $\mathbb{Z} \Sigma$ is not even contained in $\mathbb{N}\Delta$. However, one can
consider the property of Proposition~\ref{prop: altezza 2} without
modifications in this context:

\begin{twoht}
If $\gamma\in\mathbb{N}\Sigma$ is a covering difference in $\mathbb{N}
  \Delta$, then $\alt(\gamma^+) \leqslant 2$.
\end{twoht}

Notice that by the above discussion the property (2-ht) holds if $M$ is
the wonderful compactification of a non-Hermitian symmetric space. 
In the case of the wonderful compactifications of Hermitian symmetric spaces 
the same argument works without serious complications. 
In the subsequent section we show
that it is true also in the case of a model wonderful variety (of
simply connected type). We have also checked many other examples and,
as far as we know, it is possible that it holds for all wonderful
varieties.

In analogy with the case of a root system, we say that a nonzero element $D \in \mathbb{N}\Delta$
is \textit{minuscule} if it is minimal in $\mathbb{N}\Delta$ w.r.t.\ the partial order $\leqslant_\Sigma$.

\subsection{Projective normality}

The notion of low triple has been introduced in \cite{CM} 
in the case of a symmetric wonderful variety. Here we
use the same terminology with a slightly weaker definition.

\begin{definition}	\label{dfn: low triples}
Let $M$ be a wonderful variety with set of spherical roots $\Sigma$ and
with set of colors $\Delta$. Let $D,E,F \in \mathbb{N}\Delta$ be such
that $F \leqslant_\Sigma D+E$. The triple $(D,E,F)$ is called a \textit{low
  triple} if:
\begin{itemize}
\item[] for all $D' \leqslant_\Sigma D$ and $E' \leqslant_\Sigma E$ such that $F
  \leqslant_\Sigma D' + E'$, it holds $D' = D$ and $E' = E$.
\end{itemize}
The triple $(D,E,F)$ is called a
\textit{fundamental triple} if $D,E \in \Delta$.
\end{definition}

Notice that if $(D,E,F)$ is a low triple, then $s^{D+E-F} V_F \subset
\Gamma(M,\mathcal L_D)\Gamma(M,\mathcal L_E)$ if and only if $s^{D+E-F} V_F \subset
V_D V_E$.

\begin{lemma}	\label{lem: proiettiva normalita}
Let $M$ be a wonderful variety and let $n$ be such that $\alt(\gamma^+) \leqslant n$
for every covering difference $\gamma$. If 
\begin{equation}\label{P2}
s^{D+E-F} V_F \subset V_D V_E,
\end{equation}
for all low triples $(D,E,F)$ with $\alt(D+E)\leqslant n$, then 
the multiplication map
\[ m_{D,E} \colon \Gamma(M,\mathcal L_D) \otimes \Gamma(M,\mathcal L_E)
\longrightarrow \Gamma(M,\mathcal L_{D+E})
\]
is surjective for all $D,E \in \mathbb{N}\Delta$.
\end{lemma}

\begin{proof} For any $D\in \mathbb{Z}\Delta$ let $\Gamma_D =\Gamma(M,\mathcal L_D)$.
For any triple $(D,E,F)$ with $F \leqslant_\Sigma D+E$ we want to show that $s^{D+E-F}V_F \subset
\Gamma_D\Gamma _E$. We proceed by induction first on $\alt_\Sigma
(D+E-F)$ and then on $\alt(D+E)$. If $\alt_\Sigma (D+E-F)=0$ the claim is trivial.

If $(D,E,F)$ is not a low triple then there exist $D'\leqslant_\Sigma D$
and $E'\leqslant_\Sigma E$ such that $F\leqslant_\Sigma D'+E'$ and $\alt_\Sigma(D'+E'-F)
< \alt_\Sigma(D+E-F)$. Hence the claim is true for $(D ' ,E ' ,F)$, and
\[
s^{D+E-F}V_F = s^{D-D'+E-E'}s^{D'+E'-F}V_F\subset s^{D-D'+E-E'}\Gamma _{D '}\Gamma _{E '} \subset 
\Gamma_D\Gamma _E.
\]

If $(D,E,F)$ is a low triple and $\alt(D+E)\leqslant n$ then the claim is
true by assumption.

Assume now that $\alt(D+E)> n$ and that $(D,E,F)$ is a low triple. Let
$F_1 \in \mathbb{N}\Delta$ be an element covered by $D+E$ such that $F_1 \geqslant_\Sigma
F$. Since $\alt(D+E)> \alt(\gamma^+)$ for every covering difference $\gamma \in \mathbb{N}\Sigma$,
it follows that $\supp(F_1) \cap \supp(D+E) \neq \varnothing$. Fix $D_0 \in \supp(F_1) \cap
\supp(D+E)$ and set 
\[ (D_1,E_1) = \begin{cases} (D-D_0,\,E) & \text{ if } D_0 \in
  \supp(E)\\ (D,\,E-D_0) & \text{ otherwise} \end{cases}\]
Choose $F_2$ in $\mathbb{N}\Delta$ minimal w.r.t.\ $\leqslant_\Sigma$ such that $F - D_0 \leqslant_\Sigma
F_2 \leqslant_\Sigma F_1-D_0$.  Notice that $\alt_\Sigma(D_0+F_2-F) \leqslant \alt_\Sigma(F_1-F)<
\alt_\Sigma(D+E-F)$.  Moreover $(D_0,F_2,F)$ is a low triple; indeed if there exist
$F_2'\leqslant_\Sigma F_2$ and $D_0'\leqslant_\Sigma D_0$ in $\mathbb{N}\Delta$ such that $F\leqslant_\Sigma D_0'+F_2'$ then
\[ F \leqslant_\Sigma D_0' + F_2' \leqslant_\Sigma D_0' + D_1 + E_1 \leqslant_\Sigma D + E. \]
Since $(D,E,F)$ is a low triple, we get $D_0'=D_0$, while since $F_2$ is minimal we get $F_2'=F_2$.
Hence, by induction, it follows that $s^{D_0+F_2-F}V_F\subset V_{D_0}V_{F_2}$.

Notice that we have $\alt_\Sigma (D_1+E_1-F_2) \leqslant \alt_\Sigma
(D+E-F)$ and that $\alt(D_1+E_1)=\alt(D+E)-1$, hence we can apply the
inductive hypothesis to $(D_1,E_1,F_2)$. So
\[
	s^{D+E-F} V_F = s^{D_1+E_1-F_2} s^{D_0+F_2-F} V_F \subset
	s^{D_1+E_1-F_2} V_{F_2}V_{D_0} \subset \Gamma_{D_1}\Gamma_{E_1}V_{D_0}\subset \Gamma_D\Gamma_E.	
\]
\end{proof}

Together with Remark~\ref{rem:finiti}, the lemma implies that to
prove the surjectivity of $m_{D,E}$ for all $D,E\in \mathbb{N}\Delta$ it is
enough to check a finite number of cases. In particular if the property
(2-ht) holds, it is enough to check the above inclusion \eqref{P2}  
only for the low fundamental triples.

In Section~\ref{subsect: non projnorm} we show that there exist
wonderful varieties possessing low fundamental triples $(D,E,F)$ such
that $s^{D+E-F} V_F \not \subset V_D V_E$: in particular the
multiplication of sections of line bundles on a wonderful variety is
not necessarily surjective.

\section{The covering relation for model wonderful varieties}\label{sec: coperture modello}

As said in the introduction, a model homogeneous space for a
reductive group $G$ (not necessarily simply connected) is a quasi
affine $G$-homogeneous space whose coordinate ring is a model representation of
$G$: each irreducible representation of $G$ appears exactly with
multiplicity one. By \cite{Lu3}, for every group $G$
there exists a wonderful variety $M$ such that for every model homogenous
space $G/H_0$ there exists a point in $M$ whose stabilizer is
equal to $N_G(H_0)$ and, viceversa, if $H$ is the stabilizer of a point in $M$
and we set $H_0$ to be the intersections of the kernels of the
multiplicative characters of $H$ then $G/H_0$ is a model homogeneous space.
The variety $M$ is called the model wonderful variety of $G$.

Notice that $H_0$ has no characters, hence the set of 
spherical vectors (i.e., the eigenvectors of $H$) in a simple $G$-module $V$,
which is a subspace of dimension at most one,
can be also characterized as the set of vectors in $V$ fixed by $H_0$.

The description of the model wonderful varieties of an almost simple group $G$ is given in \cite{Lu3}.
For type $\mathsf{B}$ there are two different model wonderful varieties,
according as $G$ is isomorphic to the special orthogonal group
or to the spin group. In all the other cases the model wonderful variety of
$G$ is the same as the model wonderful variety of the adjoint group of $G$. 

In this section we describe the covering relation for a model
wonderful variety $M$ of simply connected type, and we prove that
the property (2-ht) holds in this case. For model wonderful varieties of
simply connected type the set of colors $\Delta$ is in bijection, via
$\omega$, with the set of fundamental weights, and the set of spherical
roots $\Sigma$ is the set of the sums $\alpha+\beta$ with 
$\alpha, \beta$ non-orthogonal simple roots.

Clearly, in order to study the multiplication of sections of line
bundles and the partial order of dominant weights in the case of a
model wonderful variety, we may reduce to the case where $G$ is almost simple.
Moreover, if $\gamma=\sum_{\alpha\in S} a_\alpha \alpha \in \mathbb{Z}\Sigma$
we denote by $\supp_S(\gamma)$ the set of simple roots $\alpha$
such that $a_\alpha\neq 0$. We have the following lemma, whose proof is immediate.

\begin{lemma}
Let $\gamma$ be a covering difference in $\mathbb{N}\Delta$ w.r.t.\ $\Sigma$, then
$\supp_S(\gamma)$ is a connected subset of $S$ and $\gamma$ is a
covering difference for the model variety of the simply connected
group associated to the root subsystem generated by $\supp_S(\gamma)$.
\end{lemma}

Therefore, we only classify the covering
differences $\gamma$ such that $\supp_S(\gamma)=S$.
In the following two sections, $r$ denotes the semisimple rank of
the group and $\Delta=\{D_1,\dots,D_r\}$ is the set of colors,
labelled in such a way that $\omega_{D_i}$ is the fundamental weight $\omega_i$.
For simplicity we also set $D_i=0$ for all $i\leqslant0$ and for all $i>r$. 
Similarly we denote $\Sigma = \{\sigma_1, \ldots, \sigma_{r-1}\}$
(see the following subsections for the definition of the single spherical roots $\sigma_i$),
and we set $\sigma_i = 0$ for all $i \leqslant 0$ and for all $i \geqslant r$.

For the rest of the section, $\gamma$ will be a covering difference in $\mathbb{N} \Delta$ w.r.t. $\Sigma$, and we set
$$\gamma = \sum_{i=1}^{r-1} a_i \sigma_i = \sum_{i=1}^r c_i D_i.$$
For convenience we also set $a_i = c_j = 0$ for all $i,j\leqslant 0$,
for all $i \geqslant r$ and for all $j > r$.

\subsection{Type $\mathsf A_r$ ($r\geqslant2$)} \label{ss: coperture A}

Let $\Sigma = \{\sigma_1,\dots, \sigma_{r-1}\}$ be the set
of spherical roots where $\sigma_i = \alpha_i + \alpha_{i+1}$.

\begin{proposition} \label{prop: coperture A}
Let $\gamma\in\mathbb{N}\Sigma$ be a covering difference in $\mathbb{N}\Delta$ with
$\supp_S(\gamma) = S$. Then either
\begin{enumerate}
\item $r$ is even and $\gamma=\sum_{i=1}^{r/2}\sigma_{2i-1}=D_1+D_r$, or
\item $r$ is odd and $\gamma=\sum_{i=1}^{r-1}\sigma_i=D_2+D_{r-1}$.
\end{enumerate}
\end{proposition}

\begin{proof}
Clearly, $a_1>0$. Assume first that $a_2>0$. Take the maximum integer $k$ such that
$a_i>0$ for all $i \leqslant k$. Then $c_j\geqslant0$ for all $3 \leqslant j
\leqslant k+2$ otherwise $\gamma^- + \sum_{i=1}^{j-2} \sigma_i \in \mathbb{N} \Delta$.

If $k$ were $< r-1$, then $c_{k+1}+c_{k+2}=-a_{k-1}-a_{k+3}$
would be $<0$, a contradiction. Therefore, $\supp_\Sigma(\gamma)=\Sigma$.

Now, since $\gamma^-+\sum_{i=1}^{r-1}\sigma_i\in\mathbb{N}\Delta$, $\gamma$ must
necessarily be equal to $\sum_{i=1}^{r-1}\sigma_i$ which is a covering
difference if and only if $r$ is odd: indeed, if $r$ is even,
$\sum_{i=1}^{r/2}\sigma_{2i-1}\in\mathbb{N}\Delta$.

Assume now that $a_2=0$. Take the maximum odd integer $k$ such that
$a_i=0$ for all even $i < k$. Then $c_k\geqslant0$ otherwise
$\gamma^-+\sum_{i=1}^{(k-1)/2}\sigma_{2i-1}\in\mathbb{N}\Delta$. Furthermore,
$c_{k-1}\geqslant0$ otherwise, reasoning as above, $a_j>0$ for all
$j\geqslant k$ and $\gamma^-+\sum_{i=k}^{r-1}\sigma_i\in\mathbb{N}\Delta$.

If $k$ were $< r-1$, then $c_{k-1}+c_k=-a_{k+1}$ would be $<0$,
a contradiction. Therefore, $a_i>0$ iff $i$ is odd, and $r$ is even.

Now, since $\gamma^-+\sum_{i=1}^{r/2}\sigma_{2i-1}\in\mathbb{N}\Delta$, $\gamma$ must
necessary be equal to $\sum_{i=1}^{r/2}\sigma_{2i-1}$ which is a
covering difference.
\end{proof}

\subsection{Type $\mathsf B_r$ ($r\geqslant2$)} \label{ss: coperture B}

Let $\Sigma = \{\sigma_1,\dots, \sigma_{r-1}\}$ be the set
of spherical roots where $\sigma_i = \alpha_i + \alpha_{i+1}$.

\begin{proposition}	\label{prop: coperture B}
Let $\gamma \in \mathbb{N}\Sigma$ be a covering difference in $\mathbb{N}\Delta$ with
$\supp_S(\gamma)=S$. Then $r$ is even and $\gamma =
\sum_{i=1}^{r/2}\sigma_{2i-1}=D_1$.
\end{proposition}

\begin{proof}
Clearly, $a_{r-1}>0$. Suppose that $k<r$ is such that $r-k$ is odd and $c_i \leqslant 0$ for
every $k < i <r$ with $r-i$ odd. Then we have the inequalities (where
$a_i=0$ if $i\leqslant0$)
\[ a_{r-1} \leqslant a_{r-3} - a_{r-2} \leqslant a_{r-5} - a_{r-4} \leqslant\ldots
\leqslant a_{k+2} - a_{k+3} \leqslant a_{k} - a_{k+1}
\] and it follows that $a_j \geqslant a_{r-1} > 0$ for every $j \geqslant k$
with $r-j$ odd. In particular, $k$ must be $>0$.

Let $k$ be maximal with $r-k$ odd and $c_k>0$. Denote $\gamma_0 =
\sum_{i= 0}^{(r-k-1)/2} \sigma_{k+2i} = -D_{k-1} + D_k$: then $\gamma_0
\leqslant_\Sigma \gamma$ and $\gamma^+ - \gamma_0 \in \mathbb{N}\Delta$, hence we must have
$\gamma = \gamma_0$, $k=1$ and $r$ even.
\end{proof}

\subsection{Type $\mathsf C_r$ ($r\geqslant3$)}\label{ss: coperture C}

Let $\Sigma = \{\sigma_1,\dots, \sigma_{r-1}\}$ be the set
of spherical roots where $\sigma_i = \alpha_i + \alpha_{i+1}$.

\begin{proposition} \label{prop: coperture C}
Let $\gamma \in \mathbb{N}\Sigma$ be a covering difference in $\mathbb{N}\Delta$ with
$\supp_S(\gamma)=S$. Then $\gamma = \sum_{i=1}^{r-1}\sigma_i=D_2$.
\end{proposition}

\begin{proof}
Clearly, $a_{r-1}>0$. Let $k$ be maximal such that $c_k>0$.
Then we have the following inequalities (for $k<r$)
\[a_{k-1} \geqslant a_k \geqslant\ldots \geqslant a_{r-1}.\]
It follows that $a_i>0$ for every $k-1 \leqslant i < r$. Therefore, $\gamma^+
- \sum_{i=k-1}^{r-1} \sigma_{i} = \gamma^+ + c_{k-2} - c_k \in
\mathbb{N} \Delta$ and we get $k=2$ with $\gamma = \sum_{i=1}^{r-1}
\sigma_i$.
\end{proof}

\subsection{Type $\mathsf D_r$ ($r\geqslant4$)} \label{ss: coperture D}

Let $\Sigma = \{\sigma_1,\dots, \sigma_{r-1}\}$ be the set
of spherical roots where $\sigma_i = \alpha_i + \alpha_{i+1}$ if
$i<r-1$ and $\sigma_{r-1} = \alpha_{r-2}+\alpha_r$.

\begin{proposition} \label{prop: coperture D}
Let $\gamma \in \mathbb{N}\Sigma$ be a covering difference in $\mathbb{N}\Delta$ with
$\supp_S(\gamma)=S$. Then $r$ is odd and
$\gamma=\sum_{i=1}^{(j-1)/2}\sigma_{2i-1}+\sum_{i=(j+1)/2}^{(r-3)/2}2\sigma_{2i-1}\;+\sigma_{r-2}+\sigma_{r-1}=D_1-D_{j-1}+D_j$
for $j$ odd $\leqslant r-2$.
\end{proposition}

\begin{proof}
Clearly, we have the inequalities $a_{r-2} > 0$ and $a_{r-1}>0$.
One has $c_{r-2}\leqslant1$, otherwise 
$\gamma^+-(-2D_{r-3}+2D_{r-2})\in\mathbb{N}\Delta$,
thus $\gamma=-2D_{r-3}+2D_{r-2}=\sigma_{r-2}+\sigma_{r-1}$,
but $\supp_S(\gamma)\neq S$.
Since $c_{r-2}=-a_{r-4} + a_{r-3} + a_{r-2} + a_{r-1}$ we get $a_{r-4}>0$ 
and moreover, if $c_{r-2}\leqslant0$, $a_{r-4}>1$. 

We go on this way step-by-step. Let $k$ be $<(r-1)/2$.
Assume $a_{r-2i}>0$ for all $i\leqslant k$ 
and moreover if, for some $j<k$, $c_{r-2i}\leqslant0$ for all $i\leqslant j$
then $a_{r-2i-2}>1$ for all $i\leqslant j$. 
One has $\sum_{i=1}^k c_{r-2i}\leqslant 1$, 
otherwise there would exist $1\leqslant i_1\leqslant i_2\leqslant k$ such that
$\gamma^+-(-D_{r-2i_2-1}+D_{r-2i_2}-D_{r-2i_1-1}+D_{r-2i_1})\in\mathbb{N}\Delta$,
thus
$\gamma=-D_{r-2i_2-1}+D_{r-2i_2}-D_{r-2i_1-1}+D_{r-2i_1}=\sum_{i=i_1+1}^{i_2}\sigma_{r-2i}+\sum_{i=2}^{i_1}2\sigma_{r-2i}\;+\sigma_{r-2}+\sigma_{r-1}$,
but $\supp_S(\gamma)\neq S$. Since
\[
	\sum_{i=1}^k c_{r-2i}=-a_{r-2k-2} + a_{r-2k-1} + a_{r-2} + a_{r-1},
\]
we get $a_{r-2k-2}>0$
and moreover, if $c_{r-2i}\leqslant0$ for all $i\leqslant k$,
$a_{r-2k-2}>1$.

Therefore, we have that $r$ is odd, $\sum_{i=1}^{(r-1)/2} c_{r-2i}>1$ and
$\gamma=-D_{r-2i_2-1}+D_{r-2i_2}-D_{r-2i_1-1}+D_{r-2i_1}$ with $1\leqslant
i_1\leqslant i_2\leqslant (r-1)/2$. Since $\supp_S(\gamma)=S$, we must have $i_2=(r-1)/2$.
\end{proof}

\subsection{Type $\mathsf E_r$ ($6\leqslant r\leqslant8$)} \label{ss: coperture E}

Let $\Sigma = \{\sigma_1,\dots, \sigma_{r-1}\}$ be the set
of spherical roots where $\sigma_1=\alpha_1+\alpha_3$, $\sigma_2=\alpha_2+\alpha_4$ and, for $i\geqslant3$, $\sigma_i = \alpha_i + \alpha_{i+1}$.

\begin{proposition} \label{prop: coperture E}
Let $\gamma\in\mathbb{N}\Sigma$ be a covering difference in $\mathbb{N}\Delta$ with
$\supp_S(\gamma) = S$. 

If $r=6$,
\begin{itemize}
\item[(1)] either $\gamma=\sigma_1+\sigma_2+\sigma_3+\sigma_4+\sigma_5=D_4-D_2$ 
\item[(2)] or $\gamma=2\sigma_1+2\sigma_2+\sigma_3+\sigma_4+2\sigma_5=D_1+D_6$.
\end{itemize}

If $r=7$,
\begin{itemize}
\item[(3)] $\gamma=\sigma_1+\sigma_2+\sigma_4+\sigma_5+\sigma_6=D_1+D_6-D_3$.
\end{itemize}

If $r=8$, 
\begin{itemize}
\item[(4)] either $\gamma=\sigma_1+\sigma_2+\sigma_3+\sigma_4+2\sigma_5+\sigma_6+\sigma_7=D_6-D_2$ 
\item[(5)] or $\gamma=2\sigma_1+2\sigma_2+\sigma_3+\sigma_4+2\sigma_5+\sigma_7=D_1+D_8-D_7$.
\end{itemize}
\end{proposition}

\begin{proof}
Suppose $r=6$. Clearly, $a_1,a_2,a_5$ are $>0$. Moreover, $c_1+c_4+c_6=a_2$.

If $c_2$ were $> 0$, then $c_1,c_4,c_6$ would all be $\leqslant0$: indeed either $\gamma^+- \sigma_2$, or $\gamma^+ - (\sigma_1 +\sigma_2)$ or $\gamma^+ - (\sigma_2 +\sigma_5)$ would be in $\mathbb{N} \Delta$, contradicting $\supp_S(\gamma) = S$ (we will use such an argument repeatedly in the rest of the proof). Therefore,
$c_2\leqslant0$ and $a_3+a_4>0$. By symmetry, we can assume $a_3>0$.

Assume $c_4>0$. Then $c_3\leqslant0$. If $a_4$ were zero, $c_2+c_3$ would be
equal to $a_1$ (which is $>0$). Therefore, $a_4>0$ and $\gamma=D_4-D_2$.

Assume now $c_4\leqslant0$. We can also assume $c_6>0$ (by symmetry, if
$a_4>0$). Then both $c_3$ and $c_5$ are $\leqslant0$, hence $a_4$ is
$>0$. Since $c_2+c_3+c_4+c_5=0$, all $c_i=0$, for $2\leqslant i\leqslant
5$. This implies $a_1,a_2,a_5\geqslant2$, $c_6\leqslant1$ (since
$(2\sigma_1+\sigma_2+\sigma_4)^+=2D_1$) and $c_1>0$. Therefore,
$\gamma=D_1+D_6$.

Suppose $r = 7$. Clearly, $a_1,a_2,a_6$ are $>0$. Moreover, $c_1+c_4+c_6=a_2+a_6$.

Assume $a_3=0$. We get $c_1>0$. Then both $c_2$ and $c_3$ are $\leqslant0$, hence $a_4$ is $>0$. Then $c_7\leqslant0$, hence $a_5>0$. Moreover, $c_4\leqslant0$ (since $(\sigma_1+\sigma_2+\sigma_4)^+=D_1+D_4$), $c_1\leqslant1$ (since $(2\sigma_1+\sigma_2+\sigma_4)^+=2D_1$) and $c_6>0$. Therefore, $\gamma=D_1+D_6-D_3$.

Assume now $a_3>0$. If $c_4$ were $>0$, then both $c_2$ and $c_3$ would be $\leqslant0$, hence $a_4>0$. This would imply $c_7\leqslant0$, hence $a_5>0$, which is impossible (since $(\sigma_1+\sigma_2+\sigma_3+\sigma_4+\sigma_5)^+=D_4$). Then $c_4\leqslant0$.

If $c_1$ were $>0$, then both $c_2$ and $c_3$ would be $\leqslant0$, hence $a_4>0$. This would imply $c_7\leqslant0$ hence $a_5>0$, and $c_1\leqslant1$ (since $(2\sigma_1+\sigma_2+\sigma_4)^+=2D_1$) hence $c_6>0$, which is impossible (since $(\sigma_1+\sigma_3+\sigma_4+\sigma_5+\sigma_6)^+=D_1+D_6$). Then $c_1\leqslant0$.

If $c_6$ were $>0$, we would get $c_7\leqslant0$, $a_5>0$ and $c_6\leqslant1$,
which is impossible (since $c_1+c_4+c_6=a_2+a_6\geqslant2$).

Suppose $r = 8$. Clearly, $a_1,a_2,a_7$ are $>0$. Moreover, $c_1+c_4+c_6+c_8=a_2$.

If $c_4$ were $>0$, we would get $c_2\leqslant0$, $a_3+a_4>0$, actually
$a_3>0$ ($a_3=0$ would imply $c_1>0$, but also $a_4>0$ and
$c_1\leqslant0$), hence $c_3\leqslant0$, $a_4>0$, $c_4\leqslant1$ and $c_1\leqslant0$,
therefore $a_5>0$ (since $a_2+a_4-a_5=c_1+c_4\leqslant1$), which is
impossible (since $(\sigma_1+\sigma_2+\sigma_3+\sigma_4+\sigma_5)^+=D_4$). Then
$c_4\leqslant0$.

Assume $c_6>0$. Then $a_5>0$ (since $a_6>0$ implies $c_7\leqslant0$), hence
both $c_2$ and $c_5$ are $\leqslant0$.

If $a_3$ were zero, we would get $c_1>0$, $a_4>0$, $a_6>0$ (since
$a_5+a_6=c_2+c_5\leqslant0$), which is impossible (since
$(\sigma_1+\sigma_3+\sigma_4+\sigma_5+\sigma_6)^+=D_1+D_6$). Then $a_3>0$, hence
$c_3\leqslant0$ and $a_4>0$.

If $a_6$ were zero, since $c_2+c_3+c_4+c_5=-a_6$, all $c_i$ would be
zero, for $2\leqslant i\leqslant 5$. This would imply $a_5\geqslant2$ (since
$-2a_2+a_5=c_2+c_5=0$), on the other hand $c_8$ would be $>0$, which
is impossible (since $(\sigma_2+\sigma_3+2\sigma_5+\sigma_7)^+=D_6+D_8$). Then
$a_6>0$, hence $c_7\leqslant0$ and $a_5\geqslant2$. Therefore, $\gamma=D_6-D_2$.

Assume now $c_6\leqslant0$ and $c_1>0$. Then both $c_2$ and $c_3$ are
$\leqslant0$, hence $a_4>0$ and $c_5\leqslant0$. If $a_6$ were $>0$, we would get
$c_7\leqslant0$, $a_5>a_6$, $a_3>0$ (since $-2a_3+a_5-a_6=c_2+c_5\leqslant0$),
this would imply $a_1\geqslant2$ hence $c_1\leqslant1$ (since
$(2\sigma_1+\sigma_2+\sigma_4)^+=2D_1$) and on the other hand $a_2\geqslant2$
hence $c_8>0$ (since $c_1+c_4+c_6+c_8=a_2$), but this is impossible
(since
$(2\sigma_1+2\sigma_2+\sigma_3+\sigma_4+2\sigma_5+\sigma_7)^+=D_1+D_8$). Then
$a_6=0$, hence $a_5>0$ and $c_8>0$. Furthermore, since $c_i=0$ for all
$2\leqslant i\leqslant 5$, $a_3>0$ and $a_1,a_2,a_5\geqslant2$. Therefore,
$\gamma=D_1+D_8-D_7$.

Finally, assume both $c_1$ and $c_6$ $\leqslant0$. Recall that $c_4\leqslant0$,
therefore $c_8\geqslant a_2>0$. Futhermore, $a_3>0$, $c_7\leqslant0$, $a_5>0$,
$c_2\leqslant0$, $c_3\leqslant0$ and $a_4>0$. Then $a_2\geqslant2$, hence $c_8\geqslant2$,
$a_7\geqslant2$ and $a_5\geqslant2$, which is impossible (since
$(\sigma_2+\sigma_3+2\sigma_5+2\sigma_7)^+=2D_8$).
\end{proof}

\subsection{Type $\mathsf F_4$} \label{ss: coperture F}

Let $\Sigma = \{\sigma_1, \sigma_2, \sigma_3\}$ be the set
of spherical roots where $\sigma_i = \alpha_i + \alpha_{i+1}$.

\begin{proposition} \label{prop: coperture F}
Let $\gamma\in\mathbb{N}\Sigma$ be a covering difference in $\mathbb{N}\Delta$ with
$\supp_S(\gamma) = S$. Then either
\begin{enumerate}
\item $\gamma=\sigma_1+\sigma_2+2\sigma_3=D_4$, or 
\item $\gamma=\sigma_1+\sigma_3=D_1+D_4-D_3$.
\end{enumerate}
\end{proposition}

\begin{proof}
Clearly, $a_1,a_3>0$.
Since $\sigma_2^+=D_2$, it follows $c_2\leqslant0$.
Then $c_4>0$ (since $c_2+c_4=a_1$), hence $a_3>a_2$.
If $a_2>0$, then $\gamma=D_4$, whereas if $a_2=0$,
then $c_1>0$ and $\gamma=D_1+D_4-D_3$.
\end{proof}

\subsection{Type $\mathsf G_2$} \label{ss: coperture G}

Let $\Sigma = \{\sigma\}$ be the set
of spherical roots where $\sigma = \alpha_1 + \alpha_2$.
The proof of the following proposition is trivial.

\begin{proposition} \label{prop: coperture G}
Let $\gamma\in\mathbb{N}\Sigma$ be a covering difference in $\mathbb{N}\Delta$ with
$\supp_S(\gamma) = S$. Then $\gamma=\sigma=D_2-D_1$.
\end{proposition}

\section{Projective normality of model wonderful varieties}\label{sec: low model}

In this section we prove the following.

\begin{theorem}	\label{teo: model}
Let $M$ be a model wonderful variety of simply connected type. The multiplication of global sections
\[
m_{\mathcal L,\mathcal L'} \colon \Gamma(M,\mathcal L) \otimes
\Gamma(M,\mathcal L') \longrightarrow \Gamma(M,\mathcal L \otimes
\mathcal L')
\]
is surjective for all globally generated line bundles $\mathcal
L,\mathcal L' \in \Pic(M)$.
\end{theorem}

Following Lemma~\ref{lem: proiettiva normalita}, we first
classify all the low fundamental triples for
the model wonderful varieties of simply connected type, and
prove that these triples all satisfy the condition \eqref{P2}.
As in the previous section, we can restrict ourselves to the case of
an almost simple group $G$. As shown in the proof of Theorem~\ref{teo: model}
at the end of the section, it is enough to consider only low fundamental triples $(D,E,F)$
such that $\supp_S(D+E-F)=S$.

We keep the notation of the previous section. We denote by
$H$ the stabilizer of a point $x_0$ in the open $G$-orbit of $M$ and by
$H_0$ the intersection of the kernels of the multiplicative characters
of $H$.

\subsection{Type $\mathsf{A}_r$}

\begin{lemma}	\label{lemma: triple fondamentali piatte A}
Let $(D_p,D_q,F)$ be a low fundamental triple with $\supp_S(D_p+D_q-F)
= S$. Then $F=0$ and $p+q = r+1$. If moreover $r$ is odd, then $p$ and $q$ are
even.
\end{lemma}

\begin{proof}
Notice that every fundamental triple is low: indeed, 
by Proposition~\ref{prop: coperture A}, for every $\eta\in\mathbb{N}\Sigma$ covering
difference in $\mathbb{N}\Delta$ one has $\alt(\eta^+) = 2$, hence every
color is minimal in $\mathbb{N}\Delta$ w.r.t.\ $\leqslant_\Sigma$. Therefore, we only
need to compute the fundamental triples $(D_p,D_q,F)$ with
$\supp_S(D_p+D_q-F) = S$.

Take a sequence
\[
	F = F_n <_\Sigma F_{n-1} <_\Sigma\ldots <_\Sigma F_0 = D_p + D_q
\] such that $F_i \in \mathbb{N}\Delta$ and $\gamma_i = F_{i-1} - F_i$ is a
        covering difference for every $i \leqslant n$. 
By Proposition~\ref{prop: coperture A}, it follows that if $\gamma_i^+ =
        F_{i-1}=D_s+D_t$ (with $1\leqslant s\leqslant t\leqslant r$) then $\gamma_i^- =
        F_i=D_{s-j}+D_{t+j}$ (with $j$ equal to 1 or 2). Therefore,
        $\supp_S(\gamma_n)=\supp_S(D_p+D_q-F) =S$, $F = \gamma_n^- = 0$
        and $q= r+1-p$.

If $r$ is odd, all the covering differences $\gamma_i= F_{i-1} - F_i$
are of the type~(2) of Proposition~\ref{prop: coperture A}, then $p$ is even.
\end{proof}

\begin{proposition}\label{prp:triplemodelloAdispari}
Let $r$ be odd and let $(D,E,F)$ be a low fundamental triple with
$\supp_S(D+E-F)=S$. Then $ s^{D+E-F}V_F\subset V_D V_E $.
\end{proposition}

\begin{proof}
By the previous lemma, $(D,E,F)=(D_p,D_{r+1-p},0)$ and $p$ is even (as
well as $r+1-p$).

Set $\Delta_{\mathrm{odd}} = \{D_i \in \Delta \, : \, i \text{ is odd}\}$; the
subset $\Delta_{\mathrm{odd}} \subset \Delta$ is distinguished and the
quotient $M' = M/\Delta_{\mathrm{odd}}$ is a symmetric wonderful variety
(with spherical roots $\alpha_{2k-1}+2\alpha_{2k}+\alpha_{2k+1}$).

By Corollary~\ref{cor: sezioni e quozienti} together with the
surjectivity of the multiplication map in the symmetric case
\[
      \Gamma(M,\mathcal L_D) \Gamma(M,\mathcal L_E) =
        \Gamma(M',\mathcal L_D)\Gamma(M',\mathcal L_E) = \Gamma(M',\mathcal L_{D+E}) =
        \Gamma(M,\mathcal L_{D+E}). 
\]
\end{proof}

\begin{proposition}\label{prp:triplemodelloApari}
Let $r$ be even and let $(D,E,F)$ be a low fundamental triple with
$\supp_S(D+E-F) = S$. Then $s^{D+E-F}V_F \subset V_D V_E$.
\end{proposition}

\begin{proof}
By Lemma~\ref{lemma: triple fondamentali piatte A}, we have $F=0$ and
$V(\omega_E) \simeq V(\omega_D)^*$, hence $V(\omega_D)^* \otimes V(\omega_E)^*
\simeq \mathrm{End}(V(\omega_D))$.

If $r$ is even, the stabilizer $H$ of a point in the open
$G$-orbit of $M$ is the normalizer in $G$ of $\Sp(r)$ and in
particular is reductive (see \cite{Lu3}).  Therefore, the
one-dimensional $H$-submodules of $V(\omega_D)^*$ and of $V(\omega_E)^*$
associated respectively with $D$ and $E$ are dual to each other, hence
we may choose the $H$-eigenvectors $h_D \in V(\omega_D)^*$ and $h_E \in
V(\omega_E)^* \simeq V(\omega_D)$ in such a way that $h_D(h_E) = 1$. If we
complete $h_E$ and $h_D$ to dual bases $\{h_E, v_1,\dots, v_n\}
\subset V(\omega_D)$ and $\{h_D, v_1^*,\dots, v_n^*\} \subset
V(\omega_D)^*$, then the identity $\mathrm{Id} \in
\mathrm{End}(V(\omega_D))$ is a $G$-invariant element which, as tensor,
is described as follows:
\[ \mathrm{Id} = h_D \otimes h_E + \sum_{i=1}^n v_i^* \otimes v_i.\] 
Therefore, $h_D \otimes h_E$ has a non-zero projection on the
isotypic component of highest weight zero and by 
Lemma~\ref{lemma: supporto moltiplicazione} 
we get that $s^{D+E}V_0 \subset V_D V_E$.
\end{proof}

\subsection{Type $\mathsf{B}_r$}

\begin{proposition}
There are no low triples $(D,E,F)$ with $\sigma_{r-1} \in \supp_{\Sigma}(D+E-F)$.
\end{proposition}

\begin{proof}
Set
\[ \Delta^{\mathrm{even}}=\{D_i \in \Delta \, : \, r-i \text{ is even}\}\]
\[ \Delta^{\mathrm{odd}}=\{D_i \in \Delta \, : \, r-i \text{ is odd}\}\]

If $F<_\Sigma D+E$ with $\sigma_{r-1} \in \supp_{\Sigma}(D+E-F)$, then
$\supp(D+E)\cap\Delta^{\mathrm{odd}}\neq\varnothing$. Indeed, if $\supp(D+E)$ were included
in $\Delta^{\mathrm{even}}$, take a sequence
\[ F = F_n <_\Sigma F_{n-1} <_\Sigma\ldots <_\Sigma F_0 = D+E
\] of coverings in $\mathbb{N}\Delta$: every covering difference
$\gamma_i=F_{i-1}-F_i$ would necessarily be of the type~(2)  
of Proposition~\ref{prop: coperture A}, hence
$\sigma_{r-1}\notin\supp_\Sigma(D+E-F)$.
 
Let $k < r$ be the maximum such that $D_k \in \supp(E+F) \cap
\Delta^{\mathrm{odd}}$. Set $D+E-F=\gamma = \sum_{i=1}^{r-1} a_i \sigma_i$. Since $c(\gamma,
D_i ) \leqslant 0$ for every $k < i <r$ with $r-i$ odd, as in the proof of
Proposition~\ref{prop: coperture B} it follows that $a_j \geqslant a_{r-1}
> 0$ for every $j \geqslant k$ with $r-j$ odd.

Denote $\gamma_0 = \sum_{i= 0}^{(r-k-1)/2} \sigma_{k+2i} = - D_{k-1} +D_k$ and 
$F'= D+E-\gamma_0$. Then $F' \in \mathbb{N}\Delta$ and $F \leqslant_\Sigma F' <_\Sigma D+E$. Set
\[ (D',E') = \left\{\begin{array}{cc} (D - \gamma_0,E) & \text{ if } D_k
\in \supp(D)\\ (D,E - \gamma_0) & \text{ otherwise}
		\end{array}\right.:
\] then $D' \leqslant_\Sigma D$, $E'\leqslant_\Sigma E$ and $F \leqslant_\Sigma D'+E'
<_\Sigma D+E$, hence $(D,E,F)$ is not a low triple.
\end{proof}

\subsection{Type $\mathsf{C}_r$}

\begin{proposition}
There are no low triples $(D,E,F)$ with $\sigma_{r-1} \in \supp_{\Sigma}(D+E-F)$.
\end{proposition}

\begin{proof}
Let $D,E,F \in \mathbb{N}\Delta$ with $F <_\Sigma D+E$ and $\sigma_{r-1} \in \supp_\Sigma(D+E-F)$. 
Denote $k \leqslant r$ the maximum such that $D_k \in \supp(D+E)$. 
Reasoning as in the proof of Proposition~\ref{prop: coperture C}, 
it follows that $\sigma_i \in \supp_\Sigma(D+E-F)$ for every $k-1 \leqslant i < r$. 
Therefore, if we set $\gamma_0 = \sum_{i=k-1}^{r-1} \sigma_{i} = -D_{k-2} + D_k$ and $F'= D+E-\gamma_0$, 
then $F' \in \mathbb{N}\Delta$ and $F \leqslant_\Sigma F' < D+E$. Set
\[
	(D',E') = \left\{\begin{array}{cc}
			(D - \gamma_0,E) & \text{ if } D_k \in \supp(E)\\
			(D,E - \gamma_0) & \text{ otherwise}
		\end{array}\right.:
\]
then $D' \leqslant_\Sigma D$, $E'\leqslant_\Sigma E$ and $F \leqslant_\Sigma D'+E' <_\Sigma D+E$, hence $(D,E,F)$ is not a low triple.
\end{proof}

\subsection{Type $\mathsf{D}_r$}

\begin{lemma}	\label{lemma: triple fondamentali piatte D}
Let $(D_p,D_q,F)$ be a low fundamental triple with $\supp_S(D_p+D_q-F) = S$. 
Then $p,q,r$ are odd, $p,q\leqslant r-2$ and either
\begin{enumerate}
\item $p+q \leqslant r-1$ with $F = D_{p+q-2}$, or
\item $p+q = r+1$ with $F = D_{r-1} + D_r$.
\end{enumerate}
\end{lemma}

\begin{proof}
Notice that, as in type $\mathsf{A}$, every fundamental triple is low. 
Therefore, we only need to compute the fundamental triples $(D_p,D_q,F)$ with $\supp_S(D_p+D_q-F) = S$. 
Assume $p\leqslant q$.

Take a sequence
\[
	F = F_n <_\Sigma F_{n-1} <_\Sigma\ldots <_\Sigma F_0 = D_p + D_q
\]
such that $F_i \in \mathbb{N}\Delta$ and $\gamma_i = F_{i-1} - F_i$ is a covering difference for every $i \leqslant n$. 
Recall the classification of covering differences of Propositions~\ref{prop: coperture A} and~\ref{prop: coperture D}.

(1) $p,q\leqslant r-2$. If $q$ were equal to $r-1$ then $(r-1)-p$ should be non-zero and odd, thus $\gamma_1=-D_{p-1}+D_p+D_{r-1}-D_r$ and $F_1=D_{p-1}+D_r$, which is impossible since the distance between the vertices $r$ and $p-1$ is even. By symmetry, the same argument works if $q=r$.

(2) $q-p$ is even. If $q-p$ were odd, there would exist $i$ such that $F_i$ is $D_{p'}+2D_{q'}$ with $q'$ equal to $r-1$ or $r$ and $(r-1)-p'$ even, which is impossible. 

(3) $r-p$ is even (as well as $r-q$). Here again if $r-p$ and $r-q$ were odd, there would exist $i$ such that $F_i$ is $D_{p'}+2D_{q'}$ with $q'$ equal to $r-1$ or $r$ and $(r-1)-p'$ even. 

Therefore, there exists $i$ such that $\gamma_j$ is of the type~(2) of Proposition~\ref{prop: coperture A} for every $j<i$ and 
$\gamma_i$ is either of type~(1) of Proposition~\ref{prop: coperture A} or of type $\mathsf{D}$. 

In the first case, $F_{i-2}=D_{p'}+D_{q'}$ with $q'=r-2$ and
$p'+q'=p+q$, $F_{i-1}=D_{p'-2}+2D_{q''}$ where $q''$ is equal to $r-1$
or $r$. Then necessarily $p'-2=1$ and $F_i=F_n=D_{r-1}+D_r$.

In the second case, $F_{i-1}=D_{p'}+D_{q'}$ with $p'+q'=p+q$ and
$F_i=D_{p'-1}+D_{q'-1}$. Then $\gamma_j$ is of the type~(2) of Proposition~\ref{prop: coperture A} 
for every $j>i$ until
$F_{n-1}=D_{p'-1-2k}+D_{q'-1+2k}$ where $p'-1-2k=2$ ($k=n-1-i$) and
$F_n$ is equal to either $D_{q'-1+2k+2}$ if $q'-1+2k\leqslant r-5$ or
$D_{r-1}+D_r$ if $q'-1+2k=r-3$.
\end{proof}

\begin{proposition}
Let $(D,E,F)$ be a low fundamental triple with $\supp_S(D+E-F) = S$,
then $s^{D+E-F}V_F \subset V_D V_E$.
\end{proposition}

\begin{proof}
We need an explicit computation.  

Denote by $U=\mathbb{C}^{2r}$ the first fundamental representation of $G$
(that is, the standard representation of $\SO(2r)$) and fix $b \in \mathsf{S}^2U$
a $G$-invariant non-degenerate symmetric 2-form. If $W \subset U$
is a maximal isotropic subspace, we get then a decomposition $U = W
\oplus W^*$. Fix a non-zero vector $e_0 \in W$ and consider the
corresponding decomposition $U = V \oplus \mathbb{C}\, e_0 \oplus V^* \oplus
\mathbb{C}\,   e_0^*$, where $V \subset W$ is a complement of the line $\mathbb{C}\,   e_0$
and where $e_0^* \in W^*$ is defined by $e_0^*(e_0) = 1$,
$e_0^*\bigr|_V = 0$.

Then (see \cite{Lu3}) the Lie algebra $\mathfrak h_0$ of $H_0$ can be described as
\[
\mathfrak h_0 = \mathfrak{sp}(V) \oplus V^* \oplus \mathrm{Skew}(V,V^*),
\]
where $\mathrm{Skew}(V,V^*) \subset \Hom(V,V^*)$ denotes the subspace 
of skew-symmetric linear maps and where $\mathfrak h_0$ is embedded in 
$\mathfrak{so}(U)$ as follows (here we denote by 
$u=(v,\lambda e_0, \psi, \mu e^*_0)$ a generic element in 
$U = V \oplus \mathbb{C}\,   e_0 \oplus V^* \oplus \mathbb{C}\,   e^*_0$):
\begin{itemize}
	\item[-]	if $f \in \mathfrak{sp}(V)$, then $f(u) = f(v) + f\cdot  \psi$,
	\item[-]	if $\phi \in V^*$, then $\phi(u) = \phi(v)e_0 - (\lambda+\mu)\phi + \phi(v) e_0^*$,
	\item[-]	if $\Phi \in \mathrm{Skew}(V,V^*)$, then $\Phi(u) = \Phi(v)$.
\end{itemize}

As already mentioned at the beginning of Section~3 every simple $G$-module possesses 
a unique $\mathfrak{h}_0$-invariant element, up to scalars. In particular, if we denote 
$h_1 = e_0 - e^*_0 \in U$ and $h_2 \in \mathsf{\Lambda}^2 V^*$ a 
$\mathfrak{sp}(V)$-invariant non-degenerate 2-form, then $h_1 \in U^{\mathfrak h_0}$ and 
$h_2 \in (\mathsf{\Lambda}^2U)^{\mathfrak h_0}$. In this way we may describe the $\mathfrak h_0$-invariant 
vectors in every exterior power $\mathsf{\Lambda}^i U$ with $i \leqslant r-1$. Set indeed 
\[
	h_i = \left\{ \begin{array}{cl}
			h_2^{\wedge k} & \text{ if $i = 2k$ is even} \\
			h_1 \wedge h_2^{\wedge k} & \text{ if $i = 2k+1$ is odd} 
		\end{array} \right.:
\]
then $h_i \in (\mathsf{\Lambda}^i U)^{\mathfrak h_0}$. 

Set $\omega_0 = 0$ and recall that if $i \geqslant 0$ then
\[
	\mathsf{\Lambda}^i U = \left\{ \begin{array}{cl}
						V(\omega_i) & \text{ if $i \leqslant r-2$} \\
						V(\omega_{r-1} + \omega_r) & \text{ if $i = r-1$}
				\end{array} \right.
\]
To conclude the proof, by Lemma~\ref{lemma: supporto
moltiplicazione}, we only need to show that, if $i,j$ are
odd integers with $i+j \leqslant r+1$, then there exists an
equivariant projection $\pi \colon \mathsf{\Lambda}^i U \otimes \mathsf{\Lambda}^j U \longrightarrow
\mathsf{\Lambda}^{i+j-2}U$ such that $\pi(h_i \otimes h_j) \neq 0$. Define
$\pi$ as follows:
\[
	\pi\big((u_1 \wedge\dots \wedge u_i) \otimes (w_1 \wedge\dots \wedge w_j)\big) = \sum_{h,k}(-1)^{h+k} b(u_h,w_k)
	u_1 \wedge\dots \wedge \hat{u_h} \wedge\dots \wedge \hat{w_k} \wedge\dots \wedge w_j.
\]
Suppose that $i,j$ are odd with $i+j \leqslant r+1$ and set $k =
(i+j-2)/2$. Notice that $\pi(h_i \otimes h_j) = b(h_1, h_1)
h_2^{\wedge k} + q$, where $q \in \mathsf{\Lambda}^{i+j-2} U$ is linearly
independent with $h_2^{\wedge k}$: since $b(h_1,h_1) \neq 0$
and since $h_2^{\wedge k} \neq 0$, it follows then
$\pi(h_i\otimes h_j) \neq 0$.
\end{proof}

\subsection{Type $\mathsf{E}_r$} \label{ss: computer}

\begin{lemma}	\label{lemma: triple fondamentali piatte E}
The low fundamental triples $(D,E,F)$ with $\supp_S(D+E-F) = S$ are the following:
\begin{itemize}
\item[-] If $r=6$: $(D_1,D_3,D_2)$ $(D_1,D_5,D_3)$ $(D_1,D_6,0)$ $(D_3,D_6,D_5)$ $(D_5,D_6,D_2)$,
\item[-] If $r=7$: $(D_1,D_6,D_3)$ $(D_6,D_6,D_2+D_7)$,
\item[-] If $r=8$: $(D_1,D_1,D_2)$ $(D_1,D_5,2D_2)$ $(D_1,D_7,D_3)$ $(D_1,D_8,D_7)$ $(D_3,D_8,D_5)$ $(D_5,D_8,D_2+D_7)$ $(D_7,D_8,D_2)$.
\end{itemize}
\end{lemma}

By making use of Proposition~\ref{prop: coperture E}, the proof of this lemma is
a quite long but trivial case-by-case computation, which we do not
report here.

\begin{proposition}
Let $(D,E,F)$ be a low fundamental triple with $\supp_S(D+E-F) = S$, then $s^{D+E-F}V_F \subset V_D V_E$.
\end{proposition}

Let us first deal with the triple $(D_1,D_6,0)$ of $\mathsf{E}_6$. The set
$\Delta_0 = \{D_2,D_3,D_4,D_5\}$ is distinguished and the quotient $M' =
M/\Delta_0$ is a symmetric wonderful variety (with spherical roots
$\{2\alpha_1+\alpha_2+2\alpha_3+2\alpha_4+\alpha_5,\alpha_2+\alpha_3+2\alpha_4+2\alpha_5+2\alpha_6$). Therefore,
we can conclude by Corollary~\ref{cor: sezioni e quozienti} and the
surjectivity of the multiplication map in the symmetric case.

For all the other triples of Lemma~\ref{lemma: triple fondamentali piatte E}
we have to use Lemma~\ref{lemma: supporto moltiplicazione}.
Since the dimension of the involved
representations is quite high, we have used the computer and, more
precisely, GAP \cite{gap}, a software for computations which contains
built-in functions to construct and deal with representations of
simple Lie algebras (see also \cite{dG}).

Although the dimension of some of the involved representations is very
high, we have succeded to make the computation accessible with a
currently available home computer. A quite convenient tool is the
quadratic Casimir operator $c$, which acts as the scalar
$(\lambda+2\rho,\lambda)$ on every irreducible representation
$V(\lambda)$. Let $(D,E,F)$ be a low fundamental triple, once the
irreducible representations $V_D^*$ and $V_E^*$ are constructed, and
the vectors $h_D$ and $h_E$ are explicitly found, there is no need to construct
the whole tensor product $V_D^*\otimes V_E^*$ 
in order to project $h_D\otimes h_E$ onto the eigenspace of $c$ relative
to $(\omega_F^*+2\rho,\omega_F^*)$.  
Moreover, in all our cases there exists only one isotypical component in $V_D^*\otimes V_E^*$
of eigenvalue $(\omega_F^*+2\rho,\omega_F^*)$ with respect to $c$.

Sometimes for $\mathsf{E}_8$ the dimension is so high that it is even quite
costly to construct the irreducible representation $V_D^*$ itself
($V_{D_5}$ has dimension 146,325,270). Here we use a further
escamotage. The set $\Delta_0 = \{D_1,D_4,D_6,D_8\}$ is distinguished
and the quotient $M' = M/\Delta_0$ is the parabolic induction of a
symmetric wonderful $\SL(8)$-variety with spherical roots
$\{\alpha_1+2\alpha_3+\alpha_4,\alpha_4+2\alpha_5+\alpha_6,\alpha_6+2\alpha_7+\alpha_8\}$. This
means that, if $D$ equals $D_3$, $D_5$ or $D_7$, then $h_D$ can be
choosen to be the $\Sp(8)$-invariant vector in the simple
$\SL(8)$-submodule $W_D\subset V_D$ of highest weight
$\omega_D$. Furthermore, $V_{D_1}$ and $V_{D_8}$ are still accessible
(the former has dimension 3875 and the latter is the adjoint
representation), $V_{D_7}\subset\mathsf{\Lambda}^2 V_{D_8}$,
$V_{D_5}\subset\mathsf{\Lambda}^4 V_{D_8}$, $V_{D_3}\subset\mathsf{\Lambda}^2V_{D_1}$,
and respectively $W_{D_7}\subset\mathsf{\Lambda}^2W_{D_8}$,
$W_{D_5}\subset\mathsf{\Lambda}^4W_{D_8}$,
$W_{D_3}\subset\mathsf{\Lambda}^2W_{D_1}$. Therefore, the vectors $h_D$ can be
explicitly determined if $D$ equals $D_1$, $D_3$, $D_5$, $D_7$ or
$D_8$, and notice that this is enough to treat all the above triples.

\subsection{Type $\mathsf{F}_4$} 

\begin{lemma}	\label{lemma: triple fondamentali piatte F}
The only low fundamental triple $(D,E,F)$ with $\supp_S(D+E-F) = S$ is $(D_1,D_4,D_3)$.
\end{lemma}

The proof is trivial, after Proposition~\ref{prop: coperture F}.

\begin{proposition}
$s^{D_1+D_4-D_3}V_{D_3} \subset V_{D_1} V_{D_4}$.
\end{proposition}

This can easily be checked via computer (as explained above).

\subsection{Type $\mathsf{G}_2$}

\begin{proposition}
There are no low triples $(D,E,F)$ with $F \neq D+E$.
\end{proposition}

The proof is trivial.

\subsection{Projective normality of model wonderful varieties}\label{ssec:proofA}

We are now ready to prove that the multiplication of sections
on a model wonderful variety of simply connected type is surjective.

A \textit{localization} of a wonderful variety $M$ is a $G$-stable subvariety of $M$, 
which is a wonderful $G$-variety by itself.
Notice that we have a bijective correspondence between localizations of $M$ and subsets of $\Sigma$. 
More precisely, for all subsets $\Sigma'$ of $\Sigma$, 
the intersection of the boundary divisors $M^\sigma$ for $\sigma\in\Sigma'$ gives a wonderful variety
with $\Sigma\smallsetminus\Sigma'$ as set of spherical roots.

\begin{proof}[Proof of Theorem~\ref{teo: model}]
By the classification of the covering differences given in the previous section,
it follows that the model wonderful varieties of simply connected type
satisfy the property (2-ht). Hence by Lemma~\ref{lem: proiettiva normalita}
it is enough to show the inclusion $s^{D+E-F} V_F \subset V_D V_E$
for all the low fundamental triples $(D,E,F)$.
Here we only need to show that we can reduce to the low fundamental triples $(D,E,F)$ such that $\supp_S(D+E-F) = S$,
whose case has been proved above in this section.

Let $(D,E,F)$ be a low fundamental triple, denote $\gamma = D+E-F$
and $S' = \supp_S(\gamma)$ and suppose that $S' \neq S$.
Let $Q$ be the parabolic subgroup associated to $S'$
and set $G' = Q/R_Q$. Since $(D,E,F)$ is a low triple,
it follows that $S'$ is connected, hence $G'$ is an almost simple group.
Consider the localization $M'$ of $M$ with $\Sigma'=\{\sigma\in\Sigma\, : \,\supp_S(\sigma)\subset S'\}$
as set of spherical roots.
Then $M'$ is the parabolic induction of the model wonderful $G'$-variety $N$ of simply connected type. 

The restrictions $\Pic(M) \longrightarrow \Pic(M')$ and $\Pic(M') \longrightarrow \Pic(N)$
identify $\Pic(N)$ with a sublattice of $\Pic(M)$.
More precisely, if $\widetilde{\Delta} \subset \Delta$ is the set of colors $\widetilde{D}$ such that
$\omega_{\widetilde{D}} = \omega_\alpha$ for some $\alpha \in S'$,
then $\Pic(N)$ is identified with $\mathbb{Z}\widetilde{\Delta}$. Moreover,
an element of $\mathbb{Z}\widetilde{\Delta}$ induces a globally generated line bundle on $N$
if and only if its coefficients are non negative.
If $\widetilde{D} \in \mathbb{N}\widetilde{\Delta}$, set $W_{\widetilde{D}} = V_{\widetilde{D}}^{U_Q}$.

Notice that $D,E \in \widetilde{\Delta}$ and consider the difference
$\widetilde F = D+E-\gamma \in \mathbb{Z} \widetilde{\Delta}$ (where $\gamma$ is regarded as sum of spherical roots of $N$
hence as an element of $\mathbb{Z} \widetilde{\Delta}$). Since $(D,E,F)$ is a low triple in $\mathbb{N}\Delta$, it follows that
$(D,E,\widetilde F)$ is a low triple in $\mathbb{N}\widetilde{\Delta}$.
Since $\supp_{S'}(\gamma) = S'$, we have that ${s^\gamma} W_{\widetilde F} \subset W_D W_E$, therefore
Proposition~\ref{prop: ind-par} implies that
$s^\gamma V_F \subset V_D V_E$ with respect to the multiplication map
$m'_{D,E}\colon\Gamma(M',\mathcal L_D) \otimes \Gamma(M',\mathcal L_E) \longrightarrow \Gamma(M',\mathcal L_{D+E})$.
This concludes the proof, since $m'_{D,E}$ is just the restriction of the multiplication
map $m_{D,E}\colon\Gamma(M,\mathcal L_D) \otimes \Gamma(M,\mathcal L_E) \longrightarrow \Gamma(M,\mathcal L_{D+E})$.
\end{proof}

Let $M$ be a wonderful variety and let $N$ be a quotient of $M$.
Then the pull-back of line bundles identifies $\Pic(N)$ with a sublattice of $\Pic(M)$
and a line bundle $\mathcal L \in \Pic(N)$ is generated by global sections if and only
if its pull-back (which we still denote by $\mathcal L$) is. 
Moreover, by Corollary~\ref{cor: sezioni e quozienti} we have $\Gamma(N,\mathcal L) = \Gamma(M,\mathcal L)$.
It follows that if $\mathcal L, \mathcal L' \in \Pic(N)$ are generated by global sections
and if the multiplication $\Gamma(M,\mathcal L) \otimes \Gamma(M,\mathcal L')
\longrightarrow \Gamma(M,\mathcal L\otimes \mathcal L')$ is surjective, then the multiplication
$\Gamma(N, \mathcal L) \otimes \Gamma(N, \mathcal L') \longrightarrow \Gamma(N,\mathcal L\otimes\mathcal L')$
is also surjective.

Let now $N$ be a localization of $M$ and let $\mathcal L, \mathcal L' \in \Pic(M)$
be generated by global sections. Then the restriction of sections to $N$ is surjective,
therefore if the multiplication $\Gamma(M,\mathcal L) \otimes \Gamma(M,\mathcal L')
\longrightarrow \Gamma(M,\mathcal L\otimes \mathcal L')$ is surjective,
the multiplication $\Gamma(N, \mathcal L\bigr|_N) \otimes \Gamma(N, \mathcal L'\bigr|_N)
\longrightarrow \Gamma(N,\mathcal L\bigr|_N \otimes \mathcal L'\bigr|_N)$ is also surjective.

If moreover $M$ is a model wonderful variety of simply connected type, 
then by the description of the restriction $\omega\colon \Pic(M ) \longrightarrow \mathcal{X}(T)$ (see \cite[Lemma~30.24]{Ti}) 
we get an isomorphism $\Pic(M) \simeq \mathcal{X}(T)$, 
which identifies the semigroup of globally generated (resp.\ ample) line bundles on $M$ 
with $\mathcal{X}(T)^+$ (resp.\ with the semigroup of regular dominant weights). 
If $N$ is any localization of $M$, then $\omega$ factors through $N$ and we get an isomorphism $\Pic(M) \simeq \Pic(N)$, 
which identifies the globally generated (resp.\ the ample) line bundles on $M$ and on $N$. 
Up to replacing $\mathcal{X}(T)$ with some sublattice generated by fundametal weights, 
the same is true whenever $M'$ is the quotient of a localization of a model wonderful variety of simply connected type 
and $N'$ is a localization of $M'$. 
Therefore, proceeding inductively, we get the following corollary of Theorem~\ref{teo: model}.

\begin{corollary}	\label{cor:localizzazione}
Let $N$ be a wonderful variety obtained from a model wonderful variety
of simply connected type via operations of localization
and quotient by distinguished sets of colors.
Then the multiplication map
\[ m_{\mathcal L,\mathcal L'}\colon\Gamma(N,\mathcal L) \otimes \Gamma(N,\mathcal L')\longrightarrow \Gamma(N,\mathcal L\otimes \mathcal L') \] 
is surjective for all globally generated line bundles $\mathcal L,\mathcal L' \in \Pic(N)$.
\end{corollary}

\section{Projective normality of comodel wonderful varieties}\label{sec:comodello}

Motivated by an application which we illustrate below in Section~\ref{sec:modelloreale}, 
here we study the surjectivity of the multiplication of sections in another class of wonderful varieties, 
which we call {\it comodel}. 

Let $G$ be simply connected of simply-laced type, and fix $T$ and $B$ as usual. 
Let $M$ be a model wonderful variety of $G$, 
with set of colors $\Delta$ (in bijection with the set of fundamental weights), 
set of spherical roots $\Sigma$ and Cartan pairing $c\colon \Delta\times \Sigma\to\mathbb Z$.
Let $G^\vee$ be a simply connected group whose root system is isomorphic to $\Phi_\Sigma$, 
the root system generated by $\Sigma$. 
Once fixed $T^\vee$ and $B^\vee$, 
its set of simple roots $S^\vee$ is thus in bijective correspondence with $\Sigma$.
From \cite{Bra} and \cite{BP} it follows that

\begin{theorem}	\label{teo: comodello-esistenza}
There exists a wonderful $G^\vee$-variety $M^\vee$   
whose set of spherical roots $\Sigma^\vee$ is equal
to the set of simple roots $S^\vee$ of $G^\vee$,
its set of colors $\Delta^\vee$ is in bijective correspondence with $\Delta$  
and, under these correspondences, its Cartan pairing
$c^\vee\colon\Delta^\vee\times\Sigma^\vee\to\mathbb Z$ 
equals the Cartan pairing $c$ of the model wonderful $G$-variety $M$.
\end{theorem}

We say that the wonderful variety $M^\vee$ of the previous theorem
is a \textit{comodel wonderful variety} of $G^\vee$, and the type of $G$
is called the {\it cotype} of $M^\vee$.

Forgetting the model wonderful variety, 
in the following $M$ is a comodel wonderful variety of $G$, 
$H$ denotes the stabilizer of a point $x_0$ in the open orbit of $M$,
and $H_0$ the kernel of its multiplicative characters; 
$\mathfrak{h}_0$ denotes the Lie algebra of $H_0$ in the Lie algebra 
$\mathfrak{g}$ of $G$. 

The comodel wonderful varieties correspond to the following cases in \cite{Bra},
an explicit description of the corresponding subgroups $H$ can be found in \cite{BP}.
\begin{itemize}
\item[-] Cotype $\mathsf A_{2m}$: type $\mathsf A_{m-1}\times\mathsf A_m$, case S-5.
\item[-] Cotype $\mathsf A_{2m+1}$: type $\mathsf A_m\times\mathsf A_m$, case S-4.
\item[-] Cotype $\mathsf D_{2m}$: type $\mathsf A_{m-1}\times\mathsf D_m$, case S-10.
\item[-] Cotype $\mathsf D_{2m+1}$: type $\mathsf A_{m-1}\times\mathsf D_{m+1}$, case S-11.
\item[-] Cotype $\mathsf E_6$: type $\mathsf A_5$, case S-50.
\item[-] Cotype $\mathsf E_7$: type $\mathsf A_6$, case S-49.
\item[-] Cotype $\mathsf E_8$: type $\mathsf D_7$, case S-58.
\end{itemize}

In this section we prove the following.

\begin{theorem}\label{teo: comodello-proj-norm}
Let $M$ be a comodel wonderful variety 
and let $\mathcal L$, $\mathcal L'$ be line bundles
generated by global sections. 
Then the multiplication map $m_{\mathcal L,\mathcal L'}$ is surjective.
\end{theorem}

Let $M$ be a comodel wonderful variety, we enumerate the colors
of $M$ as in the corresponding model wonderful variety. 
Similarly, we denote by $h_i\in V_{D_i}^*$ the $H$-semiinvariant associated to
$D_i$. This vector is invariant under $H_0$.

More explicitly, set the map $\omega \colon \Pic(M) \longrightarrow \mathcal{X}(T)$ as follows
(if the Dynkin diagram of $G$ has two connected components,
we distinguish with a superscript the fundamental weights coming
from the different components).
\begin{itemize}
\item[-] Cotype $\mathsf A_{2m}$: $\omega(D_1)=\omega'_1$, 
$\omega(D_{2i})=\omega_i+\omega'_i$ for $i<m$,
$\omega(D_{2i-1})=\omega_{i-1}+\omega'_i$ for $i>1$,
$\omega(D_{2m})=\omega'_m$.
\item[-] Cotype $\mathsf A_{2m+1}$: $\omega(D_1)=\omega_1$, 
$\omega(D_{2i})=\omega_i+\omega'_i$,
$\omega(D_{2i-1})=\omega_i+\omega'_{i-1}$ for $1<i<m+1$,
$\omega(D_{2m+1})=\omega'_m$.
\item[-] Cotype $\mathsf D_{2m}$: $\omega(D_1)=\omega_1$, 
$\omega(D_{2i})=\omega_i+\omega'_i$ for $i<m-1$,
$\omega(D_{2i-1})=\omega_i+\omega'_{i-1}$ for $1<i<m$,
$\omega(D_{2m-2})=\omega_{m-1}+\omega'_{m-1}+\omega'_m$,
$\omega(D_{2m-1})=\omega'_{m-1}$, $\omega(D_{2m})=\omega'_m$.
\item[-] Cotype $\mathsf D_{2m+1}$: $\omega(D_1)=\omega'_1$, 
$\omega(D_{2i})=\omega_i+\omega'_i$ for $i<m$,
$\omega(D_{2i-1})=\omega_{i-1}+\omega'_i$ for $1<i<m$,
$\omega(D_{2m-1})=\omega_{m-1}+\omega'_m+\omega'_{m+1}$,
$\omega(D_{2m})=\omega'_m$, $\omega(D_{2m+1})=\omega'_{m+1}$.
\item[-] Cotype $\mathsf E_6$: $\omega(D_1)=\omega_2$, 
$\omega(D_2)=\omega_3$, $\omega(D_3)=\omega_2+\omega_5$,
$\omega(D_4)=\omega_1+\omega_3+\omega_5$,
$\omega(D_5)=\omega_1+\omega_4$, $\omega(D_6)=\omega_4$.
\item[-] Cotype $\mathsf E_7$: $\omega(D_1)=\omega_3$, 
$\omega(D_2)=\omega_4$, $\omega(D_3)=\omega_3+\omega_6$,
$\omega(D_4)=\omega_2+\omega_4+\omega_6$,
$\omega(D_5)=\omega_2+\omega_5$, $\omega(D_6)=\omega_1+\omega_5$,
$\omega(D_7)=\omega_1$.
\item[-] Cotype $\mathsf E_8$: $\omega(D_1)=\omega_3$, 
$\omega(D_2)=\omega_4$, $\omega(D_3)=\omega_3+\omega_6$,
$\omega(D_4)=\omega_2+\omega_4+\omega_6$,
$\omega(D_5)=\omega_2+\omega_5$, $\omega(D_6)=\omega_1+\omega_5$,
$\omega(D_7)=\omega_1+\omega_7$, $\omega(D_8)=\omega_7$.
\end{itemize}

Since the Cartan pairing of the comodel wonderful varieties is the same as 
that of the model varieties, the classification of the covering differences is also the same 
and the property (2-ht) holds. As shown in Section \ref{ssec: proof of teo comodello}, 
in order to apply Lemma~\ref{lem: proiettiva normalita} 
it is enough to test the surjectivity on the same low fundamental triples
arising in Section~\ref{sec: low model} in the model case.  

In the computations below we use the following conventions. We denote
by $e_1,\dots,e_n$ the standard basis of $V=\mathbb{C}^n$ and by
$\varphi_1,\dots,\varphi_n$ the dual basis. We also denote by
$e_{i_1i_2\ldots i_k}$ the vector 
$e_{i_1}\wedge e_{i_2}\wedge\dots \wedge e_{i_k} \in \mathsf{\Lambda}^kV$ and similarly for $\varphi_{i_1i_2\ldots i_k}$.  
On $V\oplus V^*$ is defined the symmetric bilinear form
$(u,\varphi)\cdot(v,\psi)=\varphi(v)+\psi(u)$. A pairing between $\mathsf{\Lambda}^k V$
and $\mathsf{\Lambda}^{n-k} V$ is defined by $\langle x,y\rangle=\frac{x\wedge
  y}{e_{1\ldots n}}$ and we denote by $\gamma_k\colon\mathsf{\Lambda}^k V\longrightarrow
\mathsf{\Lambda}^{n-k}V^*$ the associated map. We also identify $V_1\otimes V_2$ with
$\Hom(V_1^*,V_2)$ in the usual way, and if $V_1=V_2=V$ we identify
$\mathsf S^2 V$ and $\mathsf{\Lambda}^2 V$ with symmetric and antisymmetric linear maps
in $\Hom(V^*,V)$.

Starting with the vector space $V$ we construct a new vector
space $W=V\oplus V^*\oplus Z$ as the direct sum of $V$, $V^*$ and 
another piece $Z$ that is zero, one or two dimensional. In
particular the dual of $W$ can be identified with $V^*\oplus V\oplus
Z^*$.  Once a basis of $Z$ is fixed, say $z_1,\dots,z_t$, we 
denote the dual basis of
$e_1,\dots,e_n,\varphi_1,\dots,\varphi_n,z_1,\dots,z_t$ by
$e^*_1,\dots,e^*_n,\varphi^*_1,\dots,\varphi^*_n,z^*_1,\dots,z^*_t$.

For a vector space $U$ we have contraction maps 
$\kappa^{ij}_U\colon \mathsf{\Lambda}^i U\otimes \mathsf{\Lambda}^j U^* \longrightarrow 
\mathsf{\Lambda}^{i-1}U \otimes \mathsf{\Lambda}^{j-1}U^*$ given by 
\begin{align*}
\kappa^{ij}_U & (
u_1\wedge\dots \wedge u_i \otimes 
\varphi_1\wedge\dots\wedge \varphi_j) = \\
& \sum_{k,l}(-1)^{k+l} \varphi_l(u_k) 
u_1\wedge\dots \wedge\hat u_k\wedge\dots \wedge u_i \otimes
\varphi_1\wedge\dots \wedge\hat \varphi_l\wedge\dots \wedge \varphi_j.
\end{align*} 
In particular we set $\kappa_U=\kappa^{2\,1}_U$.

Similarly, for a vector space $U$ with a symmetric bilinear form $(\,,\,)$ 
we have
$\tilde\kappa^{ij}_U\colon \mathsf{\Lambda}^i U\otimes \mathsf{\Lambda}^j U \longrightarrow 
\mathsf{\Lambda}^{i+j-2}U$ such that 
\begin{align*}
\tilde\kappa^{ij}_U & (
u_1\wedge\dots\wedge u_i \otimes 
v_1\wedge\dots\wedge v_j) = \\
& \sum_{k,l}(-1)^{k+l} (u_k,v_l) 
u_1\wedge\dots \wedge\hat u_k\wedge\dots \wedge u_i \wedge
v_1\wedge\dots \wedge\hat v_l\wedge\dots \wedge v_j.
\end{align*} 

\subsection{Cotype $\mathsf{A}_r$}

We test the triples of Lemma~\ref{lemma: triple fondamentali piatte A}:
$(D_p,D_{r+1-p},0)$ where, if $r$ is odd, $p$ is even.

If $r$ is odd, the set $\Delta_{\mathrm{odd}}$ of odd-indexed colors is distinguished 
and the quotient is a symmetric wonderful variety, so the surjectivity follows as in 
Proposition~\ref{prp:triplemodelloAdispari}.
 
If $r$ is even, $H$ is reductive and the surjectivity follows as in 
Proposition~\ref{prp:triplemodelloApari}.

\subsection{Cotype $\mathsf{D}_r$}

We test the triples of Lemma~\ref{lemma: triple fondamentali piatte D}:
$(D_p,D_q,F)$ where $p,q,r$ are odd, $p,q\leqslant r-2$ and either
\begin{itemize}
\item $p+q \leqslant r-1$ with $F = D_{p+q-2}$, or
\item $p+q = r+1$ with $F = D_{r-1} + D_r$.
\end{itemize}

Let $m>1$ and set $r=2m+1$. 
Let $V =\mathbb{C}^m$, and set 
$W=V\oplus \mathbb C\,e\oplus V^*\oplus\mathbb C\,\varepsilon$.
On $W$ is defined a quadratic form such that $V$ and $V^*$ are orthogonal
to $e$ and $\varepsilon$, moreover $(e,\varepsilon)=1$ and $(e,e)=(\varepsilon,\varepsilon)=0$. 
Let $G=\SL (V)\times \Spin(W)$. 
We have $\mathfrak{h}_0=\mathfrak{sl}(V)\oplus \mathsf{\Lambda}^2 V\oplus V$, 
which is embedded as a subalgebra of $\mathfrak{g}$ as follows.
First notice that we have a natural immersion of $\SL (V)$ in $\Spin (W)$,
so that $\mathfrak{sl}(V)$ can be included diagonally in $\mathfrak{g}$. 
The action of $\mathsf{\Lambda}^2 V$ on $W$ is zero on $V$,
$e$ and $\varepsilon$, while its action on $V^*$ is given by the
identification of $\mathsf{\Lambda}^2 V$ with the antisymmetric maps from $V^*$ to
$V$. The action of $V$ on $W$ given as follows: if $v,u\in V$ and 
$\varphi\in V^*$ then 
\[
v\cdot  u =0 \qquad v\cdot  e =v 
\qquad v\cdot  \varphi = -\varphi(v)(e+\varepsilon) \qquad v\cdot  \varepsilon = v.
\]  

Let now $\tilde D_i=  D_i$ for $i$ odd or for $i\neq 2m$ while let $\tilde D_{2m}=D_{2m}+D_{2m+1}$.
We have 
\[
V_{\tilde D_{2i}}^*= \mathsf{\Lambda}^i V^* \otimes \mathsf{\Lambda}^i W \qquad 
V_{\tilde D_{2i+1}}^*= \mathsf{\Lambda}^i V^* \otimes \mathsf{\Lambda}^{i+1} W .
\]
The $\mathfrak{h}_0$-invariants in $V_{\tilde D_{2i}}^*$ are generated by the vector 
$h_{2i}$ corresponding to the identity element in $\mathsf{\Lambda}^i V^* \otimes \mathsf{\Lambda}^i V\subset 
\mathsf{\Lambda}^i V^* \otimes \mathsf{\Lambda}^i W$. Similarly 
the $\mathfrak{h}_0$-invariants in $V_{\tilde D_{2i+1}}^*$ are generated by the vector 
$h_{2i+1}=h_{2i}\wedge {(e-\varepsilon)}$. So that we have
\[
h_{2i}=\sum_{j_1<\ldots<j_i} \varphi_{j_1\ldots j_i} \otimes e_{j_1\ldots j_i}
\quad\text{ and }\quad
h_{2i+1}=\sum_{j_1<\ldots<j_i}  \varphi_{j_1\ldots j_i} \otimes (e_{j_1\ldots j_i}\wedge (e-\varepsilon)).
\]
Finally if $p=2t+1$ and $q=2s+1$ then the projection 
$ \Phi\colon V_{\tilde D_{p}}^*\otimes V_{\tilde D_{q}}^*\longrightarrow V_{\tilde D_{p+q-2}}^* $
is given by
\[\Phi \left(
(x\otimes w) \otimes (y\otimes z)
\right) = 
x\wedge y \otimes \tilde\kappa^{t+1\,s+1}_W(w\otimes z)
\]
for $x\in \mathsf{\Lambda}^t V^*$, $y\in \mathsf{\Lambda}^s V^*$, 
$w\in\mathsf{\Lambda}^{t+1}W$ and $z\in\mathsf{\Lambda}^{s+1}W$.
A direct computation shows that 
\[
\Phi(h_p\otimes h_q)=(-1)^{t+s+1}2 \binom{t+s}{t} h_{p+q-2}.
\]

\subsection{Cotype $\mathsf{E}_6$}

Let $V=\mathbb{C}^3$ and $W=V\oplus V^*$. Let $\mathfrak{g} = \mathfrak{sl}(W)$ and
$\mathfrak{h}_0 = \mathfrak{sl}(V)\oplus \mathsf{S}^2 V$ where the action of
$\mathfrak{sl}(V)$ is the natural one and the action of $\mathsf{S}^2V$ is given
by $b\cdot (v,\varphi)= (b(\varphi),0)$ while the action on $W^*=V^*\oplus V$
is given by $b\cdot (\varphi,v)= (0,-b(\varphi))$. 

We analyze the triples of Lemma~\ref{lemma: triple fondamentali piatte E}: 
$(D_{1},D_{3},D_{2})$,
$(D_{1},D_{5},D_{3})$, $(D_{1},D_{6},0)$, $(D_{3},D_{6},D_{5})$,
$(D_{5},D_{6},D_{2})$. 

The set of colors $\{D_2,D_3,D_4,D_5\}$ is distinguished and 
the associated quotient is a symmetric wonderful variety, 
so the third triple follows as in 
Proposition~\ref{prp:triplemodelloAdispari}.
Therefore, by symmetry it is enough to analyze the last
two triples. We need to compute $h_3$, $h_5$, $h_6$.

We have $V_{D_6}^*=\mathsf{\Lambda}^2 W$. 
Looking at the action of $\mathfrak{sl}(V)$ we find only 
one invariant in $V\otimes V^*\subset \mathsf{\Lambda}^2 W$ 
corresponding to the identity in $\End(V^*)$.
Hence we get $h_6=e_1\wedge \varphi_1+e_2\wedge \varphi_2+e_3\wedge \varphi_3$.

The representation $V_{D_5}^*$ is contained in $\mathsf{\Lambda}^2 W \otimes W^*$
and it is the kernel of $\kappa_W$. 
It contains two invariants under the action of
$\mathfrak{sl}(V)$: $x \in \mathsf{\Lambda}^2 V\otimes V\simeq \End(V)$ and $y\in
\mathsf{\Lambda}^2 V^* \otimes V^*\simeq \End(V^*)$. Notice that $\mathsf{\Lambda}^2 W
\otimes W^*\simeq V_{D_5}^*\oplus W$ so both $x,y \in
V_{D_5}^*$. Moreover the action of $\mathsf{S}^2V$ on $x$ is clearly trivial
while it is not on $y$. So we have
$ h_5 = e_{12} \otimes \varphi^*_3 - e_{13} \otimes \varphi^*_2 + e_{23} \otimes \varphi^*_1 $.

Similarly we have $V_{D_3}^*=\ker\kappa_{W^*}\subset \mathsf{\Lambda}^2 W^* \otimes
W$ and $ h_3 = \varphi^*_{12} \otimes e_3 - \varphi^*_{13} \otimes e_2 +
\varphi^*_{23} \otimes e_1 $.

\subsubsection{Analysis of the triple $(D_{5},D_{6},D_{2})$.} 

We have $V_{D_2}^*=\mathsf{\Lambda}^3 W$.
Consider the map $\Phi\colon\mathsf{\Lambda}^2 W \otimes \mathsf{\Lambda}^2 W \otimes W^* \longrightarrow \mathsf{\Lambda}^3 W$ given by
\[
\Phi (x\otimes y \otimes \varphi) = \kappa_W(x\otimes \varphi) \wedge y
\] for $x,y\in \mathsf{\Lambda}^2W$ and $\varphi\in W^*$. 
A direct computation shows that $\Phi(h_6\otimes h_5)=3e_{123} \neq 0$.

\subsubsection{Analysis of the triple $(D_{3},D_{6},D_{5})$.} 

Consider the map $\Phi\colon\mathsf{\Lambda}^2 W \otimes
\mathsf{\Lambda}^2 W^* \otimes W \longrightarrow \mathsf{\Lambda}^2W \otimes W^*$ given by
\[ 
\Phi \big((x\wedge y) \otimes (\varphi\wedge \psi) \otimes
w\big) = 
 (\kappa_W(x\wedge y\otimes\varphi)\wedge w) \otimes \psi 
-(\kappa_W(x\wedge y\otimes\psi)\wedge w) \otimes \varphi
\] 
for $x,y,w\in W$ and $\varphi,\psi\in W^*$. 
A direct computation shows that $\Phi(h_6\otimes h_3)=2\,h_5$.

\subsection{Cotype $\mathsf{E}_7$}

Let $V=\mathbb{C}^3$ and set $W=V\oplus V^*\oplus \mathbb{C}\,e$,
$\mathfrak{h}_0 = \mathfrak{sl} (V) \oplus \mathsf{S}^2 V \oplus \mathsf{\Lambda}^2 V \oplus V$ and
$\mathfrak{g}=\mathfrak{sl}(W)$. Recall that we identify $\mathsf{S}^2 V \oplus \mathsf{\Lambda}^2 V$
with the decomposition of $\Hom(V^*,V)$ into symmetric and
antisymmetric matrices. We have an action of $\mathfrak{h}_0$ on $W$ given as
follows. Let $(a,b,\omega,u)\in \mathfrak{h}_0$ and $(v,\varphi,\lambda e)\in W$
then
\begin{align*}
a\cdot v &= a(v) & b\cdot v &= 0 & \omega\cdot v &= 0 & u\cdot v &= 0 \\ 
a\cdot \varphi &= -a^t(\varphi) & b\cdot \varphi &= b(\varphi) & \omega\cdot \varphi &= \omega(\varphi) 
& u\cdot \varphi &= 0 \\ 
a\cdot e &= 0 & b\cdot e &= 0 & \omega\cdot e &= \gamma(\omega) & u\cdot e &= u
\end{align*}
where $\gamma=\gamma_2\colon\mathsf{\Lambda}^2V \longrightarrow V^*$ is defined as above.
This defines an immersion of $\mathfrak{h}_0$ into $\mathfrak{sl}(W)$ whose image
is closed under the Lie bracket. 
Indeed, if $(a,b,\omega,u) \in \mathfrak{h}_0$ and $(a',b',\omega',u')\in \mathfrak{h}_0$ 
then:
\begin{align*}
 [a,a']&=aa'-a'a; & 
 [a,b] & = ab +b a^t; & 
 [a,\omega] &= a\omega+\omega a^t; & 
 [a,u] & = a(u); & 
 [b,b']& = 0; \\
 [b,\omega] & = b(\gamma(\omega)); & 
 [b,u] & = 0; & 
 [\omega,\omega'] & = 0; &
 [\omega,u] & = 0; &
 [u,u'] & = 0.
\end{align*}

It is also useful to write the action on $W^*$ which is given as follows.
We have $W^*=V^*\oplus V \oplus \mathbb{C}\,e^*$, then
\begin{align*}
a\cdot \varphi &= -a^t(\varphi) & b\cdot \varphi &= -b(\varphi) &  
\omega\cdot \varphi &= \omega(\varphi) & u\cdot \varphi &= -\varphi(u) e^* \\
a\cdot v &= a(v) & b\cdot v &= 0 & \omega\cdot v &= -\langle\omega,v\rangle e^* 
& u\cdot v &= 0 \\
a\cdot e^* &= 0 & b\cdot e^* &= 0 & \omega\cdot e^* &= \gamma(\omega) & u\cdot e^* &= u 
\end{align*}

The triples of Lemma~\ref{lemma: triple fondamentali piatte E} 
are $(D_1,D_6,D_3)$ and $(D_6,D_6,D_2+D_7)$.

\subsubsection{Computation of $h_1$.}

The representation associated to $D_1$ is 
$\mathsf{\Lambda}^4 W$. Looking at the action of $\mathfrak{sl}(V)$ we get three
invariants: the vector $x\in \mathsf{\Lambda}^2 V \otimes \mathsf{\Lambda}^2 V^*$
corresponding to the identity in $\End(\mathsf{\Lambda}^2 V)$, the vector
$y=e_{123}\wedge e \in \mathsf{\Lambda}^3 V \otimes \mathbb{C}\,e$ and the vector 
$z=\varphi_{123}\wedge e \in \mathsf{\Lambda}^3 V^* \otimes \mathbb{C}\, e$. Hence the
invariant is a linear combination of these vectors.  A small
computation shows that
\[
h_1=2y-x= 2e_{123}\wedge e - 
e_{12}\wedge \varphi_{12} -
e_{13}\wedge \varphi_{13} -
e_{23}\wedge \varphi_{23}.
\]

\subsubsection{Computation of $h_6$.}

The representation associated to $D_6$ is the kernel of the map
$\kappa_W$ in $\mathsf{\Lambda}^2 W \otimes W^*$. In $\mathsf{\Lambda}^2 W \otimes W^*$
there are five invariant vectors under
$\mathfrak{sl}(V)$:  $x_1\in \mathsf{\Lambda}^2 V \otimes V$, $x_2\in \mathsf{\Lambda}^2 V^*
\otimes V^*$, $y_1\in \big(V\wedge \mathbb{C}\,e\big) \otimes V^*$, $y_2\in
\big(V^*\wedge \mathbb{C}\,e\big) \otimes V$ and $z\in \big(V\wedge
V^*\big)\otimes \mathbb{C}\,e^*$.  
A small computation shows that 
\[
h_6 = 2x_1+z=
2\,e_{12}\otimes \varphi_3^* -
2\,e_{13}\otimes \varphi_2^* +
2\,e_{23}\otimes \varphi_1^* +
(e_1 \wedge \varphi_1) \otimes e^* +
(e_2 \wedge \varphi_2) \otimes e^* +
(e_3 \wedge \varphi_3) \otimes e^* .
\]

\subsubsection{Analysis of the triple $(D_1,D_6,D_3)$.}

The representation associated to the color $D_3$ is the kernel of the
wedge product in $W \otimes \mathsf{\Lambda}^4 W \longrightarrow \mathsf{\Lambda}^5 W$. 
We consider the map 
$\Phi\colon \mathsf{\Lambda}^4 W \otimes \mathsf{\Lambda}^2 W \otimes W^* \longrightarrow W \otimes \mathsf{\Lambda}^4 W$
given by
\[
\Phi\big(u\otimes (v_1\wedge v_2) \otimes \varphi)  = 
 v_1\otimes (\kappa^{4\,1}_W(u\otimes \varphi)\wedge v_2) 
-v_2\otimes (\kappa^{4\,1}_W(u\otimes \varphi)\wedge v_1).
\]
A direct computation shows that 
\[
\Phi(h_1\otimes h_6)= -4 \big(
e_1 \otimes (e_{123}\wedge \varphi_1) +
e_2 \otimes (e_{123}\wedge \varphi_2) +
e_3 \otimes (e_{123}\wedge \varphi_3) 
\big)\neq 0.
\]

\subsubsection{Analysis of the triple $(D_6,D_6,D_2+D_7)$.}

The representation $V_{D_2+D_7}^*$ associated to the color $D_2+D_7$ is
the kernel of the map $\kappa^{3\,1}_W\colon\mathsf{\Lambda}^3 W \otimes W^* \longrightarrow \mathsf{\Lambda}^2 W$.
We consider the map $\Psi\colon (\mathsf{\Lambda}^2 W \otimes W^*)\otimes (\mathsf{\Lambda}^2
W \otimes W^*) \longrightarrow \mathsf{\Lambda}^3 W \otimes W^* $ given by
\[
\Psi
\big(
(u \otimes \varphi) \otimes (v \otimes \psi) \big)= 
(u \wedge \kappa_W(v\otimes\varphi)) \otimes \psi. 
\] 
A direct computation shows that
\[
\Psi(h_6\otimes h_6)= - 6 \, e_{123}\otimes e^* \in V_{D_2+D_7}^*.
\]

\subsection{Cotype $\mathsf{E}_8$} \label{ssec:cotipoE8}

Let $V=\mathbb{C}^4$, and set $W = V \oplus \mathsf{\Lambda}^2 V \oplus V^* =X\oplus Y
\oplus Z$.  On $W$ it is defined a non-degenerate symmetric bilinear
form such that $V\oplus V^*$ is orthogonal to $\mathsf{\Lambda}^2 V$ and such that
restricted to $V\oplus V^*$ and to $\mathsf{\Lambda}^2 V$ is the one introduced at
the beginning of this section.

Let $A$ be the kernel of the wedge product $V\otimes \mathsf{\Lambda}^2 V\longrightarrow \mathsf{\Lambda}^3 V$ and set 
$\mathfrak{h}_0 = \mathfrak{sl}(V) \oplus A \oplus \mathsf{\Lambda}^2 V$.  There is an action
of $\mathfrak{h}_0$ on $W$ given as follows. Let $(a,b,\alpha)\in \mathfrak{h}_0$ and
$(v,\omega,\varphi)\in W$ then
\begin{align*}
a\cdot v &= a(v) & b\cdot v &= 0 & \alpha\cdot v &= 0 \\
a\cdot \omega &= a(\omega) & b\cdot \omega &= \gamma(b\otimes\omega) & \alpha\cdot \omega &= 0 \\
a\cdot \varphi &= -a^t(\varphi) & b\cdot \varphi &= \delta(b\otimes\varphi) & \alpha\cdot \varphi &= 
\kappa_V(\alpha\otimes\varphi) \\
\end{align*}
where $\gamma,\delta$ are defined as follows. 
\begin{align*}
 \gamma &\colon V\otimes \mathsf{\Lambda}^2 V \otimes \mathsf{\Lambda}^ 2V \longrightarrow V &
 \gamma (v\otimes \alpha\otimes \alpha')&= (\alpha,\alpha')\,v \\ 
 \delta &\colon V\otimes \mathsf{\Lambda}^2 V \otimes V^* \longrightarrow \mathsf{\Lambda}^2 V &
 \delta (v\otimes \alpha\otimes \varphi)&= \varphi(v)\,\alpha 
\end{align*}
The action of
$\mathfrak{h}_0$ on $W$ defines an immersion of $\mathfrak{h}_0$ into $\mathfrak{so}(W)$,
whose image is closed under the Lie bracket and more explicitly
for $a,a'\in \mathfrak{sl}(V)$, $b,b'\in A$ and $\alpha,\alpha'\in \mathsf{\Lambda}^2 V$ we have:
\begin{align*}
 [a,a']&=aa'-a'a  &  [a,b]&= a(b) & [a,\alpha]&=a(\alpha) \\
 [b,b']&= -\zeta(b\otimes b') & [b,\alpha]&= 0 & [\alpha,\alpha']&=0
\end{align*}
where $\zeta$ is the map 
\[
\zeta\colon (V\otimes \mathsf{\Lambda}^2 V) \otimes (V\otimes \mathsf{\Lambda}^2 V)\longrightarrow \mathsf{\Lambda}^2 V
\qquad \zeta\left((u\otimes \alpha)\otimes(u'\otimes \alpha')\right)=
(\alpha,\alpha')\, u\wedge u'.
\]

The triples of Lemma~\ref{lemma: triple fondamentali piatte E} are
$(D_1,D_1,D_2)$,
$(D_1,D_5,2D_2)$,
$(D_1,D_7,D_3)$,
$(D_1,D_8,D_7)$,
$(D_3,D_8,D_5)$,
$(D_5,D_8,D_2 + D_7)$ and
$(D_7,D_8,D_2)$. 

\subsubsection{Computation of $h_1$.}

The representation associated to $D_1$ is 
$\mathsf{\Lambda}^3 W$. Looking at the action of $\mathfrak{sl}(V)$ we get two
invariants: $x\in \mathsf{\Lambda}^2 X \otimes Y\simeq \mathsf{\Lambda}^2 V \otimes \mathsf{\Lambda}^2 V$
and $y\in Y\otimes \mathsf{\Lambda}^2 Z \simeq \mathsf{\Lambda}^2 V \otimes \mathsf{\Lambda}^2 V^*$. We notice now that 
$X$ is invariant also by the action of $A$ and $\mathsf{\Lambda}^2 V$. Indeed if $b \in A$ then 
$b\cdot x \in \mathsf{\Lambda}^3 X \simeq \mathsf{\Lambda}^3 V$, in particular we get a $\mathfrak{sl}(V)$-equivariant map
from $A$ to $\mathsf{\Lambda}^3 V$, which must be zero. A similar argument proves that 
$\alpha\cdot x =0$ for $\alpha\in\mathsf{\Lambda}^2V$.
Finally, it is easy to check that if $b=e_1\wedge e_{12}$, which belongs to $A$, then $b\cdot y \neq 0$. Hence
\[
 h_1 = x =
 e_1 \wedge e_2 \wedge e_{34} - e_1 \wedge e_3 \wedge e_{24} + e_1 \wedge e_4 \wedge e_{23} +
 e_2 \wedge e_3 \wedge e_{14} - e_2 \wedge e_4 \wedge e_{13} + e_3 \wedge e_4 \wedge e_{12}.
\]

\subsubsection{Computation of $h_3, h_5, h_7$.}

Let $P$ be the parabolic of $\Spin (W)$ defined by $g(V)\subset
V$. Notice that $H\subset P$. Let $U$ be the unipotent radical of $P$,
$L$ its Levi subgroup, and $L^{\mathrm{ss}}$ its semisimple part. Notice that
$L^{\mathrm{ss}} \simeq \SL (4)\times \SL (4)$. Let $T\subset G$ the maximal
torus of elements acting diagonally on $W$ with respect to the basis
$e_1,\dots,e_4,
e_{12},e_{13},e_{14},e_{23},-e_{24},e_{34},\varphi_4,\dots,\varphi_1$ and
$B\subset G$ the subgroup of elements whose action on $W$ is upper
triangular with respect to this basis. The natural action of $\SL(V)$ on $W$
induces an embedding $\SL(V)\longrightarrow L$
that on diagonal elements takes the form 
\[(t_1,t_2,t_3,t_4)\longmapsto 
(t_1,t_2,t_3,t_4,
t_1t_2,
t_1t_3,
t_1t_4,
t_2t_3,
t_2t_4,
t_3t_4,
t_4^{-1},
t_3^{-1},
t_2^{-1},
t_1^{-1}).
\]

The set of colors $\{D_1,D_4,D_6,D_8\}$ is distinguished, and 
$K_0=H_0U$ is the intersection of the kernels of the multiplicative characters of the stabilizer $K$ of a point in the open orbit of the quotient.  Therefore, $h_3, h_5, h_7$ must be $K_0$-invariant vectors in $V_{D_i}^*$. Let $W_i=(V_{D_i}^*)^U$. This is an
irreducible representation of $L$ of the same highest weight as
$V_{D_i}^*$: $\omega_3+\omega_7$ for $i=3$, $\omega_2+\omega_5$ for $i=5$, and
$\omega_1+\omega_6$ for $i=7$. When we restrict these representations to $\SL(V)$ we get
\[
W_3 \simeq \mathsf{\Lambda}^3 V \otimes V \qquad 
W_5 \simeq \mathsf{\Lambda}^2 V \otimes \mathsf{\Lambda}^2 V \qquad 
W_7 \simeq V \otimes \mathsf{\Lambda}^3 V
\] 
in particular there is an invariant element under $\mathfrak{h}_0$. 

For notational convenience, here and below, set 
\[
e_5=e_{12}, e_6=e_{13}, e_7=e_{14}, 
\varphi_7=e_{23}, \varphi_6=-e_{24}, \varphi_5=e_{34}.
\]
Let $U\subset W$ be the subspace spanned by $e_1,\dots,e_7$,
so that $W$ becomes $U\oplus U^*$.

Now we need to describe the spin representations.
Consider the whole exterior algebra $\mathsf{\Lambda} U^*$. 
It decomposes into odd and even degree parts 
$\mathsf{\Lambda}^\mathrm{odd} U^*\oplus\mathsf{\Lambda}^\mathrm{even}U^*$.
Since the $G$-action we are going to define is not the natural one, 
we stress the difference by using a different notation: 
set $\psi_{i_1\ldots i_k}=\varphi_{i_1\ldots i_k}$.  
Define the map $\sigma\colon W\otimes \mathsf{\Lambda} U^*\longrightarrow\mathsf{\Lambda} U^*$ such that
\[
\sigma(e_i\otimes \psi_{i_1\ldots i_k})=
\kappa^{k1}_{U^*}(\varphi_{i_1\ldots i_k}\otimes e_i)
\]
\[
\sigma(\varphi_i\otimes \psi_{i_1\ldots i_k})=
\varphi_i\wedge \varphi_{i_1\ldots i_k},
\]
and more generally the map 
$\sigma^n\colon \otimes^nW\otimes \mathsf{\Lambda} U^*\longrightarrow\mathsf{\Lambda} U^*$ such that
\[
\sigma^n(w_1\otimes\dots\otimes w_n\otimes y)=
\sigma(w_1\otimes\sigma(\dots\otimes\sigma(w_n\otimes y)\dots))
\]
which we can restrict to $\mathsf{\Lambda}^nW$ 
if we think $w_1\wedge\dots\wedge w_n$ 
as the corresponding antisymmetric tensor, 
with coefficient $\frac1{n!}$. 

To get the spin representations we can just take the map $\sigma^2$,
indeed notice that $\mathsf{\Lambda}^2W$ identifies with $\mathfrak{so}(W)$ through
$w_1\wedge w_2=(w_2, )w_1-(w_1, )w_2$.
We thus have that the vector $\psi_{i_1\ldots i_k}$ has weight 
\[\frac12\left(\sum_{i\not\in\{i_1,\dots,i_k\}}\varepsilon_i
-\sum_{i\in\{i_1,\dots,i_k\}}\varepsilon_i\right),\]
$V(\omega_6)\simeq \mathsf{\Lambda}^\mathrm{odd}U^*$ and
$V(\omega_7)\simeq \mathsf{\Lambda}^\mathrm{even}U^*$.

We get the following expressions of the $H_0$-invariants:
\[h_3=
 e_{123}\otimes\psi_{65}
+e_{124}\otimes\psi_{75}
+e_{134}\otimes\psi_{76}
-e_{234}\otimes\psi_\emptyset,
\]
\[h_5=
 e_{12}\otimes e_{1234}\wedge\varphi_5
+e_{13}\otimes e_{1234}\wedge\varphi_6
+e_{14}\otimes e_{1234}\wedge\varphi_7
+e_{23}\otimes e_{12347}
-e_{24}\otimes e_{12346}
+e_{34}\otimes e_{12345},
\]
\[h_7=
 e_1\otimes\psi_{765}
-e_2\otimes\psi_5
-e_3\otimes\psi_6
-e_4\otimes\psi_7.
\]

\subsubsection{Computation of $h_8$.} \label{ssec:comodelloh8}

The representation $V_{D_8}^*$ is the spin representation of highest
weight $\omega_6$. 
By direct computation one can show that the only $H_0$-invariant is given by
\[
h_8=
\psi_1 + \psi_{762} - \psi_{753} + \psi_{654}.
\]

\subsubsection{Analysis of the triple $(D_1,D_1,D_2)$.} 

The representation associated with $D_2$ is $\mathsf{\Lambda}^4 W$. 
The $H_0$-invariant $h_2$ is $e_{1234}$. 
Indeed, the set of colors $\Delta\smallsetminus\{D_2\}$ is distinguished and 
the quotient is homogeneous, hence a $H$-semiinvariant in $V_{D_2}^*$ 
must be $P$-semiinvariant.

Here we get \[\tilde\kappa^{3\,3}_W(h_1\otimes h_1)=3h_2.\]

\subsubsection{Analysis of the triple $(D_1,D_5,2D_2)$.} 

Consider the map $\Phi\colon\mathsf{\Lambda}^3W\otimes\mathsf{\Lambda}^2W\otimes\mathsf{\Lambda}^5W
\longrightarrow \mathsf{\Lambda}^4W\otimes\mathsf{\Lambda}^4W$ such that
\begin{align*}
\Phi \big((w_1\wedge w_2\wedge w_3)\otimes x \otimes y\big) & = (w_1\wedge w_2\wedge x)\otimes \kappa^{5\,1}_W(y\otimes w_3)+
\\
& \; -(w_1\wedge w_3\wedge x)\otimes \kappa^{5\,1}_W(y\otimes w_2)
 +(w_2\wedge w_3\wedge x)\otimes \kappa^{5\,1}_W(y\otimes w_1).
\end{align*}
We have
\[\Phi(h_1\otimes h_5)=6\,h_2\otimes h_2.\]

\subsubsection{Analysis of the triple $(D_1,D_7,D_3)$.}

Consider the map $\Phi\colon\mathsf{\Lambda}^3W\otimes W\otimes\mathsf{\Lambda}^\mathrm{odd}U^*
\longrightarrow \mathsf{\Lambda}^3W\otimes \mathsf{\Lambda}^\mathrm{even}U^*$ such that
\begin{align*}
\Phi & \big((w_1\wedge w_2\wedge w_3)\otimes w \otimes \psi\big) =\\
& (w_1\wedge w_2\wedge w)\otimes \sigma(w_3\otimes\psi)
 -(w_1\wedge w_3\wedge w)\otimes \sigma(w_2\otimes\psi)
 +(w_2\wedge w_3\wedge w)\otimes \sigma(w_1\otimes\psi).
\end{align*}
We get
\[\Phi(h_1\otimes h_7)=3h_3.\]

\subsubsection{Analysis of the triple $(D_1,D_8,D_7)$.} 

Consider the map $\Phi\colon\mathsf{\Lambda}^3W\otimes \mathsf{\Lambda}^\mathrm{odd}U^*
\longrightarrow W\otimes \mathsf{\Lambda}^\mathrm{odd}U^*$ such that
\begin{align*}
\Phi & \big((w_1\wedge w_2\wedge w_3)\otimes \psi\big) =\\
& w_1\otimes \sigma^2((w_2\wedge w_3)\otimes\psi)
 -w_2\otimes \sigma^2((w_1\wedge w_3)\otimes\psi)
 +w_3\otimes \sigma^2((w_1\wedge w_2)\otimes\psi).
\end{align*}
We get
\[\Phi(h_1\otimes h_8)=-3h_7.\]

\subsubsection{Analysis of the triple $(D_3,D_8,D_5)$.} 

On $\mathsf{\Lambda}^6W$ there is a symmetric bilinear form $(\,,\,)$ 
naturally induced by the given form on $W$. On the other hand, 
$V(\omega_6)$ and $V(\omega_7)$ are dual to each other,
so there is a natural non-degenerate pairing $\langle\,,\,\rangle$.
Consider the map 
$\Psi\colon\mathsf{\Lambda}^\mathrm{odd}U^*\otimes \mathsf{\Lambda}^\mathrm{even}U^*
\longrightarrow \mathsf{\Lambda}^6W$ such that
\[\big(u,\Psi(\psi\otimes\psi')\big)=\langle\sigma^6(u)\psi,\psi'\rangle,\]
and the map
$\Phi\colon\mathsf{\Lambda}^3W\otimes 
\mathsf{\Lambda}^\mathrm{odd}U^*\otimes\mathsf{\Lambda}^\mathrm{even}U^*
\longrightarrow \mathsf{\Lambda}^2W\otimes \mathsf{\Lambda}^5W$ such that
\begin{align*}
\Phi & \big((w_1\wedge w_2\wedge w_3)\otimes \psi\otimes\psi'\big) =\\
 & (w_1\wedge w_2)\otimes 
  \tilde\kappa^{2\,1}_W(\Psi(\psi\otimes\psi')\otimes w_3)\\
-& (w_1\wedge w_3)\otimes 
  \tilde\kappa^{2\,1}_W(\Psi(\psi\otimes\psi')\otimes w_2)\\
+& (w_2\wedge w_3)\otimes 
  \tilde\kappa^{2\,1}_W(\Psi(\psi\otimes\psi')\otimes w_1)
\end{align*}
We get
\[\Phi(h_3\otimes h_8)=h_5.\]

\subsubsection{Analysis of the triple $(D_5,D_8,D_2+D_7)$.} 

Consider the map $\Phi\colon\mathsf{\Lambda}^2W\otimes\mathsf{\Lambda}^5W
\otimes \mathsf{\Lambda}^\mathrm{odd}U^*
\longrightarrow \mathsf{\Lambda}^4W\otimes W\otimes\mathsf{\Lambda}^\mathrm{odd}U^*$ such that
\begin{align*}
\Phi & \big((w_1\wedge w_2)\otimes(z_1\wedge\dots\wedge z_5)\otimes 
\psi\big) =\\
 &\sum_i(-1)^{i+1}(z_1\wedge\dots\wedge\hat z_i\wedge\dots\wedge z_5)
  \otimes w_2\otimes \sigma^2((z_i\wedge w_1)\otimes\psi)\\
-&\sum_i(-1)^{i+1}(z_1\wedge\dots\wedge\hat z_i\wedge\dots\wedge z_5)
  \otimes w_1\otimes \sigma^2((z_i\wedge w_2)\otimes\psi)
\end{align*}
We get
\[\Phi(h_5\otimes h_8)=-3\,h_2\otimes h_7.\]

\subsubsection{Analysis of the triple $(D_7,D_8,D_2)$.}

Consider the map $\Phi\colon W
\otimes \mathsf{\Lambda}^\mathrm{odd}U^*\otimes \mathsf{\Lambda}^\mathrm{odd}U^*
\longrightarrow \mathsf{\Lambda}^\mathrm{odd}U^*\otimes \mathsf{\Lambda}^\mathrm{even}U^*$ 
such that
\[
\Phi(w\otimes \psi \otimes \psi') =
\sigma(w\otimes\psi')\otimes\psi.\]
Here we get that $\Phi(h_7\otimes h_8)$ 
is a $U$-invariant vector of weight $\omega_4$.

\subsection{Projective normality of comodel wonderful varieties.}	\label{ssec: proof of teo comodello}

\begin{proof}[Proof of Theorem~\ref{teo: comodello-proj-norm}.]
As in the case of the model wonderful varieties, 
by Lemma~\ref{lem:  proiettiva normalita}, 
we are reduced to study the low fundamental triples.

Recall that in the model case we have classified,
for every $G$ of connected Dynkin type,
the low fundamental triples $(D,E,F)$ of the model wonderful $G$-variety 
(of simply connected type) such that $\supp_S(D+E-F)=S$.

In the comodel case, these correspond to 
the low fundamental triples $(D,E,F)$
of the comodel wonderful varieties of connected Dynkin cotype 
such that $\mathrm{Cosupp}(D+E-F)=\Delta$, where 
\[\mathrm{Cosupp}\,\gamma=\{D\in\Delta\, : \,\langle\alpha^\vee,\omega_D\rangle\neq0
\mbox{ for some }\alpha\in\supp_S\gamma\}.\]

Let us go back to the model case.
Let $S$ be the set of simple roots. 
Let $M$ be a model wonderful variety with set of colors $\Delta$.
Recall that for every low fundamental triple $(D,E,F)$, 
$\supp_S(D+E-F)=S'$ is connected. 
Hence $(D,E,F)$ corresponds to a low fundamental triple $(D,E,\widetilde F)$,
of a model wonderful variety $N$
(with set of colors $\widetilde\Delta\subset\Delta$) 
of connected Dynkin type (with $S'$ as set of simple roots) 
such that $M$ is parabolic induction of $N$,
with $D,E\in\widetilde\Delta$, $D+E-F=D+E-\widetilde F$ 
and $\supp_{S'}(D+E-\widetilde F)=S'$.

This can be translated into the comodel case. 
Every low fundamental triple $(D,E,F)$,  
of a comodel wonderful variety $M$ (with set of colors $\Delta$),
corresponds to a low fundamental triple $(D,E,\widetilde F)$,
of a comodel wonderful variety $N$ of connected Dynkin cotype 
(with set of colors $\widetilde\Delta\subset\Delta$) 
such that $M$ is parabolic induction of $N$,
with $D,E\in\widetilde\Delta$, $D+E-F=D+E-\widetilde F$ 
and $\mathrm{Cosupp}(D+E-\widetilde F)=\widetilde \Delta$.

Since we have already checked the surjectivity 
for all such low fundamental triples,
we can conclude by applying Proposition~\ref{prop:  ind-par} 
as in the proof of Theorem~\ref{teo: model} (see Section~\ref{ssec:proofA}). 
\end{proof}

\section{On the normality of spherical orbit closures \\in simple projective spaces}

Let $M$ be a wonderful $G$-variety 
with set of spherical roots $\Sigma$ and set of colors $\Delta$. 
Denote by $H$ the stabilizer of a point $x_0$ in the open $G$-orbit. 

If $D$ is in $\mathbb{N}\Delta$ and $h_D \in
(V_D^*)_{\xi_D}^{(H)}$ is the associated $H$-eigenvector, consider
the orbit closure
\[
	X_D = \overline{G\cdot [h_D]} \subset \mathbb{P}(V_D^*),
\]
which is a simple possibly non-normal spherical variety (recall that a spherical variety is called \textit{simple} if it contains a unique closed orbit). 
We have a natural morphism $\phi_D \colon M \longrightarrow X_D$ such
that $\phi_D^* \mathcal{O}(1) = \mathcal L_D$.

By \cite[Corollary~7.6]{Kn} and \cite[Corollary~2.4.2.2]{BL},
every spherical orbit in a simple projective space always admits a wonderful
compactification $M$, so it is of the shape $G\cdot [h_D]$ for some $D \in \mathbb{N} \Delta$. As a consequence of the results of the previous sections, here we 
show that under some special assumptions on $M$
the variety $X_D$ is always normal.

The variety $X_D$ was studied by G.~Pezzini in \cite{Pe} when $D$ is ample, 
that is, $D\in\mathbb{N}_{>0} \Delta$. 
Under this assumption, either $X_D$ is isomorphic to
$M$ or it is not even normal. In case $X_D \simeq M$, then $M$ is
called \textit{strict}: this is equivalent to the conditions $H =
N_G(H)$ and $\Sigma \cap S = \varnothing$.
There are essentially two main classes of examples of strict wonderful varieties:
the adjoint symmetric wonderful varieties and the model wonderful varieties.

When $D$ is not ample, the variety $X_D$ was then studied in
\cite{Ma} in the symmetric case and in \cite{Ga} in general. 
More precisely, the orbit structure of $X_D$ and that of its normalization $\widetilde{X}_D$ were analyzed. 
In particular, it was proved that the normalization $\widetilde{X}_D \longrightarrow X_D$
is always bijective if $M$ is adjoint symmetric or if $G$ is of simply
laced type and $M$ is strict, while the main counterexamples where
bijectivity fails, in the strict case, arise with the model
wonderful varieties for groups of not simply laced type.

We now consider a special class of big divisors on $M$.
We say that $D \in \mathbb{N}\Delta$ is a \textit{faithful divisor} on $M$ if $\phi_D$
restricts to an open embedding of $G/H$ in $\mathbb{P}(V_D^*)$, i.e.\ if $H$ equals the stabilizer of $[h_D]$.
The formalism of distinguished sets of colors allows to characterize combinatorially
the faithful divisors. For simplicity, we restrict to the case
of a strict wonderful variety.

\begin{proposition} [see {\cite[Proposition~2.4.3]{BL}}] \label{prop: fedelta caso stretto}
Let $M$ be a strict wonderful variety and let $D \in \mathbb{N} \Delta$.
Then $D$ is faithful if and only if
every distinguished subset of $\Delta$ intersects $\supp(D)$.
\end{proposition}

Let $D \in \mathbb{N}\Delta$, denote $\widetilde{A}(D) = \bigoplus_{n \in \mathbb{N}}
\Gamma(M,\mathcal L_{nD})$ and denote $A(D) \subset \widetilde A(D)$ the subalgebra
generated by $V_D \subset \Gamma(M,\mathcal L_D)$.
Let $\widetilde X_D$ be the image of $M$ in
$\mathbb{P}\big(\Gamma(M,\mathcal L_D)^*\big)$ via the morphism $\widetilde\phi_D$
associated to the complete linear system of $D$.
Then $\widetilde{A}(D)$ is identified with the projective
coordinate ring of $\widetilde X_D$, whereas $A(D)$
is identified with the projective coordinate ring of $X_D$.
Notice that we have a natural projection $\beta_D\colon\widetilde X_D\longrightarrow X_D$
such that $\phi_D = \widetilde\phi_D \circ \beta_D$.
Since $M$ is smooth, $\widetilde{A}(D)$ is an integrally closed algebra,
therefore $\widetilde X_D$ is normal. Moreover, we have the following.

\begin{proposition} [see {\cite[Proposition~2.1]{CDM}}] \label{prop: algebra integrale}
The algebra $\widetilde A(D)$ is integral over $A(D)$.
\end{proposition}

It follows that $\beta_D \colon \widetilde X_D \longrightarrow X_D$ is the normalization if and only if
it is birational. Clearly this is the case
if $D$ is minuscule or faithful. On the other hand, $\beta_D$ is not necessarily birational:
if $M$ is the model wonderful variety of $\mathsf{C}_4$,
then $A(D_2) \subset \widetilde A(D_2)$ don't have the same quotient field.
However, we have the following.

\begin{proposition} \label{prop: caso stretto simply laced}
Suppose that $M$ is adjoint symmetric or that $G$ is simply laced and $M$ is
strict. Then $\widetilde A(D)$ and $A(D)$ have the same quotient field.
\end{proposition}

\begin{proof}
Let $\Delta_D \subset \Delta$ be the maximal distinguished subset
which does not intersect $\supp(D)$, and denote by $M'$ the quotient of $M$ by $\Delta_D$ and
by $M''$ the wonderful compactification of $G\cdot [h_D]$.
By \cite[Corollary~3.7]{Ga} and \cite[Remark~3.8]{Ga}, $M'$ and $M''$ have the same dimension,
and we have $M'\neq M''$ if and only if some spherical root of $M''$
is the double of a spherical root of $M'$.

On the other hand, by the quotient construction, the spherical roots of $M'$ are contained in $\mathbb{N}\Sigma$. 
By the description of the possible spherical roots occurring in a symmetric variety or in a strict variety for a simply laced $G$, 
it follows that $\mathbb{N}\Sigma$ contains no elements of $\Sigma(G)$ whose double is still in $\Sigma(G)$ (see \cite[Table~1]{BL}).
Therefore, $M'$ and $M''$ must have the same spherical roots and we get $M' = M''$, 
and by Corollary~\ref{cor: sezioni e quozienti} it follows that $\Gamma(M, \mathcal L_{nD}) \simeq \Gamma(M', \mathcal L_{nD})$ for every $n\geqslant 0$. 
Therefore, we are reduced to the case of a faithful divisor, and the claim follows.
\end{proof}

In the adjoint symmetric case, the above proposition was proved in \cite[Theorem~2.6]{CDM}.

If $D_1,\dots, D_m \in \mathbb{N}\Delta$, consider the variety
\[ X_{D_1,\dots, D_m} = \overline{G\cdot ([h_{D_1}] \times\dots \times
  [h_{D_m}])} \subset \mathbb{P}(V_{D_1}^*) \times\dots \times
\mathbb{P}(V_{D_m}^*)
\] and denote by $ \phi_{D_1,\dots, D_m} \colon M \longrightarrow X_{D_1,\dots, D_m}$
        the map such that $\phi_{D_1,\dots, D_m}(x) =
        (\phi_{D_1}(x),\dots, \phi_{D_m}(x))$.

\begin{proposition}	[{\cite[Proposition~1.2]{BGMR}}]	\label{prop: riduzione supporto}
Let $M$ be a wonderful variety and let $D,E \in \mathbb{N}\Delta$.
\begin{itemize}
	\item[i)] If $\supp(\omega_D) \cap \supp(\omega_E) = \varnothing$, then
          $X_{D+E} \simeq X_{D, E}$.
	\item[ii)] If $M$ is strict and $\supp(D) = 
          \{D_1,\dots, D_m\}$, then $X_D \simeq X_{D_1,\dots , D_m}$. In
          particular, if $M$ is strict, $X_D \simeq X_E$ if and
          only if $\supp(D) = \supp(E)$.
\end{itemize}
\end{proposition}

\begin{proof}
i) Since $\supp(\omega_D) \cap \supp(\omega_E) = \varnothing$, by 
  \cite[Lemma~1.1]{BGMR} we have a closed equivariant embedding $\psi_{D,E} \colon
\mathbb{P}(V^*_D) \times \mathbb{P}(V^*_E) \longrightarrow \mathbb{P}(V^*_{D+E})$, and since
$\psi_{D,E}([h_D],[h_E]) = [h_{D+E}]$ the isomorphism
$X_{D+E} \simeq X_{D,E}$ follows.

ii) Since $M$ is strict, by the description of the restriction $\omega\colon
\Pic(M) \longrightarrow \mathcal{X}(T)$ (see \cite[Lemma~30.24]{Ti}) it follows that $\supp(D) \cap
\supp(E) = \varnothing$ if and only if $\supp(\omega_D) \cap \supp(\omega_E) =
\varnothing$. Therefore, by \cite[Lemma~1.1]{BGMR} we have a closed
equivariant embedding $\psi_D \colon \mathbb{P}(V^*_{D_1}) \times\dots \times
\mathbb{P}(V^*_{D_m}) \longrightarrow \mathbb{P}(V^*_D)$ such that $\psi_D([h_{D_1}],\dots,
   [h_{D_m}]) = [h_D]$ and it follows $X_D \simeq X_{D_1,\dots,
     D_m}$.
\end{proof}

Assume now that the multiplication of sections is surjective
for every couple of globally generated line bundles.
In particular $\widetilde A(D)$ is generated by its degree one component $\Gamma(M,\mathcal L_D)$
and it follows that $A(D) = \widetilde A(D)$ if and only if $\Gamma(M,\mathcal L_D)=V_D$
if and only if $D$ is minuscule or $D=0$. Since $\widetilde A(D)$ is integrally closed,
we get the following proposition which we need for later use.

\begin{proposition} \label{prop: cono1}
Let $M$ be a wonderful variety and suppose that the multiplication of
sections is surjective for every couple of globally generated line bundles.
Let $D \in \mathbb{N}\Delta$, $D\neq0$.
\begin{itemize}
	\item[i)] If $D$ is minuscule, then $X_D$ is projectively normal.
	\item[ii)] If $\beta_D$ is birational and $X_D$ is projectively normal,
	then $D$ is minuscule.
\end{itemize}
\end{proposition}

If $D_1,\dots, D_m \in \mathbb{N}\Delta$ and if $V_1,\dots, V_m$ are
$G$-modules of sections such that $V_{D_i} \subset V_i \subset
\Gamma(M,\mathcal L_{D_i})$, consider the associated morphisms $\phi_{V_i} \colon
M \longrightarrow \mathbb{P}(V_i^*)$ and denote $X_{V_i} = \phi_{V_i}(M)$.
Also, we denote by
\[X_{V_1, \dots,V_m} \subset \mathbb{P}(V^*_1) \times \dots \times \mathbb{P}(V^*_m)\]
the image of $M$ via the map $\phi_{V_1,\dots,V_m}(x)=(\phi_{V_1}(x), \dots, \phi_{V_m}(x))$ and by
\[X_{V_1 \otimes \dots \otimes V_m} \subset \mathbb{P}(V^*_1 \otimes \dots \otimes V^*_m)\]
the image of $X_{V_1,\dots,V_m}$ via the Segre embedding.

\begin{lemma}	\label{lemma: coordinate proiettive}
Let $D_1, \dots, D_m \in \mathbb{N}\Delta$ and denote $D = \sum_{i=1}^m
D_i$. If $V_1, \dots, V_m$ are $G$-modules such that
$V_{D_i} \subset V_i \subset \Gamma(M,\mathcal L_{D_i})$, then the projective
coordinate ring of $X_{V_1 \otimes \dots \otimes V_m}$ is the subalgebra
$A(V_1,\dots, V_m) \subset \widetilde{A}(D)$ generated by the product $V_1
\cdots V_m \subset \Gamma(M,\mathcal L_D)$.
\end{lemma}

\begin{proof}
Consider the map $\phi_{V_1 \otimes \dots \otimes V_m}\colon M\to X_{V_1 \otimes \dots \otimes V_m}$.
The lemma follows by noticing that $\phi^*_{V_1 \otimes \dots \otimes
V_m} \mathcal{O}(1) = \mathcal L_D$ and that $\phi^*_{V_1 \otimes \dots
\otimes V_m} \colon V_1 \otimes \dots \otimes V_m \longrightarrow \Gamma(M,\mathcal L_D)$
is the multiplication map.
\end{proof}

\begin{proposition}
Let $M$ be a wonderful variety and suppose that the multiplication of
sections is surjective for every couple of globally generated line bundles.
Let $D_1, \dots, D_m \in \mathbb{N} \Delta$ and denote
$\Gamma_i = \Gamma(M,\mathcal L_{D_i})$. Then the variety
$X_{\Gamma_1 \otimes \dots \otimes \Gamma_m} \subset \mathbb{P}(\Gamma^*_1 \otimes \dots \otimes \Gamma^*_m)$ is projectively normal.
\end{proposition}

\begin{proof}
Denote $D = \sum_{i=1}^m D_i$. Since the
multiplication of sections is surjective, $\Gamma_1 \cdots \Gamma_m = \Gamma(M,\mathcal L_D)$,
hence by the previous lemma $A(\Gamma_1, \dots, \Gamma_m) = \widetilde{A}(D)$
and $X_{\Gamma_1 \otimes \dots \otimes \Gamma_m}$ is a projectively normal variety.
\end{proof}

\begin{corollary}
Let $M$ be a wonderful variety and suppose that the multiplication of
sections is surjective for every couple of globally generated line bundles.
\begin{itemize}
	\item[i)] Let $D,E \in \mathbb{N}\Delta$ be such that $\supp(\omega_D)
          \cap \supp(\omega_E) = \varnothing$. If $X_D, X_E$ are normal, then
          $X_{D+E}$ is normal as well.
	\item[ii)] If $M$ is strict, then $X_D$ is normal for all
          $D\in \mathbb{N}\Delta$ if and only if it is normal for all $D \in
          \Delta$.
\end{itemize}
\end{corollary}

\begin{proof}
i) Denote $\Gamma_D = \Gamma(M,\mathcal L_D)$ and $\Gamma_E = \Gamma(M,\mathcal L_E)$:
by the previous proposition, we have that $X_{\Gamma_D,\Gamma_E} \simeq
X_{\Gamma_D \otimes \Gamma_E}$ is a normal variety. On the other hand
since $X_D$ and $X_E$ are normal, we have that $X_{D,E} \simeq
X_{\Gamma_D,\Gamma_E}$, while by Proposition~\ref{prop: riduzione
  supporto}.i we have that $X_{D+E} \simeq X_{D,E}$.

ii) Since $M$ is strict, it follows by the description of 
$\omega \colon \Pic(M) \longrightarrow \mathcal{X}(T)$ that $\supp(\omega_D) \cap
\supp(\omega_E) = \varnothing$ if and only if $\supp(D) \cap \supp(E) =
\varnothing$. Therefore the claim follows straightforwardly from i).
\end{proof}

\begin{corollary} \label{cor:normality-in-sps}
Suppose that $M$ is a symmetric variety with reduced root system of
type $\mathsf{A}$ or that it is a model wonderful variety for a connected
semisimple group of type $\mathsf{A} \mathsf{D}$. Then $X_D$ is normal for all
$D\in \mathbb{N}\Delta$.
\end{corollary}

\begin{proof}
By the description of the covering relation, it follows that under
the assumptions on $M$ it holds $\alt(\gamma^+) = 2$
for every covering difference $\gamma$ in $\mathbb{N}\Delta$.
Therefore every $D \in \Delta$ is minuscule and $\Gamma(M,\mathcal L_D) = V_D$.
Therefore, $X_D$ is projectively normal for all
$D \in \Delta$ and it follows by the previous corollary that $X_D$ is normal
for all $D \in \mathbb{N}\Delta$.
\end{proof}

\section{On the normality of cones and nilpotent orbit closures}	\label{sec:orbite}

Following the same approach as \cite{CDC} and \cite{CDM},
we can apply Theorem~\ref{teo: model} to study the normality of cones over model varieties.
In particular, as pointed out by Luna some years ago,
we can apply our theory to study the normality of the closure
of spherical nilpotent orbits in the Lie algebra of $G$.

Let $M$ be a wonderful variety with set of colors $\Delta$ and
set of spherical roots $\Sigma$ and
assume that the multiplication of sections is surjective
for every couple of globally generated line bundles.
Let $D\in \mathbb{N}\Delta$ and denote by $C_D \subset V_D^*$ the cone over the
variety $X_D$ introduced in the previous section.
Analogously, denote by $\widetilde{C}_D \subset \Gamma(M,\mathcal L_D)^*$
the cone over the variety $\widetilde X_D$.
Then the coordinate ring of $C_D$ is identified with $A(D)$,
whereas that of $\widetilde C_D$ is identified with $\widetilde A(D)$,
which is an integrally closed ring. This yields a map $\alpha_D \colon \widetilde C_D \longrightarrow C_D$
that is birational if and only if $\beta_D \colon \widetilde X_D \longrightarrow X_D$ is birational.
As already recalled in the previous section, $\widetilde A(D)$ is the integral closure
of $A(D)$ if and only if $\alpha_D$ is birational, whereas $A(D) = \widetilde A(D)$
if and only if $D$ is minuscule or $D=0$.

In the case of the model wonderful varieties of simply connected type, 
after the results of Section~\ref{sec: coperture modello}, 
we have the following classification of minuscule weights, where $a,b,c \in \mathbb{N}$.

\begin{itemize}
\item[{\it Case}] $\mathsf{A}_r$, $r$ even: $D_1, D_2, \dots, D_r$, $aD_1$,
  $aD_r$, $a D_1 +D_d$ with $d$ odd, $D_m +bD_r$ with $m$ even;
\item[{\it Case}] $\mathsf{A}_r$, $r$ odd: $D_1, D_2, \dots, D_r$, $aD_1$,
  $aD_r$, $aD_1 +D_d$ with $d$ odd, $D_m + bD_r$ with $m$ odd,
  $aD_1 +bD_r$;
\item[{\it Case}] $\mathsf{B}_r$, $r$ even: $aD_r$, $D_m +aD_r$ with $m$
  even;
\item[{\it Case}] $\mathsf{B}_r$, $r$ odd: $aD_1 +bD_r$, $aD_1 +D_m +bD_r$
  with $m$ odd;
\item[{\it Case}] $\mathsf{C}_r$: 
$a D_1$;
\item[{\it Case}] $\mathsf{D}_r$, $r$ even: $D_1, D_2, \dots, D_r$,
$aD_1 +bD_{r-1} +cD_r$, $a D_1 +D_m+bD_{r-1}  +cD_r$ with $m$ odd;
\item[{\it Case}] $\mathsf{D}_r$, $r$ odd: $D_1, D_2, \dots, D_r$,
$aD_{r-1} +bD_r$, $D_m+aD_{r-1}+bD_r$ with $m$ even;
\item[{\it Case}] $\mathsf{E}_6$: $D_1$, $D_6$, $aD_2$, $aD_2 +D_3$,
  $aD_2+D_5$;
\item[{\it Case}] $\mathsf{E}_7$: $D_1$, $D_6$, $aD_2 +bD_7$, $aD_2+D_3 +bD_7 $,
  $aD_2+D_5 +bD_7 $;
\item[{\it Case}] $\mathsf{E}_8$: $D_1$, $D_8$, $aD_2$, $aD_2 +D_3$,
  $aD_2 +D_5$, $aD_2+D_7$;
\item[{\it Case}] $\mathsf{F}_4$: $aD_1$;
\item[{\it Case}] $\mathsf{G}_2$: $aD_1$.
\end{itemize}

Let $\mathfrak{g}$ be the Lie algebra of $G$, as a general reference about nilpotent orbits in $\mathfrak g$ see \cite{Col}. If $\mathfrak{g}$ is not simple, then its nilpotent orbits are products of the nilpotent orbits of its
simple factors, and so are their closures.
Therefore we may assume that $\mathfrak{g}$ is simple.

Let $e\in\mathfrak{g}$ be a non-zero nilpotent element and let $\mathcal{O}$ be its adjoint orbit. By the Jacobson-Morozov theorem there exists an $\mathfrak{sl}(2)$-triple of the form $(e,h,f)$. Choose a
maximal toral subalgebra $\mathfrak{t}$ of $\mathfrak{g}$ containing $h$ and a Borel
subalgebra $\mathfrak{b}$ containing $\mathfrak{t}$ and $e$ and such that $\alpha(h) \geqslant 0$
for every $\alpha \in S$, where we denote by $S = \{\alpha_1, \dots, \alpha_r\}$
the set of simple roots defined by the choice of $\mathfrak{t}$ and $\mathfrak{b}$. The string
$(\alpha_1(h), \dots, \alpha_r(h))$ is called the Kostant--Dynkin diagram
of $\mathcal{O}$ and it uniquely determines the orbit $\mathcal{O}$.
Moreover, every $\alpha_i(h)$ is 0, 1 or 2. Let $\theta$ be
the highest root corresponding to the choice of $S$ and define the \textit{height} of
$\mathcal{O}$ as $\height(\mathcal{O})=\theta(h)$. The height does not depend
on the various choices we have made (see \cite[\S~3.5]{Col}); furthermore,
$\mathcal{O}$ is spherical if and only if $\height(\mathcal{O})\leqslant3$,
and this last condition is equivalent to say that
it has height equal to $2$ or to $3$, see \cite{Pa}.

By making use of the projective normality of the symmetric wonderful 
varieties, in \cite{CDM} it has been proved that the closure $\overline{\mathcal{O}}$ is normal
if $\height(\mathcal{O})=2$, which is originally due to W.~Hesselink \cite{He}. 
We now study the normality of $\overline{\mathcal{O}}$ in the case of $\height(\mathcal{O}) =3$
(see \cite[Table 2]{Pa}) by making use of the projective
normality of the model wonderful varieties.

Denote by $\mathcal{U} \simeq G/H$ the orbit of the line $[e] \in \mathbb{P}(\mathfrak{g}) = \mathbb{P}(V(\theta))$,
namely the image of $\mathcal{O}$ via the natural projection.
As every spherical orbit in the projective space of a simple $G$-module,
$\mathcal{U}$ possesses a wonderful compactification, which we denote by $M_\mathcal{O}$.
In \cite{BCF} we can find a description of the stabilizer of $[e]$
as well as the associated Luna diagram. 
In particular, $M_\mathcal{O}$ turns out to be a strict wonderful variety, 
and in particular the restriction of line bundles
to the closed orbit $\Pic(M_\mathcal{O}) \longrightarrow \mathcal{X}(B)$ is always 
injective. Therefore, we may regard $\theta$ as an element of $\mathbb{N}\Delta$
and we have $\overline{\mathcal{U}} = X_\theta$ and $\overline{\mathcal{O}} = C_\theta$.

In order to study the normality of $\overline{\mathcal{O}}$, we prove the following.

\begin{theorem}\label{teo:orbite}
Let $\mathcal{O} \subset \mathfrak{g}$ be a spherical nilpotent orbit and let $M_\mathcal{O}$
be the associated wonderful variety. Then the multiplication map
\[ m_{\mathcal L,\mathcal L'}\colon\Gamma(M_\mathcal{O},\mathcal L) \otimes \Gamma(M_\mathcal{O},\mathcal L')
\longrightarrow \Gamma(M_\mathcal{O},\mathcal L\otimes \mathcal L')\] 
is surjective for all globally generated line bundles $\mathcal L,\mathcal L'$ on $M_\mathcal{O}$.
\end{theorem}

Notice that by construction $\theta$ is identified
with a faithful divisor on $M_\mathcal{O}$.
Therefore the normality (resp.\ the non-normality) of $\overline{\mathcal{O}} = C_\theta$ follows
by noticing case-by-case that $\theta$ is minuscule (resp.\ non-minuscule) in $\mathbb{N}\Delta$,
and applying Proposition~\ref{prop: cono1}.
Given a nilpotent orbit $\mathcal{O}$ of height 3, we summarize the results
on the normality of its closure in Table \ref{tab:norm-nilp}.
Given a spherical nilpotent orbit $\mathcal{O}$, there we write in the second column its Kostant-Dynkin diagram, in the third column its Bala-Carter label (see \cite[\S 8.4]{CM}), in the fourth column the corresponding case in \cite{BCF} and in the fifth column the normality of $\overline{\mathcal{O}}$.
In particular we have the following.

\begin{table}\caption{Nilpotent orbits of height 3}\label{tab:norm-nilp}
\begin{center}
\begin{tabular}{|c||c|c|c|c|c|}
\hline
& type of $G$ & $\begin{array} {c} \text{Kostant-Dynkin}\\ \text{diagram} \end{array}$ & $\begin{array} {c} \text{partition/}\\ \text{Bala-Carter label} \end{array}$ & case in \cite{BCF} & norm. \\ \hline \hline
I & $\mathsf{B}_{2n+1}$ & $(10\ldots01)$ & $(3, 2^{2n})$ & (13) & no \\ \hline
II & $\begin{array} {c} \mathsf{B}_{2n+m+1}\\ \scriptsize\text{$(m > 0)$} \end{array}$ & $(10 \ldots 01 \underbrace{0\ldots 0}_m)$ & $(3, 2^{2n},1^{2m})$ & (18) & yes \\ \hline
III & $\mathsf{D}_{2n+2}$ & $(10\ldots011)$ & $(3, 2^{2n},1) $ & (41) & yes \\ \hline
IV & $\begin{array} {c} \mathsf{D}_{2n+m+2}\\ \scriptsize\text{$(m > 0)$} \end{array}$ & $(10\ldots01 \underbrace{0\ldots 0}_{m+1})$ & $(3, 2^{2n},1^{2m+1})$ & (43) & yes \\ \hline
V & $\mathsf{E}_6$ & $(000100)$ & $\mathsf{A}_1$ & (53) & yes \\ \hline
VI & $\mathsf{E}_7$ & $(0010000)$ & $(3\mathsf{A}_1)'$ & (54) & yes \\ \hline
VII & $\mathsf{E}_7$ & $(0100001)$ & $4\mathsf{A}_1$ & (51) & yes \\  \hline
VIII & $\mathsf{E}_8$ & $(00000010)$ & $3\mathsf{A}_1$ & (52) & yes \\ \hline
IX & $\mathsf{E}_8$ & $(01000000)$ & $4\mathsf{A}_1$ & (51) & yes \\ \hline
X & $\mathsf{F}_4$ & $(0100)$ & $\mathsf{A}_1 + \widetilde{\mathsf{A}}_1 \phantom{\Big|}$ & (60) & yes \\ \hline
XI & $\mathsf{G}_2$ & $(10)$ & $\widetilde{\mathsf{A}}_1 \phantom{\Big|}$ & (66) & no \\
\hline
\end{tabular}
\end{center}
\end{table}

\begin{corollary}	[see {\cite[Theorem~5.1]{Cos}} or {\cite[Table 2]{Pa}}]	\label{cor: normality-sno}	
Let $\mathcal{O} \subset \mathfrak{g}$ be a spherical nilpotent orbit with $\height(\mathcal{O}) = 3$. Then the closure $\overline{\mathcal{O}}$ is always normal but in cases $\mathrm{I}$ and $\mathrm{XI}$ of Table \ref{tab:norm-nilp}.
\end{corollary}

The rest of the section is devoted to the proof of Theorem~\ref{teo:orbite}
and Corollary~\ref{cor: normality-sno}.
Since the case $\height(\mathcal{O}) = 2$ was already considered in \cite{CDM},
we consider only the case $\height(\mathcal{O}) = 3$.

\subsection{Cases $\mathrm{I}$, $\mathrm{III}$, $\mathrm{VII}$, $\mathrm{IX}$,$\mathrm{XI}$: model orbits}

By the description in \cite{BCF}, in these cases we have that $M_\mathcal{O}$
is the model wonderful variety of the simply connected group $G$.
Notice that $\theta$ is always
minuscule in $\mathbb{N} \Delta$ but in the cases $\mathsf{B}_{2n+1}$ and $\mathsf{G}_2$:
it follows that $\overline{\mathcal{O}}$ is normal in the cases
$\mathrm{III}, \mathrm{VII}, \mathrm{IX}$, whereas it
is not normal in the cases $\mathrm{I}, \mathrm{XI}$.

\subsection{Cases $\mathrm{IV}$ ($m$ even), $\mathrm{VI}$, $\mathrm{VIII}$:
localization of model wonderful varieties}

In this case $M_\mathcal{O}$ is not a model wonderful variety, however it
is a quotient of a localization of a model wonderful variety
and we still may proceed as in the case of a model orbit
thanks to Corollary~\ref{cor: sezioni e quozienti} and Corollary~\ref{cor:localizzazione}.
In order to conclude the argument, we now describe which localizations and quotients
we have to take into account in each of the considered cases.
We denote by $M$ the model wonderful variety of $G$.

\subsubsection{Case $\mathrm{IV}$ ($m$ even).}

Let $N$ be the boundary divisor corresponding to the
spherical root $\sigma_{2n+1} = \alpha_{2n+1}+\alpha_{2n+2}$.
Then $M_\mathcal{O}$ is the quotient of $N$ by the distinguished subset of colors
$\Delta' = \{ D_{2n+2}, D_{2n+3}, \dots, D_{2n+m+2} \}$.
Since $\theta = D_2$ is minuscule in $\mathbb{N} \Delta$, it follows
that $\overline{\mathcal{O}}$ is normal.

\subsubsection{Case $\mathrm{VI}$.}

Let $N$ be the boundary divisor corresponding to the
spherical root $\sigma_{3}=\alpha_{3}+\alpha_{4}$.
Then $M_\mathcal{O}$ is the quotient of $N$ by the distinguished subset
$\Delta'=\{D_2,D_5,D_7 \}$. Since $\theta = D_1$ is minuscule in $\mathbb{N} \Delta$,
it follows that $\overline{\mathcal{O}}$ is normal.

\subsubsection{Case $\mathrm{VIII}$.}

Let $N$ be the boundary divisor corresponding to the
spherical root $\sigma_{6}=\alpha_{6}+\alpha_{7}$.
Then $M_\mathcal{O}$ is the quotient of $N$ by the distinguished subset
$\Delta'=\{D_2,D_3,D_4,D_5 \}$.
Since $\theta=D_8$ is minuscule in $\mathbb{N} \Delta$, it follows that $\overline{\mathcal{O}}$ is normal.

\subsection{Cases $\mathrm{II}$, $\mathrm{IV}$}

In these cases we need to prove the surjectivity of the multiplication
for other classes of wonderful varieties. 
Let $G = \Spin(k)$, let $r$ be the semisimple rank of $G$ 
(i.e., $k=2r+1$ if $k$ odd or $k=2r$ if $k$ even)
and let $2\leqslant s \leqslant (k-3)/2$.
Consider the wonderful variety $M$ corresponding in \cite{BCF} to the case (18)
when $k$ is odd and to the case (43) when $k$ is even.

Its spherical roots and colors are given as follows:
$\Sigma =\{\sigma_1,\dots,\sigma_s\}$ and
$\Delta=\{D_1,\dots,D_{s+1}\}$, where
$\sigma_i = \alpha_i+\alpha_{i+1} = D_i+D_{i+1}-D_{i-1}-D_{i+2}$ for
$i=1,\dots,s-1$ and
\[
	\sigma_s = \left\{
	\begin{array}{ll}
		2(\alpha_{s+1}+\dots + \alpha_r) = 2D_{s+1} - 2D_s
		& \text{ if $k$ is odd} \\
		2(\alpha_{s+1}+\dots + \alpha_{r-2}) + \alpha_{r-1}+\alpha_r = 2D_{s+1} - 2D_s
		& \text{ if $k$ is even}
	\end{array}	\right.
\]
Notice that the Cartan pairing of $M$ does not depend on the parity of $k$.
Also, notice that $\omega(D_i)=\omega_i$ for $i=1,\dots,s+1$.

First we need to classify the covering differences for $M$.
We omit the proof of the following proposition,
which is essentially the same as that of Proposition~\ref{prop: coperture D}.

\begin{proposition}
Let $\gamma\in\mathbb N\Sigma$ be a covering difference in $\mathbb N\Delta$ with $\supp_S(\gamma)=S$.
Then $s$ is even and
$\gamma=\sum_{i=1}^{(j-1)/2}\sigma_{2i-1}+\sum_{i=(j+1)/2}^{s/2}2\sigma_{2i-1}\;+\sigma_s=D_1+D_j-D_{j-1}
$
for $j$ odd $\leqslant s+1$.
\end{proposition}

In particular, it follows that
every covering difference for $M$ satisfies the property (2-ht).
Therefore by Lemma~\ref{lem: proiettiva normalita},
in order to prove the surjectivity of the multiplication of sections
of globally generated line bundles on $M$,
we are reduced to the study of the low fundamental triples.
Here, as in Lemma~\ref{lemma: triple fondamentali piatte D}, we have the following.

\begin{lemma}\label{lemma: triple fondamentali piatte II}
Let $(D_p,D_q,F)$ be a low fundamental triple with $\supp_S(D_p+D_q-F) = S$. 
Then $s$ is even, $p,q$ are odd, $p+q\leqslant s+2$ and $F = D_{p+q-2}$.
\end{lemma}

\begin{proposition}
Let $(D,E,F)$ be a low fundamental triple.
Then $s^{D+E-F}V_F\subset V_D V_E$.
\end{proposition}

\begin{proof}
As in the proof of Theorem~\ref{teo: model}, if $\supp (D+E-F)\neq S$
we can proceed by localization and parabolic induction.
Therefore it is enough to consider the triples of Lemma~\ref{lemma: triple fondamentali piatte II}. 

Denote by $U = \mathbb C^k$ the standard representation of $\Spin(k)$
with the invariant symmetric bilinear form $b$. Let $V \subset U$
be a totally isotropic subspace of dimension $s$,
let $\omega_V\in \mathsf{\Lambda}^2 V$ be a symplectic form on $V$
and fix $e_0 \in U \smallsetminus (V \oplus V^*)$ with $b(e_0,e_0)=1$. 

Let $H$ be the generic stabilizer of $M$.
Then $H$ contains the center of $\Spin(k)$ and its image $\overline{H}$
in $\SO(k)$ is described as follows:
\[
\overline H = \{g\in \SO(k)\, : \, g\, e_0 = e_0, \quad g\,V \subset V, \quad g\,\omega_V \in \mathbb{C}\, \omega_V\}.
\]
If $j \leqslant s+1$, then we have $V_{D_j} = V(\omega_j) = \mathsf{\Lambda}^j U$.
Let $h_j$ be a non-zero $H$-semiinvariant vector in $V_{D_j}$.
If $2j\leqslant s$, then up to a scalar factor we have
$h_{2j} = \omega_V^{\wedge j}$ and $h_{2j+1} = e_0 \wedge h_{2j}$.
The projection $\pi\colon\mathsf{\Lambda}^p U \otimes \mathsf{\Lambda}^q U \longrightarrow \mathsf{\Lambda}^{p+q-2}U$
is given by contraction as follows:
\[ \pi\big((v_1\wedge\dots\wedge v_p)\otimes (u_1\wedge\dots \wedge u_q)\big)= \sum_{i,j}
(-1)^{i+j} b(v_i,u_j) v_1\wedge\dots\wedge \hat v_i\wedge\dots \wedge \hat u_j\wedge \dots \wedge u_q.\]
Therefore, if $(D_p,D_q,D_{p+q-2})$ is as in Lemma~\ref{lemma: triple fondamentali piatte II}, 
we get $\pi(h_p \otimes h_q) = h_{p+q-2}\neq 0$,
and by Lemma~\ref{lemma: supporto moltiplicazione}
it follows that $s^{D_p+D_q-D_{p+q-2}}V_{D_{p+q-2}} \subset V_{D_p} V_{D_q}$.
\end{proof}

If $k = 4n+2m+3$ and $s=2n+1$, then $M = M_\mathcal{O}$ is the wonderful variety
corresponding to the nilpotent orbit $\mathcal{O}$ of the case II,
while if $k = 4n+2m+4$ and $s=2n+1$, then $M = M_\mathcal{O}$ is the wonderful variety
corresponding to the nilpotent orbit $\mathcal{O}$ of the case IV.
By the previous proposition, the multiplication of sections
is surjective for every couple of globally generated line bundles on $M_\mathcal{O}$.
Since $\theta = D_2$ is minuscule in $\mathbb{N}\Delta$,
it follows that $\overline{\mathcal{O}}$ is normal.

\subsection{Case $\mathrm{V}$}

The variety $M_\mathcal{O}$ corresponds to the case (53) in \cite{BCF}. We have
$\Sigma=\{\sigma_1,\sigma_2,\sigma_3\}$ where
\begin{align*}
    \sigma_1 & = \alpha_1 + \alpha_6 =  2 \, D_1 -    D_3, \\
    \sigma_2 & = \alpha_2 + \alpha_4 =       D_2 -    D_3 +      D_4, \\
    \sigma_3 & = \alpha_3 + \alpha_5 =      -D_1 + 2\,D_3 - 2 \, D_4
\end{align*}
and $\omega(D_1)=\omega_1+\omega_6$, $\omega(D_2)=\omega_2$,
$\omega(D_3)=\omega_3+\omega_5$, $\omega(D_4)=\omega_4$.
The covering differences are
\[ \sigma_1,\; \sigma_2, \; \sigma_3, \; \sigma_1+\sigma_3, \; \sigma_2+\sigma_3, \;
2\,\sigma_2+\sigma_3, \; \sigma_1+\sigma_2+\sigma_3, \]
therefore we have $\alt(\sigma^+) = 2$ for every covering difference.
Correspondingly, we get the following low fundamental triples:
\[ (D_1,D_1,D_3),\; (D_2,D_4,D_3),\; (D_3,D_3,D_1+2\,D_4),\; (D_1,D_3,2\,D_4)\]
\[(D_2,D_3,D_1+D_4),\; (D_2,D_2,D_1),\; (D_1,D_2,D_4).\]

We need to show
that $s^{E+D-F}V_F\subset V_D V_E$ for all low fundamental triples
$(D,E,F)$. As in the proof of Theorem~\ref{teo: model}, if $\supp (D+E-F)\neq S$
we can proceed by localization and parabolic induction.
In this way, the triples $(D_1,D_1,D_3)$, $(D_3,D_3,D_1+2\,D_4)$, $(D_1,D_3,2\,D_4)$
reduce to the case of a symmetric wonderful variety,
the triple $(D_2,D_4,D_3)$ reduce to a model wonderful variety
and the triples $(D_2,D_3,D_1+D_4)$, $(D_2,D_2,D_1)$ reduce
to a wonderful variety studied in the case IV.
So we are left to study the case $D=D_1$, $E=D_2$ and $F=D_4$,
which can easily be checked via computer (as explained in Section~\ref{ss: computer}).
Since $\theta=D_2$ is minuscule in $\mathbb{N} \Delta$, it follows that $\overline{\mathcal{O}}$ is normal.

\subsection{Case $\mathrm{X}$}

The variety $M_\mathcal{O}$ corresponds to the case (60) in \cite{BCF}.
We have $\Sigma=\{\sigma_1,\sigma_2,\sigma_3\}$ where 
\begin{align*}
\sigma_1&=2\,\alpha_4= 2\,D_1-D_3; \\
\sigma_2&=\alpha_1 +\alpha_2 =  D_2-D_3+D_4; \\
\sigma_3&=2\,\alpha_3 = -D_1+2\,D_3-2\,D_4;
\end{align*}
and $\omega(D_1)=2\omega_4$, $\omega(D_2)=\omega_2$,
$\omega(D_3)=2\omega_3$, $\omega(D_4)=\omega_1$.
Since the Cartan pairing of $M_\mathcal{O}$ is the same as that of the previous case,
it follows that the covering differences and the low fundamental triples
are also the same.

In order to prove 
that $s^{E+D-F}V_F\subset V_D V_E$ for all low fundamental triples
$(D,E,F)$, if $\supp (D+E-F)\neq S$
we can proceed by localization and parabolic induction.
In this way, the triples $(D_1,D_1,D_3)$, $(D_3,D_3,D_1+2\,D_4)$, $(D_1,D_3,2\,D_4)$
reduce to the case of a symmetric wonderful variety,
the triple $(D_2,D_4,D_3)$ reduce to a model wonderful variety
and the triples $(D_2,D_3,D_1+D_4)$, $(D_2,D_2,D_1)$ reduce
to a wonderful variety studied in the case III.
The only remaining case, $D=D_1$ $E=D_2$ and $F=D_4$,
can easily be checked via computer (as explained in Section~\ref{ss: computer}).
Since $\theta=D_4$ is minuscule in $\mathbb{N} \Delta$, it follows that $\overline{\mathcal{O}}$ is normal.

\section{Application to the real model orbit of type $\mathsf{E}_8$}\label{sec:modelloreale}

In \cite{AHV}, Adams, Huang and Vogan study the model orbit in the Lie
algebra of type $\mathsf{E}_8$. Their study is motivated by the so-called orbit method
to construct representations of reductive Lie groups.  In their case
the group is the complex algebraic group of type
$\mathsf{E}_8$ considered as a real Lie group.  In particular they describe the
decomposition into irreducible modules of the coordinate ring of the
nilpotent orbit of type IX (see Section~\ref{sec:orbite}) and prove it is indeed a model orbit. In the
same paper they also make some conjectures about another orbit which
is the analogue of the model orbit for the split real
form of $\mathsf{E}_8$.

We start with some general preliminaries, we refer to \cite{AHV} and
to \cite{Vogan} for the motivation of these constructions coming from
the representation theory of Lie groups. Let $\tilde G_\mathbb{R}$ be a real
form of a connected and complex algebraic semisimple group $\tilde G$
and let $\sigma$ be the associated Galois involution of $\tilde
G$. There exists a complex algebraic involution $\theta$ of $\tilde G$
which commutes with $\sigma$ such that the subgroup $K_\mathbb{R}$ of points of $\tilde G_\mathbb{R}$
fixed by $\theta$ is a maximal compact subgroup of
$\tilde G_\mathbb{R}$.  Then the subgroup $K$ of points of $\tilde G$ fixed
by $\theta$ is a complexification of $K_\mathbb{R}$. The Lie algebra $\tilde{\mathfrak{g}}$ of $\tilde G$ decomposes as $\mathfrak{k} \oplus \mathfrak{p}$ where $\mathfrak{k}$ is
the Lie algebra of $K$ and $\mathfrak{p}$ is the eigenspace of eigenvalue $-1$
of $\theta$. An analogue of the nilpotent cone $\mathcal{N}$ is defined as
\[ \mathcal{N}_\theta = \mathcal{N} \cap \mathfrak{p}. \] Fix a point $e\in \mathcal{N}_\theta$,
let $\mathcal{O}$ be its $K$-orbit and $K(e)$ its stabilizer. Consider the
multiplicative character $\gamma_e$ of $K(e)$ given by
$\gamma_e(g)=\det(\Ad_g\bigr|_{\mathfrak{k}(e)})\det(\Ad_g\bigr|_\mathfrak{k})^{-1}$. If
$\chi\colon K(e)\longrightarrow \mathbb{C}^*$ is any multiplicative character we can consider
the algebraic line bundle on $\mathcal{O}$ given by $\mathcal V_\chi=K\times_{K(e)}
\mathbb{C}_\chi$. As in \cite{AHV}, the pair $(e,\chi)$ is said to be
\emph{admissible} if $\overline{\mathcal{O}}\smallsetminus \mathcal{O}$ has codimension at least
two in $\overline{\mathcal{O}}$ and $\chi^2(g)=\gamma_{e}(g)$ for all $g$ in the
identity component of $K(e)$. In this paper we need a slightly more
general definition of admissible pair.

Let $G$ be a double covering of $K$ and $G(e)$ be the inverse image of
$K(e)$ in $G$ so that $\mathcal{O}\simeq G/G(e)$. 
Let $\gamma'_{e}$ denote the character of $G(e)$ induced by $\gamma_e$.
Given a character $\chi$ of $G(e)$ we can construct the line bundle $\mathcal V_\chi$ as above.  
We say that the pair
$(e,\chi)$ is admissible if $\overline{\mathcal{O}}\smallsetminus \mathcal{O}$ has codimension at
least two in $\overline{\mathcal{O}}$ and $\chi^2(g)=\gamma'_{e}(g)$ for all $g$ in
the identity component of $G(e)$. Let also $G_\mathbb{R}$ be the inverse
image of $K_\mathbb{R}$ in $G$. Notice that $G$ is the complexification of
$G_\mathbb{R}$. The coverings of $\tilde G_\mathbb{R}$ are in correspondence with the
coverings of $K_\mathbb{R}$ hence there exists a Lie group $\hat G_\mathbb{R}$ which
is a double covering of $\tilde G_\mathbb{R}$ whose maximal compact subgroup is
$G_\mathbb{R}$.

Assume that $R(e,\chi)$ is the irreducible unitary representation of $\hat G_\mathbb{R}$,
attached to an admissible pair $(e,\chi)$ according to the orbit method. Then
the decomposition of the $G_\mathbb{R}$-finite vectors of $R(e,\chi)$ into irreducible
submodules should be equal to the decomposition into irreducible $G$-submodules of the
space of algebraic sections $\Gamma(\mathcal{O},\mathcal V_\chi)$ (see
\cite[Conjecture~2.9]{AHV}).

In \cite{AHV}, Adams, Huang and Vogan analyze the geometric side of two particular cases of this
construction. They obtain a complete description of the decomposition 
of $\Gamma(\mathcal{O},\mathcal V_\chi)$ into $G$-modules in one of these cases
and a conjectural description in the second case. In Section 
\ref{ssec:modellocomplesso} we recall the result obtained by Adams, Huang and Vogan 
in the first case and we analyze it using our techniques. In Sections \ref{ssec:modelloreale1}, \ref{ssec:modelloreale2} and \ref{ssec:modelloreale3} 
we analyze the second case considered by Adams, Huang and Vogan and we prove that their conjectural description of 
$\Gamma(\mathcal{O},\mathcal V_\chi)$ is correct.

\subsection{The case of the complex model orbit}\label{ssec:modellocomplesso}
Let $\tilde G_\mathbb{R}$ be the complex algebraic group of type $\mathsf{E}_8$. Hence $\tilde G=\tilde G_\mathbb{R}\times \tilde G_\mathbb{R}$
and $G=K$ is the diagonal subgroup. So $\mathfrak{p}$ is isomorphic to the
Lie algebra of $G$ and one can consider the nilpotent orbit $\mathcal{O}$ with Kostant-Dynkin
diagram equal to $(0,1,0,0,0,0,0,0)$ (case IX of Section~\ref{sec:orbite}). In this case
$\gamma_e$ is trivial and the stabilizer $K(e)$ is connected, so also $\chi$ has to be trivial and 
$\Gamma(\mathcal{O},\mathcal V_\chi) =\mathbb{C}[\mathcal{O}]$. In \cite{AHV} it is proved that all simple modules of $G$ appears 
in $\mathbb{C}[\mathcal{O}]$ with multiplicity one. 

We fix a maximal torus $T$, and a Borel subgroup of $G$ containing $T$. 
We denote by $\Phi$ the set of roots and by $S$ the
set of simple roots determined by these choices. We denote also by
$\varepsilon_1,\dots,\varepsilon_8$ an orthonormal basis of $\mathcal{X}(T)$ such that
$\Phi$ and $S$ have the following description (with respect to the
choice given in \cite{AHV} we have changed the sign of $\varepsilon_1$):
\begin{align*}
  \Phi   & = A \cup B \text{ where } A = \{\pm \varepsilon_i \pm \varepsilon_j\, : \, i\neq j\} {\text{ and }} \\
  B      & = \{\frac 12 (\sum_{i=1}^8 a_i \varepsilon_i) \, : \, a_i=\pm 1 {\text{ and }} \prod_{i=1}^8 a_i=1\};\\
  S      & = \{\alpha_1,\dots,\alpha_8\} \text{ where } \alpha_1 = -\varepsilon_1-\varepsilon_2, \\
  \alpha_2 & = \frac 12 (\varepsilon_1-\varepsilon_2-\varepsilon_3+\varepsilon_4+\varepsilon_5+\varepsilon_6+\varepsilon_7+\varepsilon_8) {\text{ and }} \\
  \alpha_i & = \varepsilon_{i-1}-\varepsilon_i \text{ for } i=3,\dots,8.
\end{align*}
For each root $\alpha$ choose also an $\mathfrak{sl}(2)$-triple
$x_\alpha,\alpha^\vee,y_\alpha$ where $x_\alpha$ has weight $\alpha$ and
$y_\alpha$ has weight $-\alpha$.  Let $\beta_1=-\varepsilon_1+\varepsilon_2$,
$\beta_2=\varepsilon_3+\varepsilon_4$, $\beta_3=\varepsilon_5+\varepsilon_6$, $\beta_4=\varepsilon_7+\varepsilon_8$
and define
\begin{gather*}
  e_0 = x_{\beta_1} + x_{\beta_2} + x_{\beta_3} + x_{\beta_4}, \quad
  f_0 = y_{\beta_1} + y_{\beta_2} + y_{\beta_3} + y_{\beta_4} \\
  h_0 = \beta_1^\vee + \beta_2^\vee + \beta_3^\vee + \beta_4^\vee = -\varepsilon_1^*+\varepsilon_2^*+\dots+\varepsilon_8^*.
\end{gather*}
These elements are an $\mathfrak{sl}(2)$-triple and it is clear that $e_0$
is an element with associated Kostant-Dynkin diagram equal to
$(0,1,0,0,0,0,0,0)$. Let $\mathcal{O}$ be its orbit. As we have already
recalled in Section~\ref{sec:orbite}, the orbit of the line spanned by
$e_0$ in $\mathbb{P}(\mathfrak{g})$ is the open orbit of the model wonderful variety
of $\mathsf{E}_8$. 

The colors of the model wonderful variety of $\mathsf{E}_8$ 
are $D_1,\dots,D_8$ with $\omega(D_i)=\omega^\alpha_i$, where
$\omega^\alpha_1,\dots,\omega^\alpha_8$ are the fundamental weights
w.r.t.\ the simple system $\alpha_1,\dots,\alpha_8$. Notice also that, in
the notation of the previous section, we have $\overline{\mathcal{O}} = C_{D_8}$.
Finally the spherical roots of $M$ are
\begin{gather*}
  \sigma_1=\alpha_6+\alpha_7, \quad
  \sigma_2=\alpha_4+\alpha_5, \quad
  \sigma_3=\alpha_1+\alpha_3, \quad
  \sigma_4=\alpha_2+\alpha_4, \\
  \sigma_5=\alpha_5+\alpha_6, \quad
  \sigma_6=\alpha_3+\alpha_4, \quad
  \sigma_7=\alpha_7+\alpha_8.
\end{gather*}
Here we have ordered them so that it is clear that they are a basis of a root system of
type $\mathsf{D}_7$. Notice also that, since in this case $\omega$ is injective,
the list above determines also the Cartan pairing $c$ of the
wonderful variety $M$.

\begin{theorem}[see {\cite[Theorem~1.1]{AHV}}]\label{teo:E8}
The variety $\overline{\mathcal{O}}$ is normal and 
\[ \mathbb{C}[\overline{\mathcal{O}}]=\mathbb{C}[\mathcal{O}]\simeq\bigoplus_{\lambda\in \mathcal{X}(T)^+} V(\lambda). \]
\end{theorem}
\begin{proof}
We know that $\overline{\mathcal{O}}$ is normal by the discussion in Section
\ref{sec:orbite}, in particular we have $\overline{\mathcal{O}} = C_{D_8}=\widetilde
C_{D_8}$. Moreover each adjoint orbit has even dimension so the first
equality follows by normality.  Since $\overline{\mathcal{O}} = \widetilde C_{D_8}$ we have also 
\[ 
 \mathbb{C}[\overline{\mathcal{O}}] \simeq 
 \bigoplus_{n\geqslant 0} \Gamma(M,\mathcal L_{nD_8})\simeq 
 \bigoplus_{\substack{ n\geqslant 0, \gamma \in \mathbb{N}\Sigma:\\ 
 nD_8-\gamma \in \mathbb{N}\Delta} } s^{\gamma}V_{nD_8-\gamma}.
\] Now notice that $D_8$ and the spherical roots are linearly
 independent and that $\omega$ is injective in this case, so all 
 irreducible $G$-representations occur with multiplicity at most one.
 Moreover, since the variety is irreducible, if $V(\lambda)$ and $V(\mu)$
 occur in this decomposition then also $V(\lambda+\mu)$
 occurs. Finally, we have
\begin{align*}
 D_1 & 
= 2\,D_8 - ( \sigma_4 + 2\,\sigma_5 + \sigma_6 + 2\,\sigma_7)\\
 D_2 & 
= 4\,D_8 - ( \sigma_1 + 2\,\sigma_2 + 3\,\sigma_3 + 4\,\sigma_4 + 6\,\sigma_5 + 3\,\sigma_6 + 5\,\sigma_7) \\
 D_3 & 
= 5\,D_8 - ( \sigma_1 + 2\,\sigma_2 + 3\,\sigma_3 + 5\,\sigma_4 + 7\,\sigma_5 + 3\,\sigma_6 + 6\,\sigma_7)\\
D_4 & 
= 7\,D_8 - ( \sigma_1 + 2\,\sigma_2 + 4\,\sigma_3 + 6\,\sigma_4 + 9\,\sigma_5 + 4\,\sigma_6 + 8\,\sigma_7)\\
D_5 & 
= 6\,D_8 - ( \sigma_1 + 2\,\sigma_2 + 4\,\sigma_3 + 6\,\sigma_4 + 8\,\sigma_5 + 4\,\sigma_6 + 7\,\sigma_7)\\
D_6 & 
= 4\,D_8 - ( \sigma_2 + 2\,\sigma_3 + 3\,\sigma_4 + 4\,\sigma_5 + 2\,\sigma_6 + 4\,\sigma_7)\\
D_7 & 
= 3\,D_8 - ( \sigma_2 + 2\,\sigma_3 + 3\,\sigma_4 + 4\,\sigma_5 + 2\,\sigma_6 + 3\,\sigma_7)\\
D_8 & 
= D_8 - 0
\end{align*}
hence $V(\omega_i)$ occurs in the decomposition of $\mathbb{C}[\overline{\mathcal{O}}]$, for all $i$. 
\end{proof}

As we have already recalled above the second isomorphism in the statement of the Theorem was already proved 
in \cite{AHV}. Notice that our proof of this isomorphism does not use the projective normality proved in 
Section \ref{sec: low model} but only general considerations on wonderful varieties.

The normality of $\overline{\mathcal{O}}$ was not studied in \cite{AHV}, however Panyushev in the last section of \cite{Pa} sketches  
how to deduce it by their result. Notice that our proof of the normality of $\overline{\mathcal{O}}$, 
relies on Theorem~\ref{teo: model} for which, in the case of $\mathsf{E}_8$, we used a computer. 
The normality of $\overline{\mathcal{O}}$ is due to A.~Broer (see \cite[Section~7.8, Remark iii)]{Bro}), however we could not find a proof in the literature.

\subsection{From complex to real orbits: general considerations}

Let $\tilde G, K, \tilde{\mathfrak{g}}, \mathfrak{p}$ be as at the beginning of this section.
Let $\mathcal{O}\subset \tilde{\mathfrak{g}}$ be a nilpotent adjoint orbit of the complex
algebraic group $\tilde G$.  We want to make some general standard remarks on
the intersection $\mathcal{O}\cap \mathfrak{p}$ (see for example \cite[Section~9]{R}). 
Fix $e\in \mathcal{O}\cap \mathfrak{p}$ and let $\tilde
G(e)$ and $K(e)$ be the stabilizers of $e$ in $\tilde G$ and $K$, respectively.
The subgroup $\tilde G(e)$ is stable under $\theta$ and we define
$Z=\{z\in \tilde G(e)\, : \, \theta(z)=z^{-1}\}$. We have an action of
$\tilde G(e)$ on $Z$ by $g \cdot z = g z \theta(g)^{-1}$ and we define 
$\mathbb{H}^1$ as $Z$ modulo the action of $\tilde G(e)$.

\begin{lemma}\label{lem:Ocapp1}
\begin{enumerate}[\indent i)]
\item Every connected component of $\mathcal{O}\cap \mathfrak{p}$ is a single $K$-orbit;
\item The map $ g\cdot x \mapsto g^{-1}\theta(g)$ induces a bijection from the set of $K$-orbits 
in $\mathcal{O}\cap \mathfrak{p}$ to $\mathbb{H}^1$;
\item Every connected component of $Z$ is a single $\tilde G(e)^0$-orbit
  (where $\tilde G(e)^0$ is the identity component of $\tilde G(e)$).
\end{enumerate}
\end{lemma}
\begin{proof}
Set $\eta = -\theta$. First notice that $\mathcal{O}\cap \mathfrak{p}=\mathcal{O}^\eta$. Being $\mathcal{O}$ smooth and $\eta$ an
involution we deduce that $\mathcal{O}^\eta$ is smooth. Take $e \in \mathcal{O}^\eta$.
We prove that
\[
(\tilde{\mathfrak{g}}\cdot e)\cap \mathfrak{p} = \mathfrak{k} \cdot e.
\] 
Let $\tilde{\mathfrak{g}}(e)$ be the annihilator of $e$ in $\tilde{\mathfrak{g}}$ and let $y\in
\tilde{\mathfrak{g}}$ be such that $y\cdot e \in \mathfrak{p}$. Then $\theta(y\cdot e) = -y \cdot e$, hence
$z=\theta(y)-y \in \tilde{\mathfrak{g}}(e)$. Take $u=\tfrac 12 z$ then $u\in \tilde{\mathfrak{g}}(e)$ is such that
$v=y+u \in \mathfrak{k}$ and $v\cdot e=y\cdot e$.

This proves that any $K$-orbit in $\tilde G \cdot e\cap \mathfrak{p}$ is
open. This implies $i)$.

Point $iii)$ can be proved similarly and $ii)$ is trivial.
\end{proof}

In particular notice that if we prove that $Z$ is connected then
it follows that $\mathcal{O}\cap \mathfrak{p}$ is a single $K$-orbit.

We refine this lemma, using the Jordan decomposition.
Choose an $\mathfrak{sl}(2)$-triple, $e,h,f$ such that $\theta(h)=h$ and
$\theta(f)=-f$, this is always possible, see \cite[Proposition~4]{KR}.
Let $U$ be the unipotent radical of $\tilde G(e)$ and $L=\{g\in \tilde
G(e)\, : \, g\cdot h = h\}$. Then $\tilde G(e) = L \cdot U$ is a Levi
decomposition of $\tilde G(e)$ (see \cite[Proposition~2.4]{BV}). Notice
also that $L$ and $U$ are stable under the action of $\theta$. Define
$Z_L=Z\cap L$ and $\mathbb{H}_L^1$ as $Z_L$ modulo the action of $L$ given by
$g\cdot z= g\,z\,\theta(g)^{-1}$.

\begin{lemma}\label{lem:Ocapp2}
\begin{enumerate}[\indent i)]
  \item The inclusion $Z_L\subset Z$ induces a bijection from $\mathbb{H}^1_L$
    to $\mathbb{H}^1$;
  \item If $Z_L$ is connected then $\mathcal{O}\cap \mathfrak{p}$ is a single $K$-orbit.
\end{enumerate}
\end{lemma}

\begin{proof}
Let $x,y\in Z_L$ and assume that $x = h y \theta(h)^{-1}$ with $h=\ell
u$, $\ell\in L$ and $u\in U$. It follows immediately that $x=\ell y \theta(\ell)^{-1}$.
This prove the injectivity of the map $\mathbb{H}^1_L\longrightarrow \mathbb{H}^1$ induced by the inclusion $Z_L \subset Z$.

Now let $z=\ell u \in Z$ with $\ell\in L$ and $u\in U$. From
$\theta(z)=z^{-1}$ we deduce that $\theta(\ell)=\ell^{-1}$ and
$\theta(u)= \ell u^{-1} \ell^{-1}$. Now being $U$ unipotent there
exists a unique $v\in U$ such that $v^{-2}=u$. Using again that $U$ is
unipotent we deduce that $\theta (v)= \ell v^{-1} \ell^{-1}$. Set
$h'=\ell v \ell^{-1}$ then
\[
h' (\ell u) \theta(h')^{-1}= \ell v \ell^{-1} \ell u \ell^{-1} \ell v \ell^{-1} \ell = \ell.
\]
Hence the map $\mathbb{H}^1_L\longrightarrow \mathbb{H}^1$ is surjective. This proves $i)$. Now $ii)$ follows from
Lemma~\ref{lem:Ocapp1}.
\end{proof}

\subsection{The case of the real model orbit of $\mathsf{E}_8$}\label{ssec:modelloreale1}

In the last two sections of \cite{AHV} a real version of the model
orbit is considered. We now recall from \cite{AHV} some of the
structural results about this orbit. We also prove 
Proposition~\ref{prp:unicaorbita} which is probably well known and is somehow
implicit (even if not necessary) in the discussion in \cite{AHV}.

From now on $\tilde G$ is the complex algebraic group of type $\mathsf{E}_8$ and 
$\tilde G_\mathbb{R}$ is its split real form. Then $K$ is isogenous to 
$\Spin(16)$. 
For the complex group of type $\mathsf{E}_8$ we have already introduced some notations in Section \ref{ssec:modellocomplesso}
(notice that this group had the role of the group $G$ and not of $\tilde G$ in that Section). 
We keep that notation, so $T$ is a maximal torus of $\tilde G$, $\Phi$, the associated root system, 
and $\varepsilon_1,\dots,\varepsilon_8$, $A$, $B$, $S=\{\alpha_1,\dots,\alpha_8\}$ are as in Section \ref{ssec:modellocomplesso}.

Then we can choose $\tilde G_\mathbb{R}$ so that $\mathfrak{k}$
is the Lie algebra spanned by $\mathfrak{t}$, the Lie algebra of $T$, and by
the vectors $x_\alpha$ with $\alpha \in A$. With this choice $\mathfrak{p}$ is spanned by the vectors $x_\alpha$
with $\alpha \in B$. 

As a simple system for the root system of $\mathfrak{k}\simeq \mathfrak{so}(16)$ we choose the
usual basis, but we enumerate it starting from zero, that is
\[ \tau_i=\varepsilon_{i+1}-\varepsilon_{i+2} {\text{ for }}
i=0,\dots,6 {\text{ and }} \tau_7=\varepsilon_7+\varepsilon_8.
\] We denote by $\omega^\mathsf{D}_i$ the associated fundamental weights.
In particular we obtain that $\mathfrak{p}$ is the spin representation
associated to the weight $\omega^\mathsf{D}_7$. Moreover, since $\tilde G$
is simply connected the subgroup $K$ is connected. Finally notice that
$\omega_6^\mathsf{D}\not\in\mathcal{X}(T)$ while
$\omega^\mathsf{D}_7 \in \mathcal{X}(T)$. Hence $\Spin(16)$ is a double covering of $K$. 
We set $G=\Spin(16)$.  Notice that in
\cite{AHV} it is claimed that $K$ is isomorphic to $\Spin(16)$, but this
does not affect any of their arguments.

In order to prove that the roots $\tau_1,\dots,\tau_7$ are conjugated to the
roots $\sigma_1,\dots,\sigma_7$ introduced in Section~\ref{ssec:modellocomplesso}, 
we introduce a new simple system of the root system of type $\mathsf{E}_8$.
The vectors 
\begin{align*}
 \zeta_1 & = \tfrac 12 (-\varepsilon_2-\varepsilon_3-\varepsilon_4+\varepsilon_5),&
 \zeta_2 & = \tfrac 12 (\varepsilon_1+\varepsilon_6+\varepsilon_7-\varepsilon_8),\\
 \zeta_3 & = \tfrac 12 ( \varepsilon_2+\varepsilon_3-\varepsilon_4+\varepsilon_5),&
 \zeta_4 & = \tfrac 12 (\varepsilon_1+\varepsilon_6-\varepsilon_7+\varepsilon_8),\\
 \zeta_5 & = \tfrac 12 (\varepsilon_2-\varepsilon_3+\varepsilon_4+\varepsilon_5),&
 \zeta_6 & = \tfrac 12 (\varepsilon_1-\varepsilon_6+\varepsilon_7+\varepsilon_8),\\
 \zeta_7 & = \tfrac 12 (-\varepsilon_2+\varepsilon_3+\varepsilon_4+\varepsilon_5),&
 \zeta_8 & = \tfrac 12 (\varepsilon_1-\varepsilon_6-\varepsilon_7-\varepsilon_8),\\
\end{align*}
form an orthonormal basis of $\mathfrak{t}^*$, and $\pm \zeta_i\pm \zeta_j\in \Phi$ for
$i$ odd and $j$ even.  Notice also that
\begin{gather*}
 \gamma _2 = \frac 12 (\zeta_1-\zeta_2-\zeta_3+\zeta_4+\zeta_5+\zeta_6+\zeta_7+\zeta_8) =\\
 \qquad \quad = \frac 12 (\varepsilon_1-\varepsilon_2-\varepsilon_3+\varepsilon_4+\varepsilon_5-\varepsilon_6-\varepsilon_7+\varepsilon_8) \in \Phi, \\ 
 \gamma _1 = -\zeta_1-\zeta_2,{\text{ and }} \gamma_i=\zeta_{i-1}-\zeta_i, {\text{ for }} i=3,\dots,8,
\end{gather*}
form a simple system of $\Phi$ (since they are elements of $\Phi$ with the
right scalar products). Hence there exists an element $w$
in the Weyl group such that $w(\varepsilon_i)=\zeta_i$, for all $i$. Finally, notice that 
we have 
\begin{gather*}
  \tau_1=\gamma_6+\gamma_7, \quad
  \tau_2=\gamma_4+\gamma_5, \quad
  \tau_3=\gamma_1+\gamma_3, \quad
  \tau_4=\gamma_2+\gamma_4, \\
  \tau_5=\gamma_5+\gamma_6, \quad
  \tau_6=\gamma_3+\gamma_4, \quad
  \tau_7=\gamma_7+\gamma_8 
\end{gather*}
so that $w(\sigma_i)=\tau_i$. Notice also that
\[
\tau_0 = \frac 12 (\zeta_1-\zeta_3-\zeta_5+\zeta_7+\zeta_2+\zeta_4+\zeta_6+\zeta_8) = \gamma_2+\gamma_3+\gamma_4+\gamma_5.
\] We denote by $\Lambda$ the weight lattice of $\Spin(16)$ and by
$\Lambda_\mathsf{E}=\mathcal{X}(T)$ the sublattice given by the weights of
$\mathsf{E}_8$.
We denote also by $\Lambda^+$ the dominant weights of
$\Lambda$ w.r.t.\ the simple system $\tau_0,\dots,\tau_7$ and by
$\Lambda_\mathsf{E}^+$ the dominant weights of $\Lambda_\mathsf{E}$ w.r.t.\ the simple
system $\gamma_1,\dots,\gamma_8$. We have $\Lambda^+_\mathsf{E}\subset
\Lambda^+$.  We denote by $\omega^\gamma_i$ the fundamental weights of
$\Lambda_\mathsf{E}$ w.r.t.\ the simple system $\gamma_1,\dots,\gamma_8$. Notice that
$\omega^\mathsf{D}_7=\omega^\gamma_8$.

Let now $\mathcal{O}$ be the adjoint orbit of $\tilde{\mathfrak{g}}$ considered in
Section \ref{ssec:modellocomplesso}. 
We want to study the intersection
$\mathcal{O}\cap \mathfrak{p}$. Let us first choose an element $e \in \mathcal{O}\cap \mathfrak{p}$. 

Consider the semisimple part $K_0$ of the standard Levi factor of the
maximal parabolic subgroup $P_0$ of $K$ associated with $\tau_0$. This is a group
isogeneous to $\Spin(14)$, we denote its Lie algebra by
$\mathfrak{k}_0$. Let $\mathfrak{p}_0$ be the subspace of $\mathfrak{p}$ spanned by the root
vectors of weight of the form $\tfrac 12(\sum a_i\varepsilon_i)$ with
$a_1=-1$. Its highest weight is
$\tfrac12(-\varepsilon_1+\varepsilon_2+\varepsilon_3+\varepsilon_4+\varepsilon_5+\varepsilon_6+\varepsilon_7-\varepsilon_8)$. Hence
it is an irreducible $\Spin(14)$-submodule of $\mathfrak{p}$ isomorphic to the
module $V_{D_8}^*$ of Section~\ref{ssec:comodelloh8}.  

\begin{lemma}\label{lem:hE}
The vector $h_8$ of Section~\ref{ssec:comodelloh8} belongs to $\mathcal{O}$. 
\end{lemma}
\begin{proof}
The vectors
\begin{subequations}\label{eq:u}
\begin{align}
 \eta_1 & = \tfrac 12 ( \varepsilon_1 - \varepsilon_6 - \varepsilon_7 - \varepsilon_8),&
 \eta_2 & = \tfrac 12 (-\varepsilon_2 + \varepsilon_3+\varepsilon_4+\varepsilon_5),\\
 \eta_3 & = \tfrac 12 ( - \varepsilon_1 + \varepsilon_6-\varepsilon_7-\varepsilon_8),&
 \eta_4 & = \tfrac 12 ( \varepsilon_2 - \varepsilon_3+\varepsilon_4+\varepsilon_5),\\
 \eta_5 & = \tfrac 12 ( - \varepsilon_1 - \varepsilon_6+\varepsilon_7-\varepsilon_8),&
 \eta_6 & = \tfrac 12 ( \varepsilon_2 + \varepsilon_3-\varepsilon_4+\varepsilon_5),\\
 \eta_7 & = \tfrac 12 ( - \varepsilon_1 - \varepsilon_6-\varepsilon_7+\varepsilon_8),&
 \eta_8 & = \tfrac 12 ( \varepsilon_2 + \varepsilon_3+\varepsilon_4-\varepsilon_5).
\end{align}
\end{subequations}
form an orthonormal basis of $\mathfrak{t}^*$, and $\pm \eta_i\pm \eta_j\in \Phi$ if
$i$ is odd and $j$ is even.  Notice also that
\begin{gather*}
 \hat \gamma_2 = \frac 12 (\eta_1-\eta_2-\eta_3+\eta_4+\eta_5+\eta_6+\eta_7+\eta_8)= \varepsilon_2-\varepsilon_6, \\
 \hat \gamma_1 = -\eta_1-\eta_2,{\text{ and }} \hat \gamma_i=\eta_{i-1}-\eta_i, {\text{ for }} i=3,\dots,8,
\end{gather*}
is a simple system of $\Phi$. Hence there exists an element
$w$ of the Weyl group such that $w(\varepsilon_i)=\eta_i$, for all $i$. Choose a
representative $\dot w$ of $w$ in $\tilde G$. Define $e =\dot w (e_0)$,
$h=w(h_0)=-\eta_1^*+\eta_2^*+\dots+\eta_8^*$ and $f=\dot w(f_0)$. Then $\theta(h)=h$
and $\theta(f)=-f$.  More explicitly, we have
$h=-2\varepsilon^*_1+\varepsilon^*_2+\varepsilon^*_3+\varepsilon^*_4+\varepsilon^*_5$,
$e=x_{\delta_1}+x_{\delta_2}+x_{\delta_3}+x_{\delta_4}$, where
$w(\beta_1)=\delta_1=-\eta_1+\eta_2$, $w(\beta_2)=\delta_2=\eta_3+\eta_4$,
$w(\beta_3)=\delta_3=\eta_5+\eta_6$, $w(\beta_4)=\delta_4=\eta_7+\eta_8$ and similarly
for $f$.

Notice that $e \in \mathfrak{p}_0$. Moreover, the two vectors $e$ and $h_8$ are linear
combinations with nonzero coefficients of vectors of the same weights
$\delta_1,\delta_2,\delta_3,\delta_4$ and these weights are linearly
independent, so they are conjugated under the action of the
maximal torus. 
\end{proof}

Let $w \in W$ be the element defined in the previous proof,
as shown there we can choose a representative $\dot w$ of
$w$ such that $\dot w (e_0)=h_8$. Set $e=h_8$ and $h=w(h_0)$.  In
particular the stabilizer of $e$ is the parabolic induction of the
stabilizer of $h_8$ in $\Spin(14)$. More explicitly, we have
\begin{equation}\label{eq:ke} 
\mathfrak{k}(e) = \mathbb{C}\,\omega_0^\vee \oplus \mathfrak{u}^-_0 \oplus \mathfrak{h}_0 
\end{equation}
where
$\mathfrak{u}^-_0$ is the Lie algebra of the unipotent radical of the
parabolic subgroup opposite to $P_0$, $\mathfrak{h}_0$ is the annihilator of $h_8$ in
$\mathfrak{k}_0$ and $\omega_0^\vee$ is orthogonal to
$\tau_1,\dots,\tau_7$. In particular the Levi factor of $\mathfrak{k}(e)$ is
isomorphic to $\mathfrak{gl}(4)$.

We now want to describe in some detail the stabilizer $K(e)$ and
apply Lemma~\ref{lem:Ocapp2} to prove that $\mathcal{O}\cap \mathfrak{p}$ is a single
$K$-orbit. As recalled in Section~\ref{ssec:modellocomplesso},
the Levi factor $\tilde L$ of $\tilde G(e)$ is $\Sp(8)$ and $L=\tilde
L^\theta$. Furthermore, notice that there is only one involution of $\Sp(8)$
such that the Lie algebra of the fixed point subgroup is
isomorphic to $\mathfrak{gl}(4)$. This involution can be described as
follows. Let $I_4 \in \mathrm{SL}(4)$ be the identity matrix,
and define the $8\times 8$ matrices
\[ J = \begin{pmatrix} 0 & I_4 \\ -I_4 & 0 \end{pmatrix},
\qquad \rho = \begin{pmatrix} I_4 & 0 \\ 0 & -I_4 \end{pmatrix} \] 
and $\Sp(8)$ as the matrices preserving the form $J$. Then $\Sp(8)$
is stable under the conjugation by the matrix $\rho$, which is an involution.
Moreover $L$ is isomorphic to $\GL(4)$.

\begin{proposition}\label{prp:unicaorbita}
If $\mathcal{O}$ is the model orbit of $\mathsf{E}_8$ and $\tilde G_\mathbb{R}$ is the split real
form of $\mathsf{E}_8$, then $\mathcal{O}\cap \mathfrak{p} $ is a single $K$-orbit.
\end{proposition}
\begin{proof}
By Lemma~\ref{lem:Ocapp2} and the above discussion, it is enough to prove that
\[
Z_L=\{z\in \Sp(8)\, : \, \theta(z)=z^{-1}\}
\]
is connected. Let $z=\left(\begin{smallmatrix}A & B \\ C & D \end{smallmatrix}\right)$
where $A,B,C,D$ are $4\times 4$ matrices. The condition $z\in Z_L$ is equivalent to 
\begin{equation}\label{eq:unicaorbita}
A^2=I_4+BC, \quad D=A^t, \quad B=B^t, \quad C=C^t, \quad AB=BA^t {\text{ and }}
A^t\,C=CA.
\end{equation}
Acting by $g\in\Sp(8)$, via $g\cdot z = g z \theta(g)^{-1}$, 
we remain in the same connected component. 
Using $g$ of the form
$\left(
\begin{smallmatrix} \alpha & 0 \\ 0 & ^t \alpha^{-1} \end{smallmatrix}
\right) $ with $\alpha \in \GL(4)$ we see we can assume that $B$ is of
the form $\left(\begin{smallmatrix}I_r & 0 \\ 0 &
  0 \end{smallmatrix}\right)$ and $A =\left(\begin{smallmatrix}a & b
  \\ c & d \end{smallmatrix}\right)$, and using equations
\eqref{eq:unicaorbita} we see that $c=0$, $a=a^t$ and $d^2=1$. Then
using $g$ of the form $\left(\begin{smallmatrix} \beta & 0 \\ 0 & ^t
  \beta^{-1} \end{smallmatrix} \right) $ and
$\beta=\left(\begin{smallmatrix} I_r & 0 \\ 0 &
  \gamma \end{smallmatrix}\right)$ we can also assume that $d$ is
diagonal. Now we
choose $g$ of the form $\left(\begin{smallmatrix}I & sI \\ 0 &
  I \end{smallmatrix}\right)$. We get
\[
g\cdot z = 
\begin{pmatrix}
A + s C & B+ s(A+A^t)+s^2C \\
C & A^t+sC
\end{pmatrix}.
\] If we compute the determinant of $B+ s(A+A^t)+s^2C$ we see that it
is a polynomial in $s$ and its lowest degree term is $4\det
(d)$. Hence there exists $s$ such that $B+ s(A+A^t)+s^2C$ is
invertible. So we can assume $B$ invertible and arguing as before we
can assume $B=I_4$. Now, for $B=I_4$, the equations in \eqref{eq:unicaorbita}
take the form $A=A^t=D$, and $C=-I+A^2$. Such equations define an algebraic subset 
which is isomorphic to an affine space and, in particular, connected. 
Therefore $Z_L$ is connected.
\end{proof}

\subsection{The coordinate ring of the real model orbit}\label{ssec:modelloreale2}

Here we describe the coordinate ring of $\mathcal{O} \cap \mathfrak{p}$. In \cite{AHV} it is
shown that this description follows from the vanishing of certain cohomology groups which in
the case of the real model orbit is conjectural (see Conjecture~3.13
and Theorem~7.13 in \cite{AHV}).
 
Let $M_0$ be the wonderful comodel variety of cotype
$\mathsf{E}_8$. This is a wonderful variety for the group
$\Spin(14)$. Consider the parabolic induction $M$ of $M_0$ to $\mathrm{Spin}(16)$. This is a wonderful
variety with spherical roots equal to $\tau_1,\dots,\tau_7$ and
colors $D'_0,\dots,D'_8$, where $\omega(D_0')=\omega^\mathsf{D}_0$ and,
for $i>0$, $D'_i$ is induced by the respective color of $M_0$ (see Section~\ref{sec:comodello}).

Let $\mathcal{O}_\mathfrak{p}=\mathcal{O}\cap \mathfrak{p}$. Then $X_\mathfrak{p}=\mathbb{P}(\mathcal{O}_\mathfrak{p})\subset\mathbb{P}(\mathfrak{p})$ is the open orbit of $M$,
and $\overline{\mathcal{O}_\mathfrak{p}}$ is the cone $C_D$ for $D=D'_8$. Indeed $e=h_8\in V_{\omega(D'_8)}^*=V^*
_{\omega^\mathsf{D}_7}=V _{\omega^\mathsf{D}_7}\simeq \mathfrak{p}$ and $D'_8$ is faithful.

\begin{theorem}\label{teo:E8R}
The cone $\overline{\mathcal{O}_\mathfrak{p}}$ is normal and we have the following
isomorphism of $K$-modules
\[ 
\mathbb{C}[\overline{\mathcal{O}_\mathfrak{p}}]\simeq \bigoplus_{\lambda \in \Lambda^+_\mathsf{E}}V(\lambda).
\]
\end{theorem}

\begin{proof}
Notice that the combinatorics of colors and spherical roots is essentially
the same as that of the model wonderful variety of type $\mathsf{E}_8$.  In
particular, since $D'_8$ is minuscule, from Theorem~\ref{teo: comodello-proj-norm} 
it follows that $\overline{\mathcal{O}_\mathfrak{p}}$ is normal. 
Hence, as a $K$-module we have that its coordinate ring is the sum
of all the modules $V_{nD'_8-\tau}$ with $n\geqslant 0$, $\tau \in
\mathbb{N}[\tau_1,\dots,\tau_7]$ and $nD'_8-\tau\in \mathbb{N}[D'_0,\dots,D'_8]$.
Moreover, with $\tau$ as above, the latter condition is equivalent to $nD'_8-\tau\in \mathbb{N}[D'_1,\dots,D'_8]$.
Now notice that $\omega(D_8')$ and $\tau_1,\dots,\tau_7$ are linearly
independent (this is not true for $M_0$ but it
is true for $M$ because of the presence of the extra color $D'_0$). So
we obtain that the coordinate ring of $\overline{\mathcal{O}_\mathfrak{p}}$ is the sum of
all modules $V(n\omega^\mathsf{D}_8-\tau)$ where $n$ and $\tau$ are as
above. Finally, the computation is exactly the same as that given 
in the proof of Theorem~\ref{teo:E8} for the model orbit of type $\mathsf{E}_8$, 
since the two situations are conjugated by an element of the Weyl group.
\end{proof}

Notice that our proof of the normality of $\overline {\mathcal{O}_\mathfrak{p}}$
via Theorem~\ref{teo: comodello-proj-norm} 
did not require any computer calculation. 
Moreover, the description of the coordinate ring of the normalization of $\overline {\mathcal{O}_\mathfrak{p}}$ 
is independent of Theorem~\ref{teo: comodello-proj-norm}.

The combinatorics of distinguished subsets of colors 
allows to describe completely the $K$-orbits in the closure of $\mathcal{O}_\mathfrak{p}$
(see for example \cite{Ga}), and in particular to prove that
$\overline {\mathcal{O}_\mathfrak{p}}\smallsetminus \mathcal{O}_\mathfrak{p}$ has
codimension at least two in $\overline {\mathcal{O}_\mathfrak{p}}$. Indeed, one sees that this property depends
only on the combinatorics of colors and spherical roots, 
and in this case the combinatorics is the same as that of  
the complex model orbit, whose boundary has codimension at least two.
More precisely, here $\mathrm{codim}_{\overline{\mathcal{O}}}(\overline{\mathcal{O}} \smallsetminus \mathcal{O})=16$ 
and $\mathrm{codim}_{\overline{\mathcal{O}_\mathfrak{p}}}(\overline{\mathcal{O}_\mathfrak{p}}\smallsetminus \mathcal{O}_\mathfrak{p})=8$.

Here we can avoid such an argument.
Below we prove that $\mathbb{C}[\mathcal{O}_\mathfrak{p}]= \mathbb{C}[\overline {\mathcal{O}_\mathfrak{p}}]$ and
this also implies that $\overline {\mathcal{O}_\mathfrak{p}}\smallsetminus \mathcal{O}_\mathfrak{p}$ has
codimension at least two. 

\subsection{Computation of the space of sections of a line bundle associated to an admissible pair: general considerations}

For the next lemma we put ourself in a general setting. Let $G$ be simply connected and let
$M$ be a wonderful compactification of $G/H$.  Let $E\in \mathbb{N}\Delta$.
Let $\mathbb{C}\, h$ be a line in $V(\omega_E^*)$ where $H$ acts with
character $\xi_E$. Assume that the stabilizer of $\mathbb{C}\,h$ is equal to 
$H$ and let $H_0$ be the stabilizer of $h$. 
Furthermore, assume that $\xi_E$ induces an isomorphism of $H/H_0$
with $\mathbb{C}^*$: we identify $H/H_0$ with $\mathbb{C}^*$ choosing such an 
isomorphism. Finally, notice that $H/H_0$ acts on the right on $G/H_0$.

\begin{lemma}\label{lemma:0}Let $D\in \mathbb{Z} \Delta$ and let $\chi$ denote the restriction of $\xi_D$ to $H_0$.
We have the following isomorphism of $G$-modules
\[ \Gamma(G/H_0,\mathcal V_\chi) \simeq 
\bigoplus_{\substack{n\in\mathbb{Z}, \sigma \in \mathbb{Z}\Sigma :\\ D+nE-\sigma \in \mathbb{N}\Delta}} V_{D+nE-\sigma}. \] 
\end{lemma}

\begin{proof} We have \[ \Gamma(G/H_0,\mathcal V_\chi)\simeq (\mathbb{C}[G]\otimes \mathbb{C}_{-\chi})^{H_0}\simeq \bigoplus_{\lambda \in \mathcal{X}(T)^+}
  V(\lambda^*) \otimes (V(\lambda)\otimes \mathbb{C}_{-\chi} )^{H_0}.\] Now notice
  that $H$ acts on $(V(\lambda)\otimes \mathbb{C}_{-\chi} )^{H_0}$ and so the latter
  decomposes into $H$-stable subspaces where $H$ acts by a character of
  the form $\xi_D+n\xi_E$. Moreover, all these eigenspaces are of
  dimension one, since $H$ is spherical. Hence we have
\[ \Gamma(G/H_0,\mathcal V_\chi) \simeq 
\bigoplus_{\lambda\in\mathcal{X}(T)^+, n\in \mathbb{Z}}V(\lambda^*) \otimes (V(\lambda)\otimes \mathbb{C}_{-\xi_D-n\xi_E} )^{H}.\]

Now recall that the space of spherical vectors $V(\omega_F^*)^{(H)}_{\xi_F}$ is nonzero if and only if $F\in\mathbb{N}\Delta$,
and $\xi_F=\xi_{D+nE}$ if and only if $F=D+nE-\sigma$ for $\sigma\in\mathbb{Z}\Sigma$.
\end{proof}

Let us specialize this identity to the case $D=0$ and compare the
coordinate ring of $G/H_0$ with the coordinate ring of the
normalization of its closure in $V(\omega_E^*)$. Since
\[
\mathbb{C}[\widetilde C_E] = 
\bigoplus _{n\geqslant 0 }\Gamma(M,\mathcal L_{nE})
\simeq \bigoplus _{\substack{n\geqslant 0, \sigma \in \mathbb{N}\Sigma :\\ nE-\sigma \in \mathbb{N}\Delta}} V_{nE-\sigma},
\] 
we have the following.
\begin{lemma}\label{lemma:A}
The equality $\mathbb{C}[G/H_0] = \mathbb{C}[\widetilde C_E]$
 is equivalent to the fact that, for all $n\in \mathbb{Z} $ and $\sigma \in
\mathbb{Z}\Sigma$, if $nE-\sigma\in \mathbb{N}\Delta$ then $n\geqslant 0$ and $\sigma \in
\mathbb{N}\Sigma$.
\end{lemma} 

We now apply this lemma to our special case in which $G$ is
$\Spin(16)$, $H_0=G(e)$, $M$ is the wonderful compactification of
$X_\mathfrak{p}$ and $E=D'_8$.

\begin{lemma}\label{lemma:B}
For all $n\in\mathbb{Z}$ and $\sigma\in\mathbb{Z}\Sigma$, if $nD'_8-\sigma\in\mathbb{N}\Delta$ 
then $n\geqslant0$ and $\sigma\in\mathbb{N}\Sigma$.
\end{lemma}
\begin{proof}
Write $nD'_8-\sigma=\sum_{i=0}^8 a_iD'_i$ and notice that the
conditions $a_1,\dots,a_8\geqslant 0$ are equivalent to the conditions
obtained from Theorem~\ref{teo:E8} 
by applying Lemma~\ref{lemma:A} to the complex model orbit of type $\mathsf{E}_8$. 
\end{proof}

From this and Theorem~\ref{teo:E8R}, by applying Lemma~\ref{lemma:A} to the real model orbit, we get that
\[
\mathbb{C}[\mathcal{O}_\mathfrak{p}]=\mathbb{C}[\overline{\mathcal{O}_\mathfrak{p}}].
\]

\subsection{Sections of the line bundle associated to the admissible pair for the real model orbit}\label{ssec:modelloreale3}

We now want to compute the characters $\gamma_e$ and $\gamma'_e$. For this
we further analyze the stabilizer $T(e)$ of $e$ in $T$. Let
$\Lambda^\vee_\mathsf{E}$ be the lattice dual to $\Lambda_\mathsf{E}$. 
Recall the definition of the vectors $\eta_i$ given in the equations \eqref{eq:u} 
and denote by $\eta^*_i$ the vectors of the 
dual basis. Let $x_1^*=\eta^*_1+\eta^*_2$ and $x^*_i=\eta^*_{2i-1}-\eta^*_{2i}$
for $i=2,3,4$ and define $R^\vee$ as the sublattice of
$\Lambda^\vee_\mathsf{E}$ generated by $x^*_1,\dots,x^*_4$ and $R^\vee_0$ as
the sublattice generated by $x^*_1+x^*_2, x^*_1+x^*_3, x^*_1+x^*_4$.
Then $R^\vee$ and $R^\vee_0$ are direct factors of
$\Lambda_\mathsf{E}^\vee$. Finally, let $T_0$ be the maximal torus of $K_0$
(the subgroup of $K$ isogenous to $\Spin(14)$ introduced above) 
contained in $T$ such that $T_0(e)=T(e)\cap T_0$. We have
\[ 
T = \Lambda_\mathsf{E}^\vee \otimes _\mathbb{Z} \mathbb{C}^*, \qquad T(e) = R^\vee \otimes _\mathbb{Z}
\mathbb{C}^* \quad {\text{ and }} \quad T_0(e) = R^\vee_0 \otimes _\mathbb{Z} \mathbb{C}^*.
\] 

We already know that the Levi factor of $K(e)$ is isomorphic to
$\GL(4)$ so that any character is a power of the determinant. We
describe now the center of $\GL(4)$ and the determinant as an element
of dual lattice of $R^\vee$. Let
$z^*=x^*_1-x^*_2-x^*_3-x^*_4=2\varepsilon^*_1+\varepsilon^*_2+\varepsilon^*_3+\varepsilon^*_4+\varepsilon^*_5$,
then using the description of $\mathfrak{h}_0$ given in Section~\ref{ssec:cotipoE8} 
it is easy to see that $z^*$ is a central
cocharacter. In particular, if $R^\vee_Z=\mathbb{Z} z^*$ then
\[
T_Z= R_Z^\vee\otimes_\mathbb{Z} \mathbb{C}^*
\] is the center of $L\simeq \GL(4)$. Now we compute the character
$\gamma_e$. We have already noticed that it is enough to compute its restriction to $T$. 
Moreover its restriction to 
$T_0(e)$ must be trivial since $\SL(4)$ has no characters so we compute
its restriction to
$T_Z$. Using the description of the stabilizer $\mathfrak{k}(e)$ in \eqref{eq:ke}
and the description of $\mathfrak{h}_0$ given in Section~\ref{ssec:cotipoE8} it is easy to prove that 
$\langle \gamma_e, z^*\rangle=4$. Hence we see that $\gamma_e$ restricted to the Levi of $K(e)$, which
is isomorphic to $\GL(4)$, equals the determinant and
\[
\gamma_e=x_1-x_2-x_3-x_4.
\]
In particular it is not
the square of a character. However, on the covering $G$ of $K$ we can
consider the character $-\omega_0^\mathsf{D}=-\varepsilon_1$, and we denote by $\chi$ its
restriction to $G(e)$. We have 
\[\chi=\tfrac12 (x_1-x_2-x_3-x_4)=\tfrac 12 \gamma'_e.\] Notice that
$\xi_{D'_0}=\chi$. Indeed, $\omega_{D'_0}=\omega_0^\mathsf{D}$ and $G(e)$ 
contains the unipotent radical $U_0^-$ of $P_0^-$. Hence the only
$G(e)$-semiinvariant in $V(\omega_0^\mathsf{D})$ is the lowest weight vector
which has weight $-\omega_0^\mathsf{D}$.

We now describe the space of sections $\Gamma(\mathcal{O}_\mathfrak{p},\mathcal V_\chi)$. In
\cite{AHV} the same description follows from a vanishing result which
is not proved in this case (see Conjecture 8.6 and Theorem~8.10 in
\cite{AHV}).

\begin{theorem}\label{teo:E8chi}  
We have the following isomorphism of $G$-modules
\[ \Gamma(\mathcal{O}_\mathfrak{p},\mathcal V_\chi)\simeq\bigoplus_{\lambda\in \Lambda^+_\mathsf{E}} V(\omega^\mathsf{D}_0+\lambda). \]
\end{theorem}
\begin{proof}
We have seen above that $\xi_{D'_0}=\chi$. Hence we can apply Lemma~\ref{lemma:0} 
with $D=D'_0$ and $E=D'_8$. We obtain
\[
\Gamma(\mathcal{O}_\mathfrak{p},\mathcal V_\chi)\simeq
\bigoplus_{\substack{n\in \mathbb{Z},\sigma\in \mathbb{Z}\Sigma:\\ D'_0+nD'_8-\sigma\in \mathbb{N}\Delta}} 
V_{D'_0+nD'_8-\sigma} .
\] 
Write $D'_0+nD'_8-\sigma=\sum_{i=0}^8 a_i D'_i$. Notice that as in the
proof of Lemma~\ref{lemma:B} this implies $a_1,\dots,a_8\geqslant0$. 
Hence we obtain $n\geqslant 0$ and $\sigma\in \mathbb{N}\Sigma$.  
In particular the condition $a_0\geqslant 0$ is automatically satisfied. 
As already noticed in the proof of Theorem~\ref{teo:E8R},
$\Lambda^+_\mathsf{E}$ consists of the weights of the form 
$\omega(nD'_8-\sigma)$ for $n\geqslant0$, $\sigma\in\mathbb{N}\Sigma$ and $nD'_8-\sigma\in\mathbb{N}\Delta$.
\end{proof}

\section{Degeneracy of the multiplication}

Here we give a counterexample to the surjectivity of the multiplication of sections of line bundles, on wonderful varieties, generated by global sections. 

\subsection{Counterexample}\label{subsect: non projnorm}

Let $G$ be $\SO(2r+1)$ and let $M$ be the model wonderful
variety for the group $G$. It is not isomorphic to
the model wonderful variety for $\Spin(2r+1)$. Its set of
colors $\Delta=\{D_1,\dots,D_r\}$ is in bijection, via $\omega$, with
the set $\{\omega_1,\dots,\omega_{r-1},2\omega_r\}$. Its set
of spherical roots is
$\{\alpha_1+\alpha_2,\dots,\alpha_{r-1}+\alpha_r,2\alpha_r\}$. If
$r=2$ or $r=3$, then the multiplication of sections is surjective. Set
$r=4$ and consider the low triple $(D_2,D_2,D_1)$: then
$V(\omega_1)\not\subset V(\omega_2)^{\otimes 2}$, hence $s^{2D_2-D_1}
V_{D_1}\not\subset V_{D_2}^2$. In particular, the multiplication of
sections is not surjective, for all $r\geqslant 4$.

\subsection{Degeneracy of the multiplication}

Roughly speaking, in its nature the previous example does not
express a lack of the multiplication, but rather a lack of the
tensor product. Indeed $V(\omega_1) \not \subset V(\omega_2)^{\otimes
  2}$ but $V(2\omega_1) \subset V(2 \omega_2)^{\otimes 2}$, so that it
expresses the fact that the saturation property does not hold for
groups of type $\mathsf{B}$: similar things cannot happen if $G$ is of type
$\mathsf{A}$ and, conjecturally, whenever $G$ is simply laced. It is worth
noticing that the same situation holds if we consider the
multiplication of the wonderful variety considered in previous
example: $s^{2D_2-D_1} V_{D_1}\not\subset V_{D_2}^2$ but
$s^{4D_2-2D_1} V_{2D_1} \subset V_{2D_2}^2$.

We briefly recall what the saturation property is. Let $G$ be a simply
connected almost simple algebraic group. We say that the \textit{saturation
  property} holds for $G$ if, whenever $d >0$ and $\lambda, \mu, \nu \in
\Lambda^+$ are such that $\nu \leqslant \lambda + \mu$ and $V(d\nu) \subset
V(d\lambda) \otimes V(d \mu)$, then it holds also $V(\nu) \subset V(\lambda)
\otimes V(\mu)$. In \cite{KT} A.~Knutson and T.~Tao showed that the saturation
property holds if $G$ is of type $\mathsf{A}$, while in \cite{KM} M.~Kapovich and J.~Millson
conjectured that the saturation property holds whenever $G$ is simply laced.

We want to consider the saturation property in the more general
context of the multiplication law attached to a wonderful variety, the
classical case corresponding to the wonderful compactification of an
adjoint group (\cite[Lemma~3.1]{Ka}). We say that the \textit{saturation property} holds for
a wonderful variety $M$ with set of colors $\Delta$ and set of spherical
roots $\Sigma$ if, whenever $d>0$ and $D,E,F \in \mathbb{N} \Delta$ are such
that $D \leqslant_\Sigma E+F$ and $s^{d(E+F-D)} V_{dD} \subset V_{dE}V_{dF}$,
then it holds also the inclusion $s^{E+F-D} V_D \subset V_E V_F$.

Suppose that $M$ is a wonderful variety and let $E,F \in
\mathbb{N}\Delta$. Then the following inclusion holds
\[
	V_E V_F \subset \bigoplus_{\substack{D \in \mathbb{N}\Delta \, : \, D \leqslant_\Sigma E+F\\ 
                                             V(\omega_D) \subset V(\omega_E) \otimes V(\omega_F)}} 
        s^{E+F-D}V_D.
\]
If the equality holds in the previous inclusion, then we say that the
product $V_E V_F$ is \textit{non-degenerate}, while if the
equality holds for every $E,F \in \mathbb{N}\Delta$ then we say that
the multiplication of $M$ is \textit{non-degenerate}. It is
easy to show that if the multiplication of $M$ is
non-degenerate and if the saturation property holds for $G$,
then the latter holds for $M$ as well.

Another equivalent description of the multiplication follows
by identifying sections of line bundles on $M$ with $H$-semiinvariant
functions on $G$, $H$ acting on the right.
More precisely, given $E,F \in \mathbb{N}\Delta$,
we may identify $\Gamma(M,\mathcal L_E)$, $\Gamma(M,\mathcal L_F)$ and $\Gamma(M,\mathcal L_{E+F})$
with submodules of $\mathbb{C}[G]^{(H)}$, and the multiplication of sections
$\Gamma(M,\mathcal L_E) \otimes \Gamma(M,\mathcal L_F) \longrightarrow \Gamma(M,\mathcal L_{E+F})$
coincides with the restriction of the multiplication defined in $\mathbb{C}[G]^{(H)}$.
Therefore we regard the modules $V_E$, $V_F$, $V_E V_F$ inside $\mathbb{C}[G]^{(H)}$,
and the multiplication is the multiplication of functions.

\begin{example}
Let $G = \SL(2)$ and consider the basic case of
the rank one wonderful variety $M = \mathbb{P}^1 \times \mathbb{P}^1$,
whose generic stabilizer is the maximal torus $T$.
Then $\Sigma = \{\alpha\}$ is the unique simple root of $G$
and $\Delta = \{D^+, D^-\}$, where $D^+ = \mathbb{P}^1 \times \{[1,0]\}$
and $D^- = \{[1,0]\} \times \mathbb{P}^1$.
Moreover $\alpha = D^+ + D^-$ and
$\omega_{D^+} = \omega_{D^-} = 1$ equal the fundamental weight of $G$.

Given $n \in \mathbb{N}$, identify the simple $G$-module $V(n)$ of highest weight $n$
with $\mathbb{C}[x,y]_n$, the set of homogeneous polynomials of degree $n$ in two variables.
Given $aD^+ +bD^- \in \mathbb{N}\Delta$, the corresponding $T$-eigenvector
is $h(a,b) = x^a y^b$.

The projection $\pi_m \colon V(a+b) \otimes V(c+d) \longrightarrow V(a+b+c+d-2m)$ 
is described on the basis of $T$-eigenvectors as follows:
\[x^a y^b \otimes x^c y^d  \longmapsto  
\sum_{u+v=m} (-1)^u u! v! {a \choose u} {b \choose v} {c \choose v} {d \choose u} x^{a+c-m} y^{b+d-m}. \]
As shown by the following two examples, the multiplication of $M$ is degenerate.
\begin{itemize}
	\item[-] Consider $D^+ + D^- \leqslant_\Sigma  2(D^+ + D^-)$.
		Then $V(2) \subset V(2)^{\otimes 2}$ but $\pi_1(h(1,1) \otimes h(1,1)) = 0$,
		so that	$V_{D^+ +D^-} \not \subset V_{D^+ +D^-}^2$.
		This can also be explained since $V(2)$ is not contained
		in the second symmetric power of $V(2)$.
	\item[-]  Consider $2D^+ + D^- \leqslant_\Sigma  (3D^+ + D^-) + (D^+ +2D^-)$.
		Then $V(3) \subset V(4) \otimes V(3)$ but $\pi_2(h(3,1) \otimes h(1,2)) = 0$,
		 so that $V_{2D^+ +D^-} \not \subset V_{3D^+ +D^-} V_{D^+ +2D^-}$.\\
\end{itemize}
\end{example}

More generally, the multiplication is degenerate whenever $\Sigma \cap S \neq \varnothing$.
This can be reduced to the basic case of $\SL(2)/T$
as shown in the following proposition. 

\begin{proposition}	\label{prop: inducing-degenericity}
Suppose that $S \cap \Sigma \neq \varnothing$.
Then the multiplication is degenerate.
\end{proposition}

\begin{proof}
Let $\alpha \in S \cap \Sigma$ and suppose that the multiplication is non-degenerate.
Consider the rank one wonderful subvariety $M'$ defined by intersecting all
the $G$-stable prime divisors $M^\sigma$ with $\sigma \in \Sigma \smallsetminus \{\alpha\}$,
denote $\Delta'$ its set of colors. Then $M'$ is quotient of a parabolic induction of a
wonderful variety for $\SL(2)$ isomorphic to $\mathbb{P}^1 \times \mathbb{P}^1$,
whose generic stabilizer is a maximal torus of $\SL(2)$
and whose set of colors is identified with
$\Delta'(\alpha) = \{D \in \Delta' \, : \, P_\alpha D \neq D\}$,
where $P_\alpha$ denotes the minimal parabolic associated to $\alpha$.
It follows then by Corollary~\ref{cor: moduli ortogonali} that,
for every $D' \in \mathbb{N}[\Delta' \smallsetminus \Delta'(\alpha)]$,
$V_{D'}V_{E'}=V_{D'+E'}$ for all $E' \in \mathbb{N} \Delta'$.

Let $D', E', F' \in \mathbb{N}\Delta'$ be such that $D' \leqslant_{\Sigma'} E' + F'$
and $V(\omega_{D'}) \subset V(\omega_{E'}) \otimes V(\omega_{F'})$.
Denote $\rho \colon \Pic(M) \longrightarrow \Pic(M')$ the restriction.
From the combinatorics of colors and spherical roots it follows that
there exist $E'_0, F'_0 \in \mathbb{N}[\Delta' \smallsetminus \Delta'(\alpha)]$ and $E,F \in \mathbb{N}\Delta$
such that $E'+E'_0 = \rho(E)$ and $F' +F'_0 = \rho(F)$.
Let $D =E+F-(E'+F'-D')$ and notice that $D \in \mathbb{N}\Delta$ and $D'+E'_0+F'_0 = \rho(D)$.
Moreover the inclusion $V(\omega_{D'}) \subset V(\omega_{E'}) \otimes V(\omega_{F'})$
implies that $V(\omega_D) \subset V(\omega_E) \otimes V(\omega_F)$,
hence the non-degeneracy of the multiplication of $M$
implies that $V_D \subset V_E V_F$ and it follows
$V_{D'+E'_0+F'_0} \subset V_{E'+E'_0} V_{F'+F'_0}$.
On the other hand, by Corollary~\ref{cor: moduli ortogonali}, 
$V_{D'+E'_0+F'_0} \subset V_{E'+E'_0} V_{F'+F'_0}=V_{E'}V_{F'}V_{E'_0+F'_0}$ if and only if 
$V_{D'} \subset V_{E'} V_{F'}$.

Therefore, we have deduced that the multiplication of $M'$ is non-degenerate,
but this is a contradiction by the previous example
together with Proposition~\ref{prop: ind-par}.
\end{proof}

\begin{remark}
More generally, if $M$ is a wonderful variety whose multiplication is non-degenerate and if $M' \subset M$ is a localization of $M$, notice that the multiplication of $M'$ is non-degenerate as well. If indeed $\rho : \Pic(M) \longrightarrow \Pic(M')$ is the restriction of line bundles and if $D,E \in \mathbb{N} \Delta$, then $\rho(V_D V_E) = V_{\rho(D)} V_{\rho(E)}$, and if $V_F \subset V_D V_E$, then $\rho(V_F) \neq 0$ if and only if $\supp_\Sigma(D+E-F) \subset \Sigma'$ (where $\Sigma' \subset \Sigma$ are the spherical roots of $M'$). On the other hand, given $E', F' \in \mathbb{N} \Delta'$, by making use of Corollary \ref{cor: moduli ortogonali}, one can reduce the description of the product $V_{E'} V_{F'}$ to that one corresponding to a pair $E'',F'' \in \rho (\mathbb{N} \Delta)$ (where $\Delta'$ is the set of colors of $M'$).

As well, if $M \longrightarrow M'$ is a quotient of $M$ by a distinguished set of colors and if $H \subset H'$ are the corresponding spherical subgroups, by Corollary \ref{cor: sezioni e quozienti} it follows that the multiplication of $M'$ is non-degenerate if that of $M$ is non-degenerate. Indeed, if we identify the set of colors $\Delta'$ of $M'$ with a subset of $\Delta$, then $V_E V_F \subset \mathbb C[G]^{(H')} \subset \mathbb C[G]^{(H)}$ for all $E,F \in \mathbb{N}\Delta'$.

By using Proposition \ref{prop: ind-par}, similar statements are easily proved for parabolic inductions and localizations at a parabolic subgroup (see e.g.\ \cite[1.1.4]{BL} for the definition of the latter). Therefore the non-degeneracy property of the multiplication of a wonderful variety is stable under the standard operations of localization, quotient by a distinguished subset of colors, parabolic induction, localization at a parabolic subgroup.
\end{remark}

Given a spherical subgroup $H \subset G$, define its \textit{spherical
closure} $\overline{H}$ as the kernel of the action of the normalizer
$N_G(H)$ on the set of colors of $G/H$. If $H$ is equal to its spherical closure, then
we say that $H$ is \textit{spherically closed}.
By a theorem of F.~Knop
\cite[Corollary~7.6]{Kn}, if $H$ is spherically closed then $G/H$ admits
a wonderful compactification. If $G$ is not simply
laced, then not spherically closed spherical subgroups $H$ 
such that $G/H$ admits a wonderful compactification exist.
The projection $G/H \longrightarrow G/\overline{H}$ canonically
identifies the colors of $G/\overline{H}$ with those of $G/H$.

Generally speaking, if $M$ is the wonderful compactification of $G/H$ 
where $H$ is not spherically closed,
then the multiplication may be degenerate.
Indeed, it is easy to show that $\mathbb{C}[G]^{(H)} = \mathbb{C}[G]^{(\overline{H})}$,
therefore, if $\overline{\Sigma}$ is the set of spherical roots of $G/\overline{H}$ and if
$D \in \mathbb{N}\Delta$ is such that $V_D \subset V_E V_F$, then we must have
$D\leqslant_{\overline{\Sigma}} E+F$. In this way we may construct examples
of non-spherically closed spherical subgroups $H$ of $G$ 
(possessing no simple spherical roots)
whose associated multiplication is degenerate.

\begin{example} \label{ex: moltiplicazione non-sfer-closed}
Consider the non-adjoint symmetric wonderful variety $M$ for
$\Sp(8)$ with spherical roots $\sigma_1 = \alpha_1 + 2 \alpha_2 +
\alpha_3, \sigma_2 = \alpha_3 + \alpha_4$. Then $M$ possesses two colors $D_2$
and $D_4$, where $\omega_{D_2} = \omega_2$ and $\omega_{D_4} = \omega_4$. Then
we have $D_2 <_\Sigma 2D_2$ and $V(\omega_2) \subset V(\omega_2)^{\otimes
  2}$, on the other hand $\overline{\Sigma} = \{\sigma_1 , 2\sigma_2\}$ and $2D_2 -
D_2 = \sigma_1 + \sigma_2 \not \in \mathbb{N} \overline{\Sigma}$, therefore it cannot be
$V_{D_2} \subset V_{D_2}^2$.
\end{example}

\begin{proposition}	\label{prop: semigruppo-differenze}
Let $\mathscr{M} \subset \mathbb{N}\Sigma$ be the semigroup generated by the set of differences $\{E+F-D \, : \, V_D \subset V_E V_F \}$.
Then $\mathscr{M} = \mathbb{N} \overline{\Sigma}$.
\end{proposition}

\begin{proof}
As already noticed we have $\mathscr M \subset \mathbb{N} \overline{\Sigma}$. Here we show that indeed $\mathscr{M} \supset \overline{\Sigma}$.
Since $\mathbb{C}[G]^{(H)} = \mathbb{C}[G]^{(\overline{H})}$ we may assume that $H$ is spherically closed. 
Let $\sigma \in \Sigma$, proceeding by localization and parabolic induction as in Proposition~\ref{prop: inducing-degenericity} 
we may assume that $M$ is a cuspidal (i.e.\ parabolic induction of no other wonderful variety) rank one wonderful variety 
with spherical root $\sigma$ (see \cite[Table 1]{Was} for a complete list of such varieties). 
With three exceptions (cases 9B, 9C, 15 in \cite{Was}) such wonderful varieties are all flag varieties 
for their automorphism group, so that they are projectively normal. 
Being $H$ spherically closed, notice that apart from the cases 9B, 7C, 9C, 12, 15 in \cite{Was} 
we have that $\sigma = D + E$ for some $D,E \in \Delta$: 
then $(0,D,E)$ is a low triple and the projective normality of $M$ implies that $V_0 \subset V_D V_E$, 
hence $\sigma \in \mathscr M$. 
Moreover, since the case 7C is a quotient of the case 9C, we are reduced to the cases 9B, 7C, 12 and 15, 
where the following inclusions may be checked directly (we index the colors as the corresponding fundamental weights):
\begin{itemize}
	\item[9B)] $G$ is of type $\mathsf{B}_r$, $\sigma = \alpha_1 + \ldots + \alpha_r = D_1$ and we have $V_{D_r} \subset V_{D_1} V_{D_r}$.
	\item[7C)] $G$ is of type $\mathsf{C}_r$, $\sigma = \alpha_1 + 2\alpha_2 + \ldots + 2\alpha_{r-1} +\alpha_r = D_2$ and we have $V_{D_2} \subset V_{D_2}^2$.
	\item[12)] $G$ is of type $\mathsf{F}_4$, $\sigma = \alpha_1 + 2\alpha_2 + 3 \alpha_3 + 2\alpha_4 = D_4$ and we have $V_{D_4} \subset V_{D_4}^2$.
	\item[15)] $G$ is of type $\mathsf{G}_2$, $\sigma = \alpha_1 + \alpha_2 = D_2 - D_1$ and we have $V_{2D_1} \subset V_{D_1} V_{D_2}$.
\end{itemize}
\end{proof}

Suppose now that $M$ is a strict wonderful variety and suppose that
$E,F \in \mathbb{N}\Delta$ are such that $E+F$ is not faithful. The following
example shows that the product $V_E V_F$ may be degenerate,
essentially reducing to a not spherically closed case.

\begin{example} \label{ex: moltiplicazione modello C4}
Let $M$ be the model wonderful variety of type $\mathsf{C}_4$. 
Then $D_2 <_\Sigma 2D_2$
and $V(\omega_2) \subset V(\omega_2)^{\otimes 2}$
(more precisely, $V(\omega_2)$ is also contained in
the second symmetric power of $V(\omega_2)$).
Notice that $2D_2$ is not faithful: the maximal distinguished subset of $\Delta$ which does
not intersect the support of $2D_2$ is $\Delta^* = \{D_1, D_3\}$, and
the quotient $M/\Delta^*$ is the symmetric wonderful variety $M'$
considered in Example~\ref{ex: moltiplicazione non-sfer-closed}, whose
colors we still denote by $D_2$ and $D_4$. 
By Corollary~\ref{cor: sezioni e quozienti} we may identify $\Gamma(M,\mathcal L_{D_2}) =
\Gamma(M',\mathcal L_{D_2})$ and $\Gamma(M,\mathcal L_{2D_2}) =
\Gamma(M',\mathcal L_{2D_2})$, so that we may regard the product $V_{D_2}^2$
inside $\Gamma(M',\mathcal L_{2D_2})$ and by 
Example~\ref{ex: moltiplicazione non-sfer-closed} it follows that
$V_{D_2} \not \subset V_{D_2}^2$.\\
\end{example}

\begin{question} Let $M$ be a strict wonderful variety and
let $E,F \in \mathbb{N} \Delta$ be such that $E+F$ is a faithful divisor. Is
the product $V_E V_F$ non-degenerate?
\end{question}

Suppose that $G$ is simply laced: then the class of the strict wonderful
varieties is stable with respect to the operation of quotient by a
distinguished set of colors, so phenomena as that in Example~\ref{ex:
  moltiplicazione modello C4} cannot happen. Therefore, if the answer
to the previous question were affirmative, proceeding by induction on the
rank of $M$, it would follow that the multiplication is non-degenerate
whenever $M$ is a strict wonderful variety for a simply laced group.

\bibliographystyle{amsplain}

\end{document}